\documentclass[11pt]{article}

\usepackage[margin=1in]{geometry}
\usepackage[utf8]{inputenc}
\usepackage[english]{babel}
\usepackage{amsmath}
\usepackage{amsfonts}
\usepackage{graphicx}
\usepackage{subfigure}
\usepackage{blindtext}
\usepackage{cite}
\usepackage{amsbsy}
\usepackage{indentfirst}
\usepackage{pstricks}
\usepackage{faktor}
\usepackage{float}
\usepackage{datetime}
\usepackage{yfonts}
\usepackage{mathrsfs}
\usepackage{hyperref}
\usepackage{datetime}
\usepackage[ampersand]{easylist}
\usepackage{upgreek}
\usepackage{theoremref}
\usepackage{tikz-cd}
\usepackage{amsthm}
\usepackage[normalem]{ulem}

\theoremstyle{definition}
\newtheorem{theorem}{Theorem}
\newtheorem{corollary}[theorem]{Corollary}
\newtheorem{proposition}[theorem]{Proposition}
\newtheorem{lemma}[theorem]{Lemma}
\newtheorem{remark}[theorem]{Remark}
\newtheorem{definition}[theorem]{Definition}
\newtheorem{example}[theorem]{Example}

\newcommand{\R}{{\mathbb R}}
\newcommand{\Z}{{\mathbb Z}}
\newcommand{\mlt}{{\mu}}
\newcommand{\tn}{\xi}

\newcommand{\ka}[1]{{\kappa_{#1}}}
\newcommand{\ga}[1]{{\gamma_{#1}}}
\newcommand{\beq}{\begin{equation}}
\newcommand{\eeq}{\end{equation}}
\newcommand\ds[1]{\displaystyle{#1}}
\newcommand{\sym}{\text{sym}}
\newcommand{\symi}{\text{sym-index}}
\def\cv{{ \Gamma}}
\providecommand{\bysame}{\leavevmode\hbox to3em{\hrulefill}\thinspace}
\providecommand{\MR}{\relax\ifhmode\unskip\space\fi MR }

\providecommand{\href}[2]{#2}
\definecolor{myCyan}{RGB}{0,255,255}
\definecolor{myGreen}{RGB}{0,128,0}

\begin{document}

\title{Non-congruent non-degenerate curves with identical signatures.}
\author{Eric Geiger and Irina A. Kogan\\ \small{North Carolina State University, Raleigh, NC, USA 27695}}

\date{}
\maketitle

\abstract{
While the equality of differential signatures (Calabi et al, Int. J. Comput. Vis. 26: 107-135, 1998) is known to be a necessary condition for congruence, it is not sufficient (Musso and Nicolodi, J. Math Imaging Vis. 35: 68-85, 2009). Hickman (J. Math Imaging Vis. 43: 206-213, 2012, Theorem 2) claimed that for non-degenerate planar curves, equality of Euclidean signatures implies congruence. We prove that while Hickman’s claim holds for simple, closed curves with simple signatures, it fails for curves with non-simple signatures. In the latter case, we associate a directed graph with the signature and show how various paths along the graph give rise to a family of non-congruent, non-degenerate curves with identical signatures. Using this additional structure, we formulate congruence criteria for non-degenerate, closed, simple curves, and show how the paths reflect the global and local symmetries of the corresponding curve.
}

\vspace{.5cm}


\section{Introduction}
 Determining whether or not two planar curves are congruent under some group action is an important problem in geometry and has applications to computer vision and image processing. To address this problem, the signature curve parameterized by differential invariants was introduced by Calabi, Olver, Shakiban, Tannenbaum, and Haker \cite{Calabi1998} and has been used in various applied problems including medical imaging and automated puzzle assembly \cite{Bou00, Bruck1992, Bruck1995, HO14, grim-shakiban17}.
The origin of this method go back to Cartan's solution  of the group equivalence problem for submanifolds under Lie group actions \cite{C53}.   For a modern exposition see \cite[Chapter 8]{olver:purple}, in particular, the notion of classifying manifolds.

 In this paper, we  focus on immersed closed planar curves under actions of the special Euclidean group $SE(2)$ consisting of rotations and translations.
 If a curve $\Gamma$ is parameterized by a periodic map $\gamma\colon\R\to\R^2$ which is at least $C^3$-smooth, then the corresponding \emph{signature map} is defined by  $\sigma_\gamma(t)=(\kappa(t), \dot\kappa(t)), \,\, t\in\R$, where $\kappa$ is the Euclidean curvature and $\dot\kappa$ is its derivative with respect to the arc-length, explicitly given by formulas  (\ref{eq-kappa}) and (\ref{eq-dot-kappa}), in Section~\ref{sect-prelim}.
 The \emph{signature set} of $\Gamma$ (also called \emph{the signature of $\Gamma$}) is the image of $\sigma_\gamma(t)$. In \cite{Calabi1998}, it was stated that two curves are congruent if and only if their signatures are identical without explicitly mentioning the necessary regularity conditions. In the case of the special Euclidean group, these regularity conditions require that a curve has \emph{no vertices}, i.e. no points where $\dot\kappa$ is zero.
 Since every closed curve has at least two vertices and a simple closed curve has at least four vertices \cite{DeTurck2007, Guggenheimer1963}, the signature congruence criterion does not directly apply to closed curves.
 In \cite{Hoff2013}, the authors propose the following augmented procedure for deciding congruence of two closed curves: split the curves into arcs with no vertices and then decide using the signatures if each arc of the first curve can be paired with a congruent arc of the second curve. If yes, determine if pairwise arc congruence can be achieved by a \emph{common} group element.
In \cite{HO14}, the authors successfully used this multi-step procedure for solving an automated jigsaw puzzle assembly problem.  
A more  refined study of the  structure of the signature sets and signature maps  in relation  with the structure of the symmetry groupoid of a submanifold of an arbitrary dimension under an action of an arbitrary Lie group  was conducted  by  Olver in \cite{olver15}. We  point out the connection with this work  throughout the paper.
 However,  these previous works do not provide a complete answer to the  following important question: what conditions on the signature  set and the signature map  are necessary and sufficient to solve the global equivalence problem for submanifolds?  The  current paper aims to address this question in a simple, but practically important special case of closed simple planar curves under the action of the special Euclidean group.
 
 In \cite{Musso2009} (Theorem 1), Musso and Nicolodi proved that any closed phase portrait in $\R^2$ is the Euclidean signature of a 1-parameter family of non-congruent unit speed closed, at least $C^3$ smooth curves.  This family is constructed   by  inserting segments of constant curvature (degenerate vertices) into a curve. As in \cite{Hickman2012}, we say that a curve is \emph{non-degenerate} if all its vertices are isolated and call it \emph{degenerate} otherwise.  The curves with degenerate vertices are degenerate, but as Example~\ref{ex-deg}  shows, a degenerate curve  may have no degenerate vertices.

 By construction, the non-congruent families of curves with identical signatures appearing in the proof of Theorem 1 in \cite{Musso2009} contain at most one non-degenerate curve, while the rest are degenerate. A subsequent claim by Hickman in Theorem 2 of \cite{Hickman2012} stated that the Euclidean signature can be used to distinguish non-congruent non-degenerate $C^3$ smooth curves.
 In Proposition~\ref{prop-simple} of the current paper, we confirm  that Hickman's claim is true for simple closed curves  \emph{with an additional assumption that the signatures  are simple curves}\footnote{It is very common for  a simple curve to have non-simple signature.}, but as we illustrate  by examples in  Section~\ref{sect-smooth}, without this additional assumption, the claim fails. In Remark~\ref{rem-gap}, we explain why   the argument given in   \cite{Hickman2012}  is not valid for curves with non-simple signatures.
It is worth noting that the curves of a special type, \emph{cogwheels}, introduced in Section 5 of \cite{Musso2009} and    \cite{olver15} (Example 3.9) also provide counter-examples to Hickman's claim \footnote{In \cite{Musso2009} the cogwheels  are constructed  using error functions (likely for computational reasons) and, in fact, the permutation of cogs produces curves with slightly different signatures.
In \cite{olver15} the explicit formulas are not given, but the Mathematica code provided by the author shows that the construction was also done in terms of error functions. In both cases  replacing the error functions with smooth bump functions results in families of non-degenerate curves with identical signatures.}.
However,  these examples were created to show that curves with different symmetry groups may have the same  signatures and the non-degeneracy of the curves was not emphasized. Furthermore, these examples do not give full insight into a general mechanism for deciding existence of non-congruent non-degenerate curves with a given signature and for constructing such curves.

 In Section~\ref{sect-graph}, we provide a general mechanism for constructing families of non-congruent non-degenerate curves with a given closed signature  containing finitely many points of self intersection by introducing a directed graph with multiple edges (a  \emph{signature quiver}) associated with a signature. To each edge we assign a multiplicity and weight.  A path along a signature quiver respects the directions of the edges, and all paths along the same signature quiver that contain every edge of the quiver give rise to  curves with the same signature. A path also reflects the size of the global symmetry group of a curve as well as the  size of the local symmetry sets.  A deeper exploration  of how the structure   of the local symmetry groupoid \cite{olver15} is reflected in the corresponding path along the signature quiver is an interesting subject for a future investigation. 

The paper is structured as follows. In Section~\ref{sect-prelim}, we set up notations and conventions used in this paper, give the main definitions, and review important known results.
In Section~\ref{sect-cong}, we formulate and prove congruence criteria for non-degenerate closed simple curves. In particular we show that non-degenerate closed simple curves with identical \emph{simple} signatures are congruent.
In Section~\ref{sect-smooth} we construct four non-congruent, non-degenerate,
closed, simple, $C^\infty$-smooth curves of equal lengths, identical signatures, and the same signature index,
thus providing an explicit counter-example to the claim in \cite{Hickman2012}. 
In Section~\ref{sect-graph}, we associate a quiver to the signature of a curve. We use this extra structure to formulate a congruence criterion for non-degenerate curves with non-simple signatures (Theorem~\ref{thm-cong-graph}). We show how  various paths on the signature quiver can be used to generate non-congruent, non-degenerate curves with the same signature. In Proposition~\ref{prop-mlt}, we show how  a  path along a quiver encodes information about the global symmetry group  and  local symmetry sets of the corresponding curve.
In Section~\ref{sect-ex} we revisit examples found in \cite{Musso2009} and examine them using the additional structure of this quiver.

Maple code that can be used to generate families of non-degenerate non-congruent curves with identical signatures as well as specific examples in this paper can be found at \url{egeig.com/research/ncndcis}.
\section{Euclidean signatures of planar curves} \label{sect-prelim}
In this section, we set up notations and conventions used in this paper, give the main definitions, and review important known results.

We define  a planar \emph{curve} $\cv$ as the image of a \emph{continuous locally injective}\footnote{A map $\gamma\colon I \to \R^2$, where $I$ is an open  subset of $\R$ is \emph{locally injective} if for any $t\in I$, there exists an open neighborhood $J$, such that $\gamma|_J$ is injective.}  map $\gamma\colon \R\to\R^2$.   We  call $\cv$ \emph{closed} if its parameterization $\gamma$ is periodic. Occasionally, we restrict the domain of $\gamma$ to an open (or a closed) interval  $I\subset \R$ and call the image of such interval a \emph{curve piece.}\footnote{The term \emph{piece} has a different meaning in \cite{olver15}. Also, if $I$ is open, then a curve piece satisfies our definition of a  curve, but we  still use the term curve piece when we want to emphasize that the piece is a proper subset of a larger  curve. } If, in the subset topology, $\gamma(I)$ is homeomorphic to $I$, we  call the image an open (or a closed) \emph{curve segment}. If, in the subset topology, $\gamma(I)$ is homeomorphic to a circle, we  call the image a \emph{loop}.

A point $p$ of a curve (or a curve piece) $\cv$ is called  \emph{simple} if there exists an open subset of $\R^2$ containing $p$ whose intersection with $\cv$ is either homeomorphic to $\R$ or to a half-open interval. Otherwise, $p$ is called a point of \emph{self-intersection}.    We call $\cv$ \emph{simple} if all its points are simple.  A simple curve (or a simple curve piece) is a locally Euclidean connected set  in the subset topology, and therefore, is an  embedded connected   topological $1$-dimensional submanifold of $\R^2$, possibly with boundary, and  so is either an open (closed) curve segment or is a loop. An open or closed curve segment has {an injective} parameterization. A loop has  {a locally injective} parameterization. 

The congruence problem will be considered on a smaller  set of  $C^3$-smooth immersed planar curves or curve pieces, i.e. the set $\mathfrak{G}=\{\Gamma\subset\R^2| \exists I\subset\R, \gamma\colon I \to\Gamma\}$, where $I$ is homeomorphic  to $\R$ or is a closed interval,   $\gamma$ is a surjective map, the derivatives of  $\gamma(t)=(x(t),y(t))$ up to order 3 exist and are continuous,  and $\gamma$ is regular, i.e. $|\gamma'(t)|\neq 0$ for all $t$. Although a curve $\Gamma\in \mathfrak{G}$ admits an uncountable  set of regular and  non-regular parameterizations, we will  only consider its regular  $C^3$-smooth parameterizations.    We will mostly focus on simple closed curves, although some examples of non-closed or non-simple curves will appear. 
For a simple curve, the tangent vector  $\gamma'$ of a  regular parameterization  prescribes a consistent orientation.  On a non-simple  curves an orientation at self-intersection points is undefined, and, moreover, a regular parametrization may traverse the same simple curve piece in opposite    directions, thus inducing a \emph{double orientation}  at some simple points. Figure~\ref{fig:Cinf_multi-oriented} gives an example of a  non-simple curve that does not admit a regular parametrization prescribing a consistent orientation for all simple points. In the context of this paper, we call a curve \emph{oriented} if each of its simple points has a single or double orientation induced by a regular parameterization. 
 
 By $SE(2)$ we denote the special Euclidean group acting on $\R^2$ by compositions of rotations and translations. The $SE(2)$-action on the plane induces an action on the set of oriented curves  and curve pieces $\mathfrak G$, that is,  for  $g\in SE(2)$ and a curve (piece) $\Gamma\subset\R^2$, parameterized by $\gamma$, we define
 a curve (piece) $g\cdot\Gamma=\{g\cdot p\,|\,p\in \Gamma\}$  parameterized  by $g\circ\gamma$. In the context of this paper, we give the following definition of congruence:
\begin{definition} Oriented planar curves (or curve pieces) $\Gamma_1$ and $\Gamma_2$ are called \emph{congruent} ($\Gamma_1\cong\Gamma_2$) if there exists $g\in SE(2)$, such that 
 $\Gamma_2=g\cdot\Gamma_1$, where \emph{equality here means equality of the sets and orientation}. 
\end{definition}
\begin{definition} An element $g\in SE(2)$ is a \emph{symmetry of $\Gamma$} if 
$$g\cdot \Gamma=\Gamma.$$ 
It easy to show  that the set of such elements, denoted $\sym(\Gamma)$   is a subgroup of $SE(2)$ called the  \emph{symmetry group of $\Gamma$}. 
The cardinality of $\sym(\Gamma)$ is called the \emph{(global) symmetry index} of $\Gamma$ and will be denoted by $\symi(\Gamma)$. 
\end{definition}
Note that if $\Gamma$ is a closed curve, then its symmetry group is isomorphic  to a subgroup of the rotation group $SO(2)$. (See the proof of Lemma 4 in \cite{Musso2009} for the explicit formula for the center of rotations).
Following \cite{olver15} (Definition 2.1), we  introduce the notion of a local symmetry of $\Gamma$ based at a point.
\begin{definition}\label{def-loc-sym} An element $g\in SE(2)$ is a \emph{local symmetry based at a point $p\in \Gamma$} if there is an open subset $U\subset \R^2$, containing $p$, such that
\beq \label{eq-loc-sym} g\cdot(\Gamma\cap U)=\Gamma\cap (g\cdot U).\eeq 
The set of such elements is denoted $\sym(\Gamma_p)$. 
We  call the cardinality of $\sym(\Gamma_p)$ the \emph{(local) symmetry index} of $\Gamma$ at $p$ and  denote it by $\symi(\Gamma_p)$\footnote{In  \cite{olver15}, in a much more general setting, the symmetry index at $p$ is defined as the cardinality of the set $\sym(\Gamma_p)/\sym^*(\Gamma_p)$, where  $\sym^*(\Gamma_p)$ denotes the subset of $\sym(\Gamma_p)$ consisting of elements that fix $p$. However, if $p$ is not a point of self-intersection and $\Gamma$ is oriented, then 
under the action of $SE(2)$ the group $\sym^*(\Gamma_p)$ is trivial.}
\end{definition}

In contrast to the set $\sym(\Gamma)$ of  global symmetries, the set of local symmetries  $\sym(\Gamma_p)$, in general, does not form a group, but a disjoint union of all such sets  has a structure of a \emph {groupoid} (see Definition 2.4 in \cite{olver15}).  We also note that, even for closed curves,  $\sym(\Gamma_p)$ may not be  contained in $SO(2)$ (see Remark~\ref{rem-loc-sym} for an illustration).


If $\Gamma$ is a simple curve with a parameterization $\gamma(t)=(x(t),y(t))$, then its  classical (signed) Euclidean curvature at a point $p=\gamma(t)$ is
\beq\label{eq-kappa}\kappa(t)=\frac{\det(\gamma'(t),\gamma''(t))}{|\gamma'(t)|^3}=\frac{x'(t)y''(t)-y'(t)x''(t)}{(x'(t)^2+y'(t)^2)^{\frac 3 2 }}. \eeq
If $\Gamma$ is not simple, then although $\kappa(t)$ is well defined for all $t$, $\kappa(p)$ can take multiple values  at a point of self-intersection, or at a simple point with a double orientation. 
The arc-length one-form is given by 
\beq\label{eq-ds_1}ds=|\gamma'| dt=\sqrt{ x'(t)^2+y'(t)^2}\,dt. \eeq
By $\dot {\phantom{\kappa}}$ we will denote the derivative with respect to the arc-length 
\beq\label{eq-ds_2}\frac d{ds}=\frac 1{\sqrt{ x'^2+y'^2}}\frac d{dt}\eeq and in particular we have
\beq\label{eq-dot-kappa} \dot{\kappa}=\frac{(\gamma'\cdot\gamma')\det(\gamma',\gamma''')-3(\gamma'\cdot\gamma'')\det(\gamma',\gamma'')}{|\gamma'|^6},
 \eeq
 where for better readability we omitted parameter $t$. It can be easily verified that the value of $\kappa(p)$, and $\dot\kappa(p)$ does not depend on the parameterization\footnote{As long as the parametrization is consistent with a chosen  orientation of $\Gamma$. The change in the orientation sends $\kappa(p)$  to $-\kappa(p)$.}  and so each of these functions can be thought as a function from $\Gamma$ to $\R$. If $\Gamma$ is parameterized by the arc-length parameter 
$s(t)=\int_0^t|\gamma'(\tau)| d\tau$
then $\tilde\gamma(s)=\gamma(t(s))$ is a unit speed curve ($|\tilde\gamma'(s)|=1$) and  for $\tilde\gamma(s)$, $\dot \kappa(s)=\kappa'(s)$.
If $\Gamma$ is closed and parameterized by its arc-length then {the minimal period of $\gamma$,} $L$, is the length of $\Gamma$.
The following classical result is fundamental to the study of planar curves in Euclidean Geometry (see, for instance, Theorem 2-10 in \cite{Guggenheimer1963}). 
\begin{proposition} \label{prop-cong} Let $\gamma_1$ and $\gamma_2$ be unit speed parameterizations of curves $\Gamma_1$ and $\Gamma_2$ with corresponding curvature functions $\kappa_1\colon \R\to \R$, and $\kappa_2\colon \R\to \R$.
  If there exists $c\in\R$, such that $\kappa_1(s)=\kappa_2(s+c)$ then $\Gamma_1$ and $\Gamma_2$ are congruent.
  The converse is true if $\Gamma_1$ and $\Gamma_2$ are simple.
\end{proposition}

   \begin{remark}\label{rem-flower}
   Figure~\ref{fig:flower} illustrates why we need the simple condition for the second statement in Proposition~\ref{prop-cong}. Due to the non-transversal intersections, this curve allows several unit speed parameterizations which induce the same orientation, but are not related by translation, e.g. we can first trace the middle and then the petals, or conversely, or interlace traveling along  the middle and the petals. The  the corresponding curvatures as  functions of  the arc-length are not related by a translation, but  the flower with a chosen orientation is congruent to itself according to our definition.  We will encounter a similar example later on (see Figure~\ref{fig:Cinf_notSimple_sym3_1}).
   \end{remark}
\begin{figure}
  \centering
  \includegraphics[width=6cm]{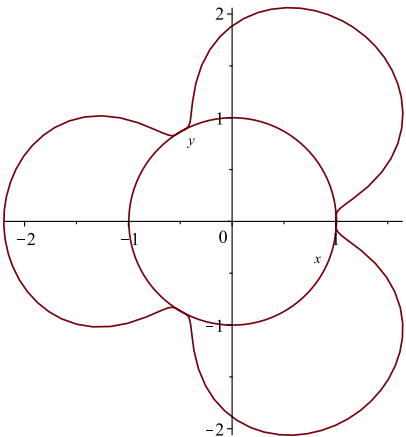}
  \caption{A curve that allows different unit speed parameterizations with the same orientation}
  \label{fig:flower}
\end{figure}

Additionally, we collect the following known facts, \cite{kuhnel2015, Musso2009}, that we will use in the paper:
\begin{proposition}\label{prop-integrals} Let $\Gamma$ be a closed curve with arc-length parameterization $\gamma(s)$, and let $\kappa(s)$ be the corresponding curvature function. Let $L$ and $\ell$ be  the {minimal} periods of $\gamma$ and $\kappa$ respectively. Then
\begin{enumerate}
 \item $\ds{\int_0^L}\kappa(s)ds=2\pi \xi$, where $\xi$ is the turning number of $\Gamma$.
\item If $\Gamma$ is simple, then $\ds{\int_0^\ell}\kappa(s)ds=\frac {2\pi} m$, where $m=\frac L \ell =\symi(\Gamma)$.\end{enumerate} 

\end{proposition}

To construct many of our examples, we exploit the following two known lemmas.
\begin{lemma}\label{lemma-cong}
 Given any  function $\kappa\colon \R\to \R$ of class $C^h$, $h\geq 0$, there exists a unique curve $\Gamma$ with a unit speed parameterization $\gamma\colon \R\to \R^2$ of class  $C^{h+2}$, such that $\gamma(0)=(0,0)$, $\dot{\gamma}(0)=(1,0)$, and $\kappa(s)$ is the curvature of $\Gamma$ at the point $\gamma(s)$. \end{lemma}
 \begin{proof}
To find $\gamma(s)$, explicitly, we note that the vector-function 
\begin{equation}\nonumber\label{F-Ssol}
  T(s) = \left(\cos\eta(s), \sin\eta(s)\right),\text{ where } \eta(s) = \int_0^s \kappa(\tau)d\tau,
\end{equation}
is the unique $C^{h+1}$-smooth solution of the Frenet-Serret equation:
\begin{equation}\label{F-Seq}
  \dot{T}(s) = T(s)\left(\begin{matrix} 0 & \kappa(s) \\ -\kappa(s) & 0\end{matrix}\right)
\end{equation}
with initial conditions $T(0)=(1,0)$. 
The curve 
 \beq\label{eq-gamma} \gamma(s)=\left ( \int_0^s \cos\eta(\tau)d\tau, \, \int_0^s \sin\eta(\tau)d\tau\right)\eeq
 is the unique  $C^{h+2}$-solution to $\gamma'(s)=T(s)$ with initial condition $\gamma(0)=(0,0)$. Since $|T(s)|=1$, then $s$ is the arc-length parameter and by construction $\kappa(s)$ is the curvature of $\Gamma$ at the point $\gamma(s)$. 
  \end{proof}

  The above lemma is well known and is discussed at the beginning of Section 2 of \cite{Musso2009}. We included its full proof because we use (\ref{eq-gamma}) and its numerical approximation for constructing explicit examples. The proof of the following lemma appears in \cite{Musso2009} and is useful for identifying and constructing closed curves.

\begin{lemma}\label{lem-closed}
  (Musso and Nicolodi \cite{Musso2009}, Lemma 4)\footnote{In \cite{Musso2009}, $\kappa$ is assumed to be $C^1$, but the proof is valid for  a continuous function $\kappa$.}
  
  Let  $\kappa : \mathbb{R} \to \mathbb{R}$ be a  periodic continuous function, with minimum period $\ell$.
  If
  \beq \label{eq-closed}\frac{1}{2\pi}\int_0^\ell \kappa(s)ds = \frac{\tn}{m},\eeq
  where $m$ and $\tn$ are relatively prime integers and $m>1$,
  then  the map   $\gamma\colon\R\to\R$, given by (\ref{eq-gamma}) is a unit-speed parametrization of a closed  curve $\Gamma$ of length $m\ell$.

\end{lemma}
It follows then from Proposition~\ref{prop-integrals} that $\tn$ is the turning number of $\gamma$ over  the interval $[0,m\ell]$, and if $\Gamma = Im(\gamma)$ is simple, then $\tn = 1$ and $m$ is the symmetry index of $\Gamma$.
Note that if $m = 1$ and $\gamma$ is closed, then $L = \ell$.
However, it is possible for $\gamma$ to be open when $m=1$ (see examples in Section \ref{sect-smooth}).

As in \cite{Calabi1998}, we define the \emph{Euclidean signature} (SE(2)-signature)
\footnote{Strictly speaking we should call this signature the ``special Euclidean signature'', because the full Euclidean group includes reflections and both $\kappa$ and $\dot\kappa$ change sign under reflections and are, therefore, not invariant.
We also note, that if we reverse a curve's orientation then $\kappa$ changes its sign, but $\dot\kappa$ preserves the sign. Thus change in the orientation induces the  reflection of the signature about the vertical axis. } of $\Gamma$ to be the planar curve parameterized by $(\kappa(t),\dot{\kappa}(t))$ (we remind that $\dot {\phantom{\kappa}}$ denotes the derivative with respect to the arc-length defined by the operator \eqref{eq-ds_2}).  In fact, we will use the following related notions: 

\begin{definition}\label{def-sig}The \emph{signature map} of a simple curve $\Gamma$ is the map $\sigma_{\Gamma}\colon \Gamma\to\R^2$, defined by $\sigma_\Gamma(p)=(\kappa(p),\dot\kappa(p))$. The \emph{signature} (the \emph{signature set}) of $\Gamma$ is the image of this map: $S_{\Gamma}=\mathrm{Im}{\sigma_\Gamma}$. For a given parameterization $\gamma\colon\R\to\Gamma$,  a \emph{parameterized signature map} $\sigma_\gamma \colon \R\to S_{\Gamma}$ is defined by  $\sigma_\gamma(t)=(\kappa(t),\dot{\kappa}(t))$.
\end{definition}
If $\Gamma$ is not simple, then  $\sigma_\gamma$ is a well defined continuous  map, but $\sigma_\Gamma$ can be multivalued at the self-intersection points and at simple points with double orientation. The latter is because  $\kappa$ changes its sign when an orientation is reversed and so each point with double orientation will give rise to two points on the signature. If $\Gamma$ is simple, then $\sigma_\gamma=\sigma_\Gamma\circ \gamma$. Definition~\ref{def-sig} can also be extended to curve segments. If $\Gamma$ is a circle (a circular arc) or a straight line (a straight segment), then the function $\kappa$ is constant and so the signature degenerates to a point. Otherwise the signature of $\Gamma$ is a non-constant  curve in   $\R^2$. 
\begin{remark}\label{rem-sig}
Since for an arc-length parameterization of a curve, $\dot \kappa(s)=\kappa'(s)$, any signature curve is a \emph{phase portrait} -- a curve  that can be parameterized by a function and its derivative: $s\to (f(s),f'(s))$. A phase portrait can only be traveled to the right when it is above the horizontal axis and to the left if it is below.  Therefore, the \emph{orientation} of  signature curve is uniquely defined, and in particular   the signature of  a closed curve is a parameterized closed curve with a clockwise orientation. \end{remark} 

It is known that $\kappa$ and $\dot \kappa$ are invariant with respect to actions of $SE(2)$. Namely if $\tilde \Gamma=g\cdot\Gamma$ for some $g\in SE(2)$ then for any $p\in \Gamma$, we have
$$\kappa_\Gamma(p)=\kappa_{\tilde\Gamma}(g\cdot p) \text{ and } \dot\kappa_\Gamma(p)=\dot\kappa_{\tilde\Gamma}(g\cdot p). $$
It immediately follows that the signatures of congruent curves are identical. As it is shown by examples in \cite{Hickman2012}, \cite{Musso2009}, and \cite{olver15}, the converse is not always true: non-congruent closed simple curves may have identical signatures.
So it is natural to ask under what conditions, or with what additional information, signatures can be used to distinguish non-congruent curves. The notion of a \emph{curve-vertex} is important in addressing this question.

\begin{definition} \label{def-vertex}A point $p\in \Gamma$ is called a \emph{vertex} if $\dot\kappa(p)=0$. A curve with isolated vertices is called \emph{non-degenerate} and otherwise is called \emph{degenerate}.
\end{definition}

Curve pieces of constant curvature, circular arcs and straight segments, are referred to as  \emph{degenerate vertices}, and the curves containing such pieces are degenerate.  The following example shows a degenerate curve with no degenerate vertices. 

 \begin{example} \label{ex-deg}
   Consider the $C^3$-smooth curve 
   \beq\label{eq-degen}\gamma(t) = \begin{cases} (t, t^6\sin(\frac1t)) & t \in (\frac{-1}{4\pi},0) \cup (0,\frac{1}{4\pi}) \\ (0,0) & t = 0 \end{cases}\eeq
   whose curvature function is shown in Figure~\ref{fig:deg}.
   While this curve does not contain any degenerate vertices, at $t=0$ there is a non-isolated vertex.
\begin{figure}
  \centering
  \includegraphics[width=6cm]{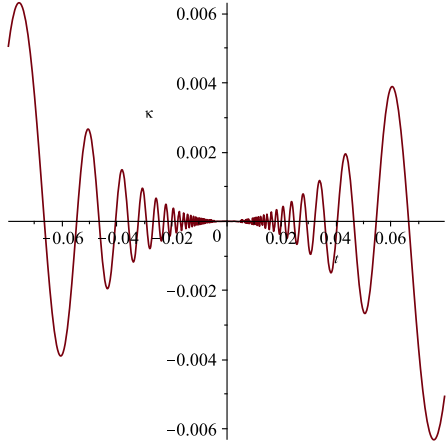}
  \caption{Curvature $\kappa(t)$ of a degenerate curve \eqref{eq-degen}.}
  \label{fig:deg}
\end{figure}

 \end{example}

As discussed in \cite{Hoff2013}, signatures of curves with no vertices distinguish non-congruent curves. In fact, the proof stated in \cite{Calabi1998} is valid in this case.
However, every closed curve has at least two vertices and \emph{simple closed curves have at least four vertices} (see \cite{DeTurck2007} for a very informative exposition of the four vertex theorem and its converse, as well as Theorem 2-18 in \cite{Guggenheimer1963}).

\begin{remark}
If a curve has a degenerate vertex (a piece of constant curvature), then  the corresponding intervals on the parameter space map to a point on the horizontal axis.  The signature map of a non-degenerate curve with a regular parametrization is \emph{locally injective}.
\end{remark}
Theorem 1 proven in \cite{Musso2009} states that, in fact, any closed phase portrait is the signature of a 1-parameter family of non-congruent $C^3$-smooth closed curves. These families are constructed by inserting degenerate vertices into a given curve  and so, by the above remark, these new curves have the same signature as the original curves, but by  construction each family contains \emph{at most one} non-degenerate curve. 
In the current paper, we provide a general mechanism for constructing families  of \emph{non-degenerate curves} with the same signature. 
 \begin{remark}\label{rem-closed-phase}  The signature of a closed curve is always closed, but the converse is not true.  In Figure~\ref{fig:sinSig}, the signature of the sine  curve $\gamma(t)=(t, \sin t)$ is pictured.  Although, according to Theorem 1 in \cite{Musso2009}, every closed phase portrait $(\kappa(s),\dot \kappa(s))$ is the signature of a family of closed curves,  this family is not guaranteed  to contain  either a simple or non-degenerate curve. In this example, since the integral of the curvature function over its minimal period is 0, there are   no closed, simple, non-degenerate curves with the signature pictured in  Figure~\ref{fig:sinSig}. Identifying closed phase portraits that are signatures of closed, non-degenerate, simple curves is an interesting problem.\footnote{For a sufficient condition see Corollary~5 in \cite{Musso2009}.}
 \end{remark}

 \begin{figure}
   \centering
   \includegraphics[width=6cm]{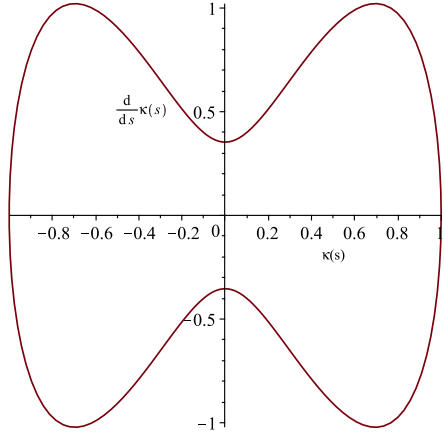}
   \caption{A closed phase portrait that is not the signature of any non-degenerate closed curve.}
   \label{fig:sinSig}
 \end{figure}
 As seen in examples in the following sections, the signature of a simple curve does not have to be simple, and different points on the signature may have an unequal number of preimages under the signature map.  This more refined information is reflected by local and global indices of the signature map.  Following \cite{Musso2009} and \cite{olver15} we define 

\begin{definition}\label{def-sig-index}
  The  \emph{(local) signature index} of a curve $\Gamma$ at $p\in\Gamma$ is
\beq \label{eq-sig-in-p} \text{sig-index}(\Gamma_p) = \#\sigma_\Gamma^{-1}(\zeta), \text{ where }\zeta=\sigma_\Gamma(p),\eeq
 while 
 the \emph{(global) signature index} of $\Gamma$ is
 \beq\label{eq-sig-ind}\text{sig-index}(\Gamma)= \ds{\min_{p\in\Gamma}}\{\text{sig-index}(\Gamma_p)\},\eeq
 where $\#$ denotes the cardinality of the set.
\end{definition}


It is clear that (\ref{eq-sig-ind}) is consistent with the definition 
$\text{sig-index}(\Gamma)= \min_{\zeta \in S_\Gamma}\#\{\sigma_\Gamma^{-1}(\zeta)\}$ in \cite{Musso2009} (Definition 2).
If $\Gamma$ is a simple curve, and $\gamma\colon\R\to \Gamma$ is {a regular} parameterization with a minimal period $L$ and $\sigma_\gamma=\sigma_\Gamma\circ\gamma$ as in Definition~\ref{def-sig}, then 

\beq\label{eq-sig-in2}\text{sig-index}(\Gamma) = \min_{\zeta \in S_\Gamma}\#\{ t\in[0,L)|\sigma_\gamma(t)=\zeta\}.\eeq

The two curves shown in Figure~\ref{fig:Cinf-sym} have unequal symmetry and equal signature indices. However, the following  lemma shows that for a large class of curves  the symmetry index divides the signature index:

\begin{lemma}
  Let $\Gamma$ be a closed, non-degenerate, simple planar curve. Then for all $p\in \Gamma$, the local signature  index   $\text{sig-index}(\Gamma_p)$  is finite and is divisible by the symmetry index $\text{sym-index}(\Gamma)$. The same is true for the global signature index  $\text{sig-index}(\Gamma)$.
\end{lemma}
\begin{proof}Since $\Gamma$ is simple, closed  and is not a circle, its $SE(2)$-symmetry group,  $\sym(\Gamma)$, is a finite subgroup of  the group of rotations  about a point $C\in \R^2$, such that $C\notin\Gamma$.
  Let $m =\# \sym(\Gamma)= \text{sym-index}(\Gamma)$. For $p\in\Gamma$, let $\zeta=\sigma_\Gamma(p)$. 
  From the assumptions on $\Gamma$, it follows  that  the  signature map  $\sigma_\Gamma\colon  \Gamma \to S_\Gamma$ is a well defined continuous locally injective map with compact domain, and so, for any $\zeta\in S_\Gamma$, the preimage set  $\Gamma_\zeta=\sigma^{-1}_\Gamma (\zeta)$  is finite. Since the signature map $\sigma_\Gamma$ is $SE(2)$-invariant,   the action of $\sym(\Gamma)$ restricts to  $\Gamma_\zeta$, and, since $C\notin\Gamma$, it is free. Therefore, $m$ divides    $\#\Gamma_\zeta=\text{sig-index}(\Gamma_p)$.  
  Since the global signature index $\text{sig-index}(\Gamma)$ of $\Gamma$ is defined as the minimum of $\text{sig-index}(\Gamma_p)$, as $p$ varies along $\Gamma$, we immediately have that $\text{sig-index}(\Gamma)$ is  also divisible by $m$.
\end{proof}
 Upon developing more machinery, we show in Proposition~\ref{prop-mlt}, that  at $p\in \Gamma$, such that $\sigma_\Gamma(p)$ is a simple point of the signature, $\text{sym-index}(\Gamma_p)=\text{sig-index}(\Gamma_p)$ (see also Proposition 3.8 of \cite{olver15}).
 
   As the cogwheels shown in Figures~\ref{fig:cogEx}, \ref{fig:cogwheel2}, and \ref{fig:cogwheel3} illustrate,  non-congruent non-degenerate, closed simple curves  of the same length may have identical signatures,  global symmetry and signature indices. Moreover, there is a bijection between any two of these curves, such that the corresponding points have the same  same  local symmetry and signature indices. Thus additional information will be needed to distinguish non-congruent curves. 
\section{Congruence criteria for non-degenerate curves}\label{sect-cong}
In this section, we formulate several congruence criteria for non-degenerate curves. Our proofs are inspired by arguments given in \cite{Hickman2012}. We start with the following local congruence criteria.
\begin{proposition}[Local Congruence]\label{prop-loc}
 Let $\Gamma_1$ and $\Gamma_2$ be two non-degenerate curves with parameterizations $\gamma_1\colon\R\to\R^2$ and $\gamma_2\colon\R\to\R^2$ respectively.
Assume there exist open intervals $I_1\subset\R$ and $I_2\subset\R$, such that the restrictions of the parameterized signature maps $\sigma_1=\sigma_{\gamma_1}|_{I_1}$ and $\sigma_2=\sigma_{\gamma_2}|_{I_2}$ are injective and have the same image $\hat S$ homeomorphic to $\R$. Then the pieces of curves $\hat\Gamma_1=\gamma_1(I_1)$ and $\hat\Gamma_2=\gamma_2(I_2)$ are congruent.
\end{proposition}
\begin{proof} Without loss of generality we can assume that $\gamma_1$ and $\gamma_2$ are arc-length parameterizations. Then, for $i=1,2$, $\sigma_i(s)=(\kappa_i(s),\kappa_i'(s))$, $s\in I_i$.  Let $\Lambda_i=\{s\in I_i |\gamma_i(s)\text{ is a vertex of } \Gamma_i\}$. 
We show that $\Lambda_i$ is a discrete set of points. Indeed assume the contrary. Then, there exists a convergent sequence of distinct points $\{s_n\}_{n=0}^\infty\subset \Lambda_i$, such that  $s=\lim_{n\to\infty} s_n\in \Lambda_i$. By continuity  and local injectivity of $\gamma_i$, we obtain  that the vertex  $p=\gamma_i(s)$ is not isolated. This contradicts the non-degeneracy assumption.

Since the image of $\hat S$ of $\sigma_i$   is homeomorphic to $\R$, and   a bijective continuous map from an open interval to $\R$ is a homeomorphism, it follows that $\sigma_1$ and $\sigma_2$ are homeomorphisms and so  the composition $\rho=\sigma_1^{-1}\circ\sigma_2\colon I_2\to I_1$ is also  a homeomorphism. Then, restricted to $I_2$, we have 
\beq \label{eq-rel}\kappa_2(s)=\kappa_1(\rho(s))\text{ and } \kappa_2'(s)=\kappa'_1(\rho(s)).\eeq 
 We  first show that $\rho: I_2 \to I_1$ is differentiable for all $s$ and $\rho'(s)=1$ for all $s\in I_2$.
 From (\ref{eq-rel}) it follows that $\Lambda_1=\rho(\Lambda_2)$. 
 
 Since $I_2$ is open,  for any $s_0\in I_2$, a function
\beq\label{eq-alpha}\alpha(h)=\rho(s_0+h)-\rho(s_0)\eeq
is defined on a sufficiently small interval containing $0$.
Since $\rho$ is continuous, $\lim_{h\to 0}\alpha(h)=0$, but since $\rho$ is injective, $\alpha(h)\neq 0$ for $h\neq 0$.
By our assumption of $\kappa_2$ being $C^1$-smooth, $\kappa_2'(s_0)$ is defined and using the first equality in (\ref{eq-rel}) and (\ref{eq-alpha}) we can write:
\begin{align*}
\kappa_2'(s_0)&=\lim_{h\to 0}\frac{\kappa_1(\rho(s_0+h))-\kappa_1(\rho(s_0))} h\\
&=\lim_{h\to 0}\frac{\kappa_1(\rho(s_0)+\alpha(h))-\kappa_1(\rho(s_0))} h\\
&=\lim_{h\to 0}\frac{\kappa_1(\rho(s_0)+\alpha(h))-\kappa_1(\rho(s_0))} {\alpha(h)}\cdot \frac {\alpha(h)} h
\end{align*}
Since $\kappa_1$ is $C^1$-smooth, we have 
$$\lim_{h\to 0}\frac{\kappa_1(\rho(s_0)+\alpha(h))-\kappa_1(\rho(s_0))} {\alpha(h)}=\kappa_1'(\rho(s_0)).$$

Now assume that $s_0\in I_2\backslash \Lambda_2$ and so  $\rho(s_0)\in I_1\backslash \Lambda_1$. Then   $\kappa_1'(\rho(s_0))\neq 0$. It then follows that $\lim_{h\to 0 }\frac {\alpha(h)} h$ exists and from the second equation in  (\ref{eq-rel}) it has to be equal to $1$. On the other hand, 
from \eqref{eq-alpha} $\lim_{h\to 0 }\frac {\alpha(h)} h=\rho'(s_0)$.
Thus, $\rho$ is differentiable on the set $I_2\backslash \Lambda_2$  and $\rho'(s_0)=1$ for all  $s_0\in I_2\backslash \Lambda_2$  . Since $\Lambda_2$ is discrete  this implies that there exists $c\in \R$, such that $\rho(s)=s+c$. Then the first equality in (\ref{eq-rel}) implies that $\kappa_2(s)=\kappa_1(s+c)$ and the conclusion of the theorem follows from Proposition~\ref{prop-cong}.

\end{proof}
We now show that the congruence criterion in \cite{Hickman2012}[Theorem 2] stating that: \emph{``Two non-degenerate $C^3$ curves $\gamma_1$ and $\gamma_2$ can be mapped to each other by a proper\footnote{In  \cite{Hickman2012}, \emph{special} Euclidean transformations are called \emph{proper} Euclidean transformations.} Euclidean transformation, $\gamma_2 = g\cdot\gamma_1, g\in SE(2)$, if and only if their signature curves are identical $S_1 = S_2$.''}  is valid for simple curves with simple signatures. Later in Section~\ref{sect-smooth}  we show that it \emph{is not valid} for curves with non-simple signatures.
\begin{proposition}[Curves with simple signature]\label{prop-simple} Assume $\Gamma_1$ and $\Gamma_2$ are simple closed non-degenerate curves with the same signature $S$. Assume $S$ is a simple closed curve. Then $\Gamma_1$ and $\Gamma_2$ are congruent.
\end{proposition}
\begin{proof} Let $\gamma_1$ and $\gamma_2$ be unit speed parameterizations of $\Gamma_1$ and $\Gamma_2$, respectively and let $p_1=\gamma_1(0)\in \Gamma_1$ and $\zeta=\sigma_{\Gamma_1}(p_1)\in S$. Then there exists $p_2\in \Gamma_2$, such that
$\sigma_{\Gamma_2}(p_2)=\zeta$. By applying a shift in the parameterization of $\Gamma_2$ we may assume that $\gamma_2(0)=p_2$. 
Let $\kappa_1(s)$ and $\kappa_2(s)$ be corresponding parameterized curvature functions and $\ell_1$, $\ell_2$ their respective minimal periods. 

As we discussed in Remark~\ref{rem-sig}, for $i=1,2$, the maps $\sigma_{\gamma_i}\colon\R\to S$ are locally injective clockwise parameterizations of the simple curve $S$  with minimal periods $\ell_i$.  They have the same images, the set $S\backslash \{\zeta\}$ homeomorphic to $\R$. Since  by the intermediate value theorem a locally injective and surjective map from an open interval to $\R$ is a homeomorphism, it follows that $\sigma_{\gamma_1}|_{(0,\ell_1)}$ and $\sigma_{\gamma_2}|_{(0,\ell_2)}$ are homeomorphisms onto their  image $S\backslash \{\zeta\}$.   Then, from  Proposition~\ref{prop-loc}, the open curve segments $\hat\Gamma_1=\gamma_1\big((0,\ell_1)\big)\subset \Gamma_1$ and $\hat\Gamma_2=\gamma_2\big((0,\ell_2)\big)\subset \Gamma_2$ are congruent. From Proposition~\ref{prop-cong}, it follows that there exists $c$, such that
$\kappa_2|_{(0,\ell_2)}(s)=\kappa_1|_{(0,\ell_1)}(s+c)$. But since both intervals are minimal periods of the corresponding function and both start with zero, we have $c=0$ and $\ell_1=\ell_2$. Then by continuity and periodicity of $\kappa_1$ and $\kappa_2$, $\kappa_2(s)=\kappa_1(s)$ for all $s\in \R$. Invoking Proposition~\ref{prop-cong} again we conclude that $\Gamma_1\cong\Gamma_2$.

\end{proof}
Since most of the curves appearing in \cite{Musso2009} and \cite{Hickman2012} and in the current paper have non-simple signatures, we show in Figure~\ref{fig:SimpleSigCurve} an example of a curve with a very simple signature.
\begin{figure}
  \centering
  \subfigure[A simple, closed, non-degenerate curve $\Gamma$ with curvature $\kappa(s) = \sin(s)+\cos(s)+\frac15$.]{
  \includegraphics[width=6cm]{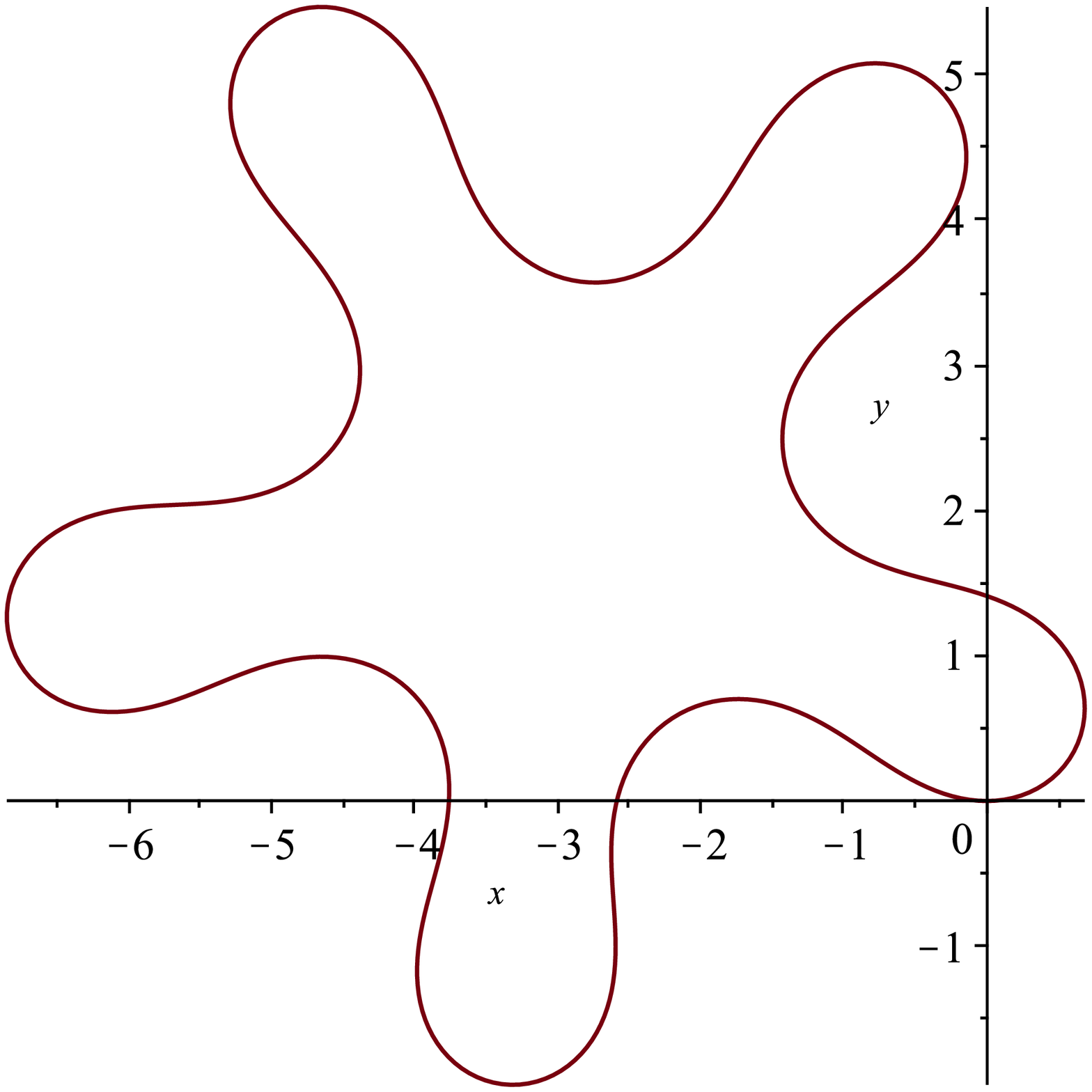}
}
  \subfigure[The signature $S_\Gamma$ is a circle.]{
  \includegraphics[width=6cm]{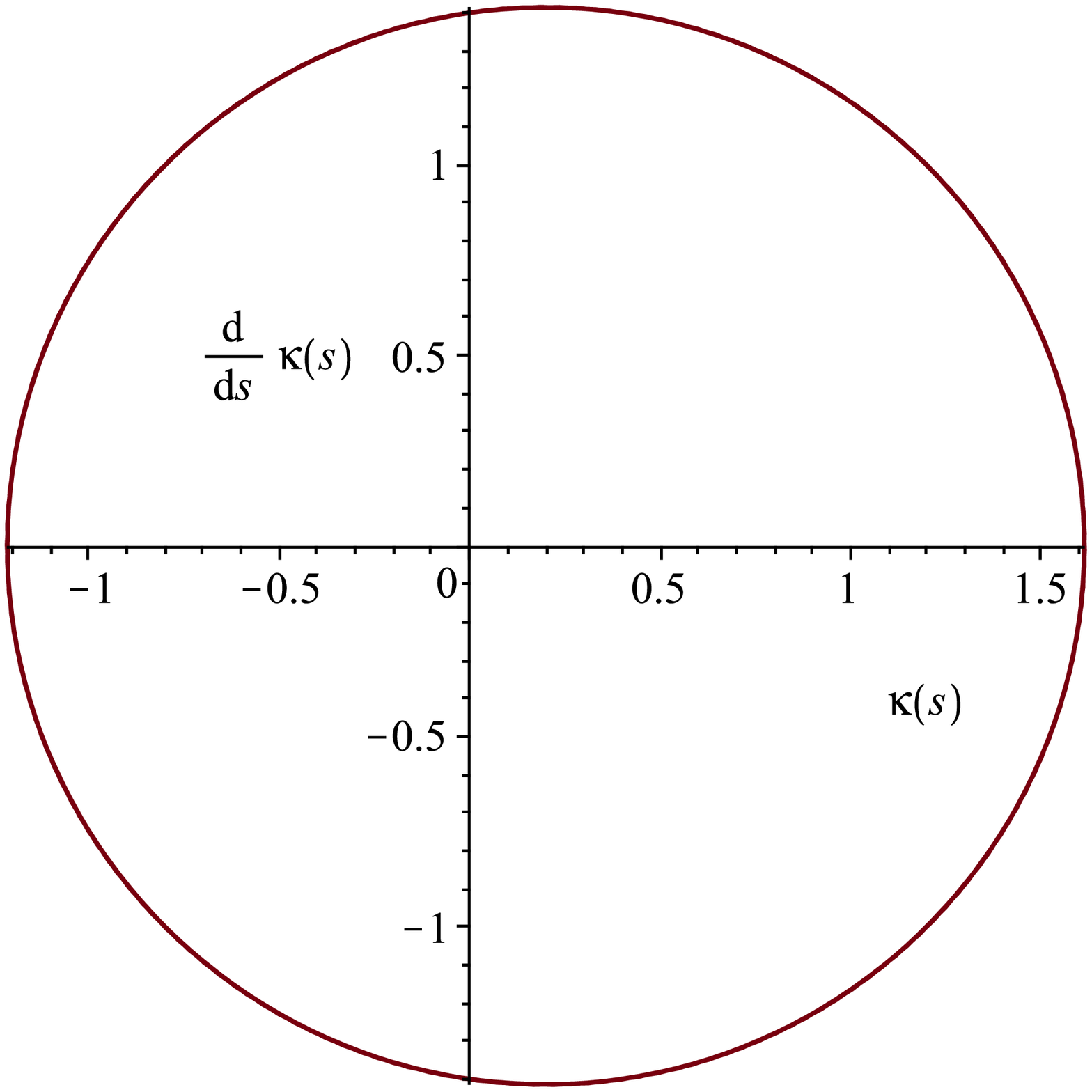}
}
  \caption{A simple, closed, non-degenerate curve with simple signature.}
  \label{fig:SimpleSigCurve}
\end{figure}
We conclude this section with  general congruence criterion. As our subsequent examples indicate, to establish congruence of curves that have equivalent non-simple signatures, one needs to pay close attention at how the parameterized signature map traces the signature. 
\begin{theorem}[Non-degenerate curve congruence]
  \label{thm-cong-non-deg} Let $\gamma_1\colon\R\to\R^2$ and $\gamma_2\colon\R\to\R^2$ be parameterizations of non-degenerate curves $\Gamma_1$ and $\Gamma_2$, respectively, and $\sigma_{\gamma_1}$, and $\sigma_{\gamma_2}$ the corresponding parameterized signature maps. 
 If there exists a homeomorphism $\rho\colon \R\to\R$, such that $\sigma_{\gamma_2}=\sigma_{\gamma_1}\circ\rho$ then $\Gamma_1$ and $\Gamma_2$ are congruent. If $\Gamma_1$ and $\Gamma_2$ are simple then the converse is true as well.
\end{theorem}
\begin{proof} 
  We note that the first statement of the theorem is true for arc-length parameterizations of $\Gamma_1$ and $\Gamma_2$ if and only if it is true for arbitrary regular parameterizations of $\Gamma_1$ and $\Gamma_2$. Thus we may assume that $\gamma_1$ and $\gamma_2$ are arc-length parameterizations.

  First, assume that there exists a homeomorphism $\rho\colon \R\to\R$ such that $\sigma_{\gamma_2} = \sigma_{\gamma_1}\circ \rho$. Then (\ref{eq-rel}) is valid for all $s\in \R$. Following the same argument as in the proof of Proposition~\ref{prop-loc} we can show that there exists $c\in \R$, such that $\kappa_2(s)=\kappa_1(s+c)$. Now the congruence of $\Gamma_1$ and $\Gamma_2$ follows from Proposition~\ref{prop-cong}.

 Next, assume $\Gamma_1 \cong \Gamma_2$ and are  simple. Then there exists $g\in SE(2)$, such that $\gamma_2$ and $g \circ \gamma_1$ are two arc-length parameterizations of $\Gamma_2$. 
From Proposition~\ref{prop-cong}, it follows that there exists $c\in\R$ such that $\gamma_2(s) = g \circ \gamma_1(s+c)$.
 Since curvature is invariant under  the $SE(2)$-action,  it follows that $\kappa_2(s) = \kappa_1(s+c)$.
So for the homeomorphism $\rho(s) = s+c$, we have $\sigma_{\gamma_2} = \sigma_{\gamma_1}\circ \rho$, as desired.
\end{proof}

\section{Non-congruent, non-degenerate, simple, \texorpdfstring{$C^\infty$}{C infinity smooth}-curves of the same length with identical signature}
\label{sect-smooth} 
In this section, we start by constructing two non-congruent, non-degenerate,
closed, simple, $C^\infty$ curves $\Gamma_1$ and $\Gamma_2$ of equal lengths, identical signatures, the same signature index (equal to 6) and the same symmetry group ($\Z_6$).
These curves provide a counterexample to Theorem 2 in \cite{Hickman2012}
and help to identify a gap in the argument presented in \cite{Hickman2012} (See Remark~\ref{rem-gap}). 
We conclude this section by constructing two more smooth closed simple curves $\Gamma_3$ and $\Gamma_4$ with the same signature index and length as $\Gamma_1$ and $\Gamma_2$, but different symmetry groups: $\Gamma_3$ has symmetry group $\Z_3$ and $\Gamma_4$ has the symmetry group $\Z_2$.

We start by building $C^\infty$-smooth curvature functions $\ka 1(s)$ and $\ka 2(s)$, both with a minimal period of 8, such that for every 
integer $n$, $\ka 2(s)$ restricted to the interval $[2n, 2(n+1)]$ is a horizontal shift of $\ka 1(s)$ restricted to the interval $[2m, 2(m+1)]$ for some integer $m$ and vice versa. However, globally $\ka1(s)$ and $\ka 2(s)$ are not related by a horizontal shift. Functions $\ka 1(s)$ and $\ka 2(s)$, restricted to a minimal period are given in Figure \ref{fig:Cinf-kappa} and their explicit formulas are obtained as follows.

 Let $r_1<r_2$ be two real numbers
and consider the smooth cutoff function:
\[
  h_{r_1,r_2}(s) = \begin{cases}
  1 & s \leq r_1 \\
  \frac{e^{\frac{1}{s-r_1}}}{e^{\frac{1}{s-r_1}} + e^{\frac{1}{r_2-s}}} & r_1 < s < r_2 \\
  0 & s \geq r_2
\end{cases}
\]
and its reflection across the line $x = \frac{r_1+r_2}{2}$
\[
  g_{r_1,r_2}(s) = \begin{cases}
  0 & s \leq r_1 \\
  \frac{e^{\frac{1}{r_2-s}}}{e^{\frac{1}{s-r_1}} + e^{\frac{1}{r_2-s}}} & r_1 < s < r_2 \\
  1 & s \geq r_2.
\end{cases}
\]

For both $h_{r_1,r_2}$ and $g_{r_1,r_2}$,
 the limit of their derivatives of any order as they approach $r_1$ from the left and $r_2$ from the right is 0 (see \cite{Lee1}, Page 42). 
 
Note that  due to the symmetry $ \int_{r_1}^{r_2}h_{r_1,r_2}(s)ds=\int_{r_1}^{r_2}g_{r_1,r_2}(s)ds$,
we have
\begin{align*}\int_{r_1}^{r_2}h_{r_1,r_2}(s)ds
&=\int_{r_1}^{r_2}\frac{e^{\frac{1}{s-r_1}}}{e^{\frac{1}{s-r_1}} + e^{\frac{1}{r_2-s}}}ds\\
&=\int_{r_1}^{r_2}\frac{e^{\frac{1}{r_2-s}}}{e^{\frac{1}{s-r_1}} + e^{\frac{1}{r_2-s}}}ds\\
&=\int_{r_1}^{r_2}1 - \frac{e^{\frac{1}{s-r_1}}}{e^{\frac{1}{s-r_1}} + e^{\frac{1}{r_2-s}}}ds\\
&= {r_2-r_1}-\int_{r_1}^{r_2}h_{r_1,r_2}ds.
\end{align*}
And so
\beq\label{int-gh}\int_{r_1}^{r_2}h_{r_1,r_2}ds=\int_{r_1}^{r_2}g_{r_1,r_2}ds=\frac{r_2-r_1}{2}.\eeq

Define a function $f_{r_1,r_2}: \mathbb{R} \to \mathbb{R}$, using $h_{r_1,\frac{r_1+r_2}{2}}$ and $g_{\frac{r_1+r_2}{2},r_2}$ together as follows: 
\[
  f_{r_1,r_2}(s) = \begin{cases}
    0 & s \leq r_1\\
    g_{r_1,\frac{r_1+r_2}{2}}(s) & r_1 < s \leq \frac{r_1+r_2}{2} \\
    h_{\frac{r_1+r_2}{2},r_2}(s) & \frac{r_1+r_2}{2} < s < r_2 \\
    0 & s \geq r_2
  \end{cases}
\]
Then $f_{r_1,r_2}$ is a $C^\infty$ function such that $f_{r_1,r_2}(s) \equiv 0$ for $s \notin (r_1,r_2)$, 
$0 < f(s) < 1$ for $s \in (r_1,\frac{r_1+r_2}{2}) \cup (\frac{r_1+r_2}{2},r_2)$, and $f(\frac{r_1+r_2}{2}) \equiv 1$.
Due to (\ref{int-gh}), 
\beq\label{int-f}\int_{r_1}^{r_2}f_{r_1,r_2}(s)ds=\frac{r_2-r_1}{2}.\eeq

On the interval $[0,8]$ functions $\ka 1, \ka 2$, shown in Figure \ref{fig:Cinf-kappa}, are defined  by:
\begin{align}
\label{ka1} \kappa_1(s)& = \frac{\pi}{3}\big( f_{0,2} - 2f_{2,4} + 3f_{4,6} - f_{6,8} \big)(s)\\
\label{ka2}\kappa_2(s) &= \frac{\pi}{3}\big( f_{0,2} + 3f_{2,4} - 2f_{4,6} - f_{6,8} \big)(s)
\end{align}
and then periodically extended to two $C^\infty$ functions on $\R$ with a period 8.
\begin{figure}
  \centering
    \subfigure[Curvature $\ka 1(s)$.]{
    \includegraphics[width=6cm]{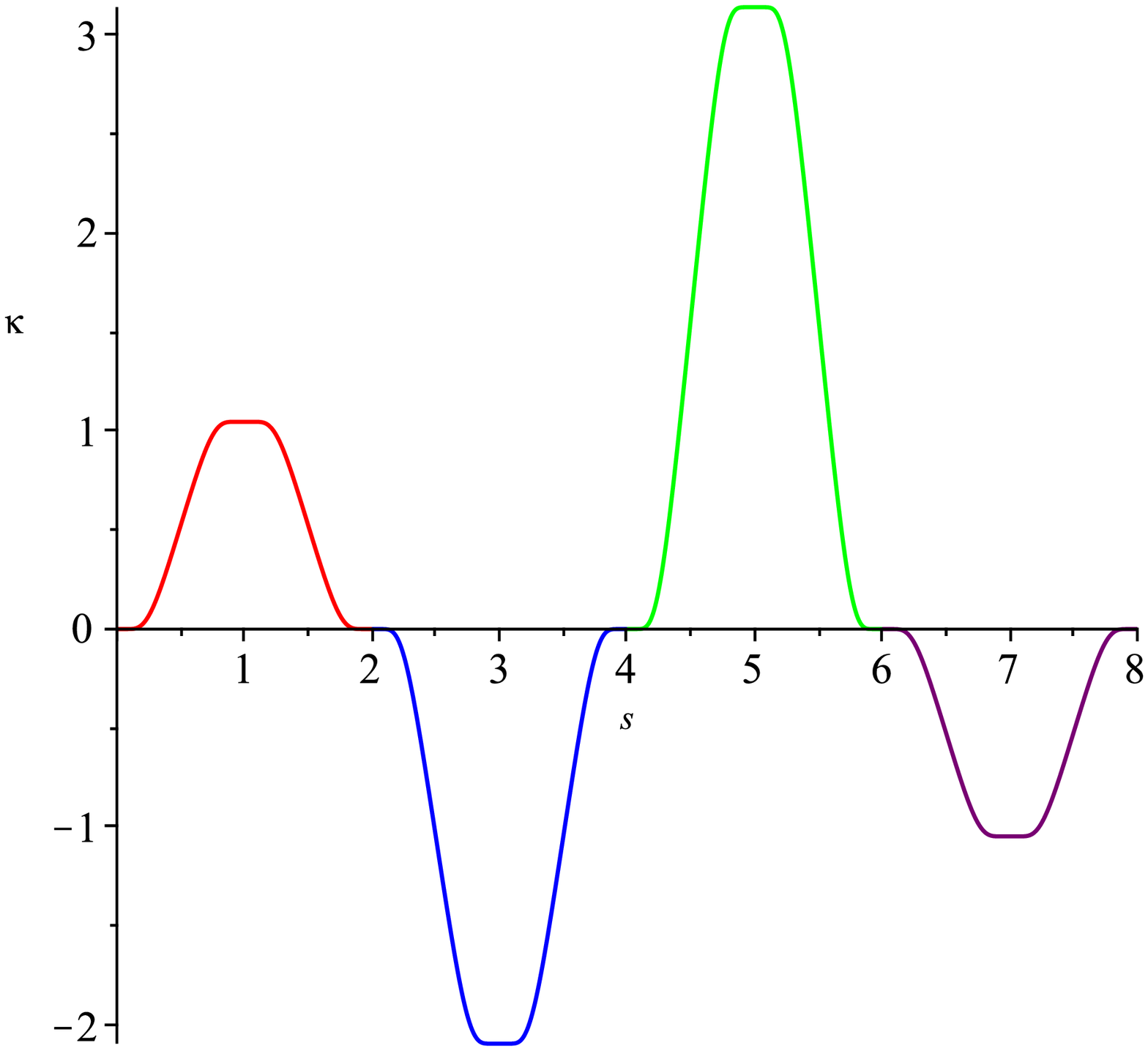}
  }
    \subfigure[Curvature $\ka 2(s)$.]{
    \includegraphics[width=6cm]{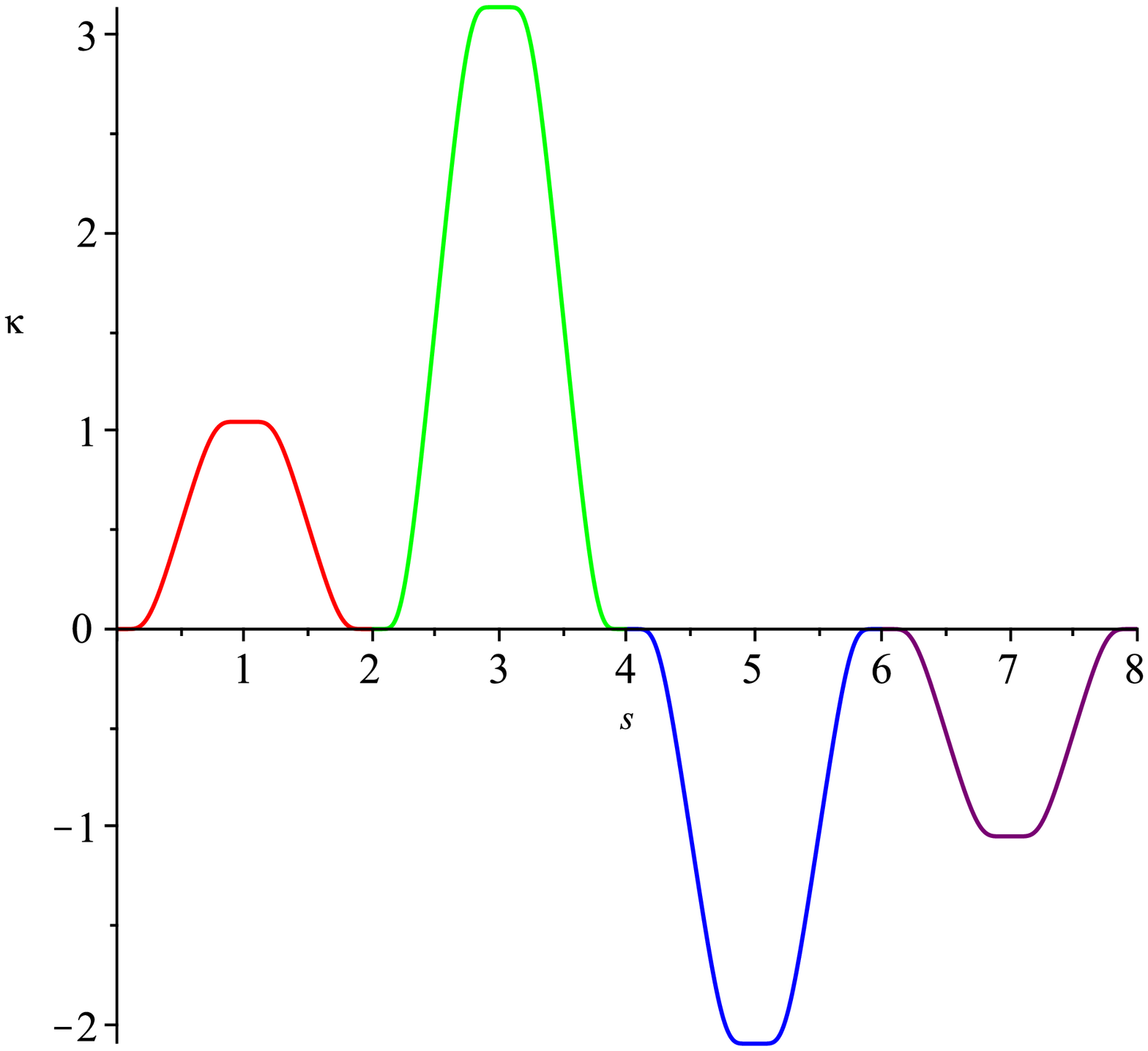}
  }
  \caption{Curvature functions (\ref{ka1}) and (\ref{ka2}) made from scaled cutoff functions.}
\label{fig:Cinf-kappa}
 \end{figure}
 The corresponding parameterizations $\ga 1$ and $\ga 2$ defined by (\ref{eq-gamma}) are $C^\infty$-smooth, because they are obtained by the integration and composition of $C^\infty$-smooth functions. The plots of the corresponding curves $\Gamma_1$ and $\Gamma_2$ in Figure~\ref{fig:Cinf} are obtained
 via a numerical approximation of the integrals appearing in (\ref{eq-gamma}).
These curves are \emph{non-degenerate} since their curvature functions $\kappa_1$ and $\kappa_2$ only contain isolated critical points which happen exactly at the integer values of arc-length.

Due to (\ref{int-f}), $\int_0^8\kappa_1(s)ds=\int_0^8\kappa_2(s)ds=\frac\pi 3$, where $\ka 1$ and $\ka 2$ are defined by (\ref{ka1}) and (\ref{ka2}), respectively. Then, by Lemma~\ref{lem-closed}
these are the curvatures of two closed curves of length $48$ and symmetry index 6.

\begin{figure}
  \centering
    \subfigure[$\Gamma_1$ is the image of $\ga 1(s)$.]{
    \label{fig:Cinfa}\includegraphics[width=6cm]{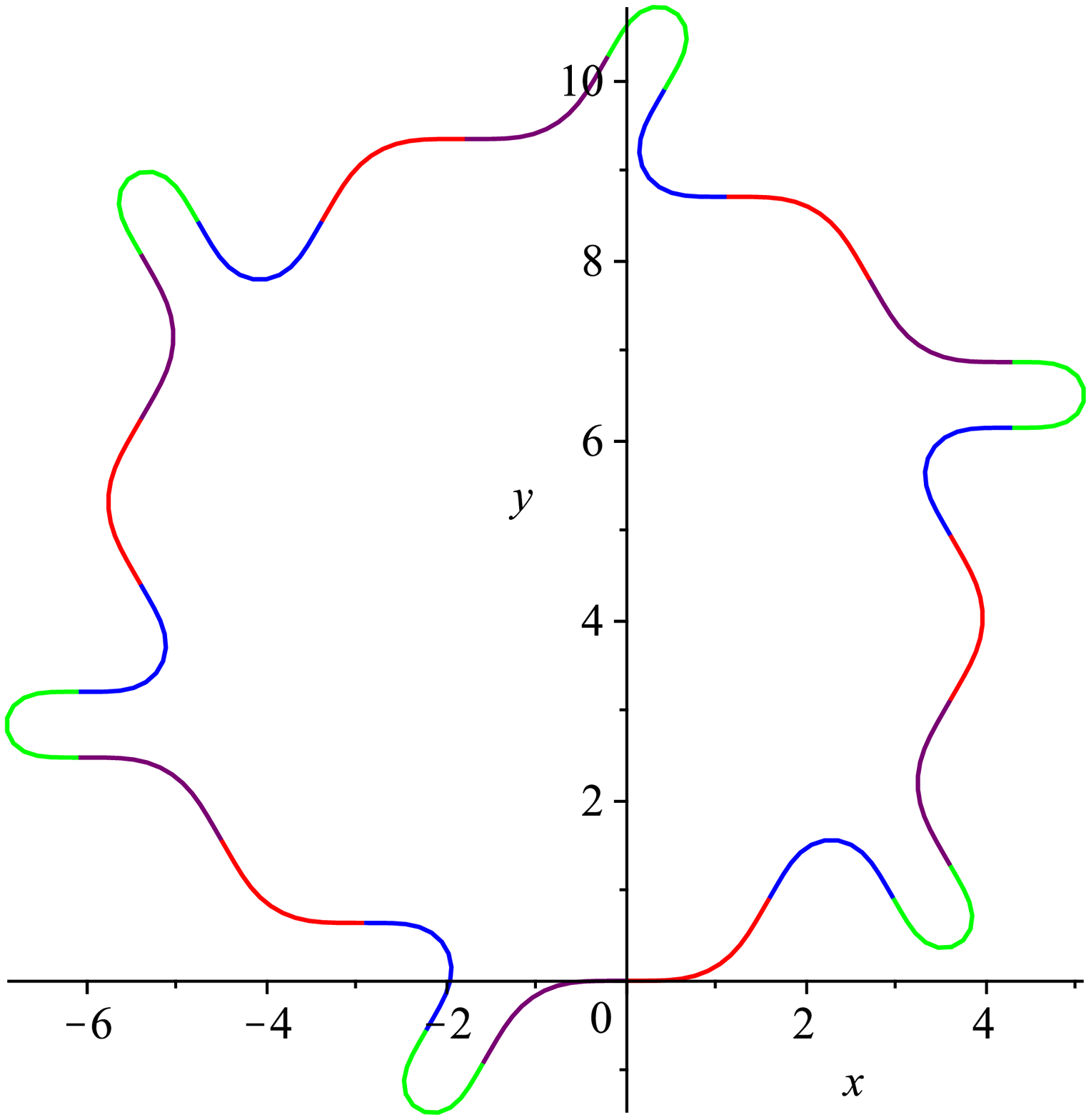}
  }
    \subfigure[$\Gamma_2$ is the image of $\ga 2(s)$.]{
    \label{fig:Cinfb}\includegraphics[width=6cm]{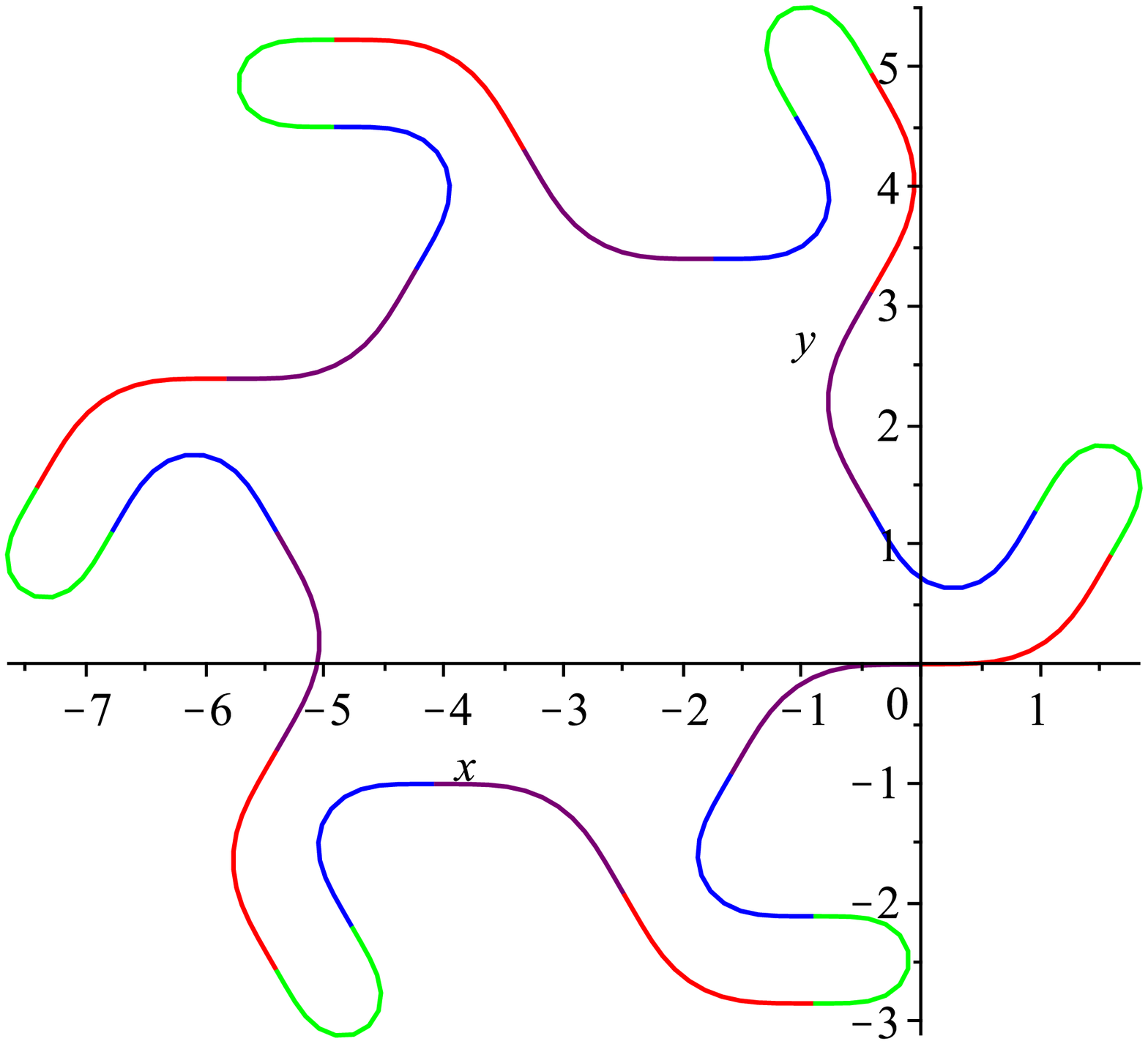}
  }
  \caption{The $C^{\infty}$ curves with curvatures $\ka 1$ and $\ka 2$ given by (\ref{ka1}) and (\ref{ka2}), respectively.}
  \label{fig:Cinf}
\end{figure}

The signatures of $\Gamma_1$ and $\Gamma_2$ are identical, because for every $s_1 \in \mathbb{R}$, there exists $s_2 \in \mathbb{R}$ such that
$${\kappa}_1(s_1) = {\kappa}_2(s_2)\text{ and }
\frac{d \ka1}{ds}(s_1) = \frac {d \ka 2}{ds}(s_2).$$ 
and vice versa. This common signature is presented in Figure~\ref{fig:CinfSig}.
\begin{figure}
  \centering
  \includegraphics[width=6cm]{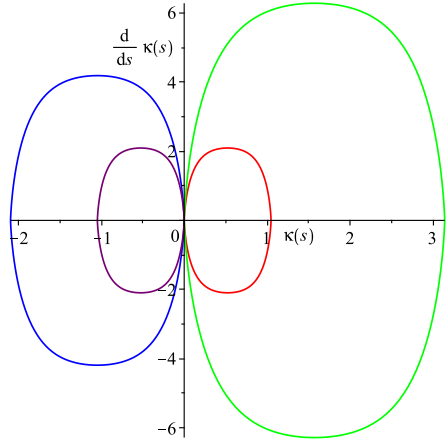}
  \caption{The common signature of non-congruent $C^\infty$ curves $\Gamma_1$ and $\Gamma_2$.}
  \label{fig:CinfSig}
\end{figure}

\begin{remark}\label{rem-gap} \rm
The above example can be used to illustrate the gap in the proof of Theorem 2 in \cite{Hickman2012}, where it is claimed that two non-degenerate $C^3$ curves with the same signature must be congruent. The proof relies on Theorem 3 in \cite{Hickman2012} which, re-written in our notation, claims the following: 
\emph{Let $\ga 1\colon I_1\to \R^2$ and $\ga 2 \colon I_2\to \R^2$ be $C^3$-smooth unit speed parameterizations of curves $\Gamma_1$ and $\Gamma_2$ with respective curvature functions $\kappa_1$ and $\kappa_2$. Let $\Lambda_1$ denote the set of $s\in I_1$, such that $\gamma_1(s)$ is a vertex of $\Gamma_1$. Then $\Gamma_1$ and $\Gamma_2$ have identical signatures if and only if 
there exists a \emph{continuous} surjective map $\rho: I_2 \to I_1$, differentiable for all $s\notin \rho^{-1}(\Lambda_1)$, such  that: 
\beq \label{eq-rho}\rho'(s)=1,\text {for all } s\notin \rho^{-1}(\Lambda_1)
\eeq
\beq
\label{eq-ka-rho} \ka 2 = \ka 1 \circ \rho.\
\eeq
 }
 We show this statement does not hold for curves $\ga 1\colon[0,8]\to \mathbb{R}^2$ and $\ga 2\colon[0,8]\to\R^2$ constructed in this section (see Figure~\ref{fig:Cinf}), with curvatures $\ka 1$ and $\ka 2$ given by (\ref{ka1}) and (\ref{ka2}) (see Figure~\ref{fig:Cinf-kappa}). 
 Indeed, assume $\rho\colon [0,8]\to[0,8]$ is a continuous map satisfying (\ref{eq-rho}) and (\ref{eq-ka-rho}). Differentiation of (\ref{eq-ka-rho}), combined with condition (\ref{eq-rho}) implies that 
\beq
\label{eq-ka-rho'} \ka 2' = \ka 1' \circ \rho 
\eeq
for all $s\notin \rho^{-1}(\Lambda_1)$. Since the set $\Lambda_1$ is discrete, by continuity (\ref{eq-ka-rho'}) must hold for all $s\in [0,8]$.
We note that $\ka2 (1) = 1$, $\ka 2'(1) = 0$ and the only value of $s\in [0,8]$, where $\ka1 (s) = 1$ and $\ka 1'(s) = 0$ is $s=1$. Therefore from (\ref{eq-ka-rho}) and (\ref{eq-ka-rho'}), it must be the case that $\rho(1) = 1$.
Similarly since $\ka 2 (3) = 3$ and the only value of $s\in [0,8]$, where $\ka1 (s) = 3$ is $s=5$, it must be the case that $\rho(3) = 5$.
If $\rho$ were continuous, then by the intermediate value theorem there exists $s_0 \in (1, 3)$, such that 
 $\rho(s_0) = 3$. Then
 $\ka 2(s_0) = \ka 1(\rho(s_0))= \ka 1(3)= -2$ due to (\ref{eq-ka-rho}). However $\ka 2(s) \geq 0$ 
for all $s_0 \in (1, 3)$. Contradiction implies that $\rho$ cannot be continuous. 
\end{remark}

Next we  use our curvature functions $\ka 1(s)$ and $\ka 2(s)$ to define two new curvature functions $\ka 3(s)$ and $\ka 4(s)$ restricted to the interval $[0,16]$ and $[0,24]$ respectively:

\beq
\label{ka3}
\ka 3(s) = \begin{cases}
  \ka 1(s) & 0 \leq s \leq 8 \\
  \ka 2(s) & 8 < s \leq 16 \end{cases}
\eeq

\beq
\label{ka4}
\ka 4(s) = \begin{cases}
  \ka 1(s) & 0 \leq s \leq 8 \\
  \ka 2(s) & 8 < s \leq 16 \\
  \ka 1(s) & 16 < s \leq 24 \end{cases}
\eeq
and then periodically extend to two $C^\infty$ functions on $\mathbb{R}$ with minimal period 16, and 24.

We see that $\int_0^{16}\ka 3(s)ds = \frac{2\pi}{3}$, and $\int_0^{24}\ka 4(s)ds = \pi$ and so by Lemma~\ref{lem-closed} these are curvatures of two closed curves $\Gamma_3$ and $\Gamma_4$ with arc-length parameterization $\gamma_3$ and $\gamma_4$ of length 48 and symmetry index 3 and 2 respectively shown in Figure~\ref{fig:Cinf-sym}.
By the same reasoning above, the signatures of $\Gamma_3$ and $\Gamma_4$ are identical.

\begin{figure}
  \centering
    \centering
    \subfigure[$\Gamma_3$ has symmetry group $\mathbb{Z}_3$.]{
    \label{fig:c3}
    \includegraphics[width=6cm]{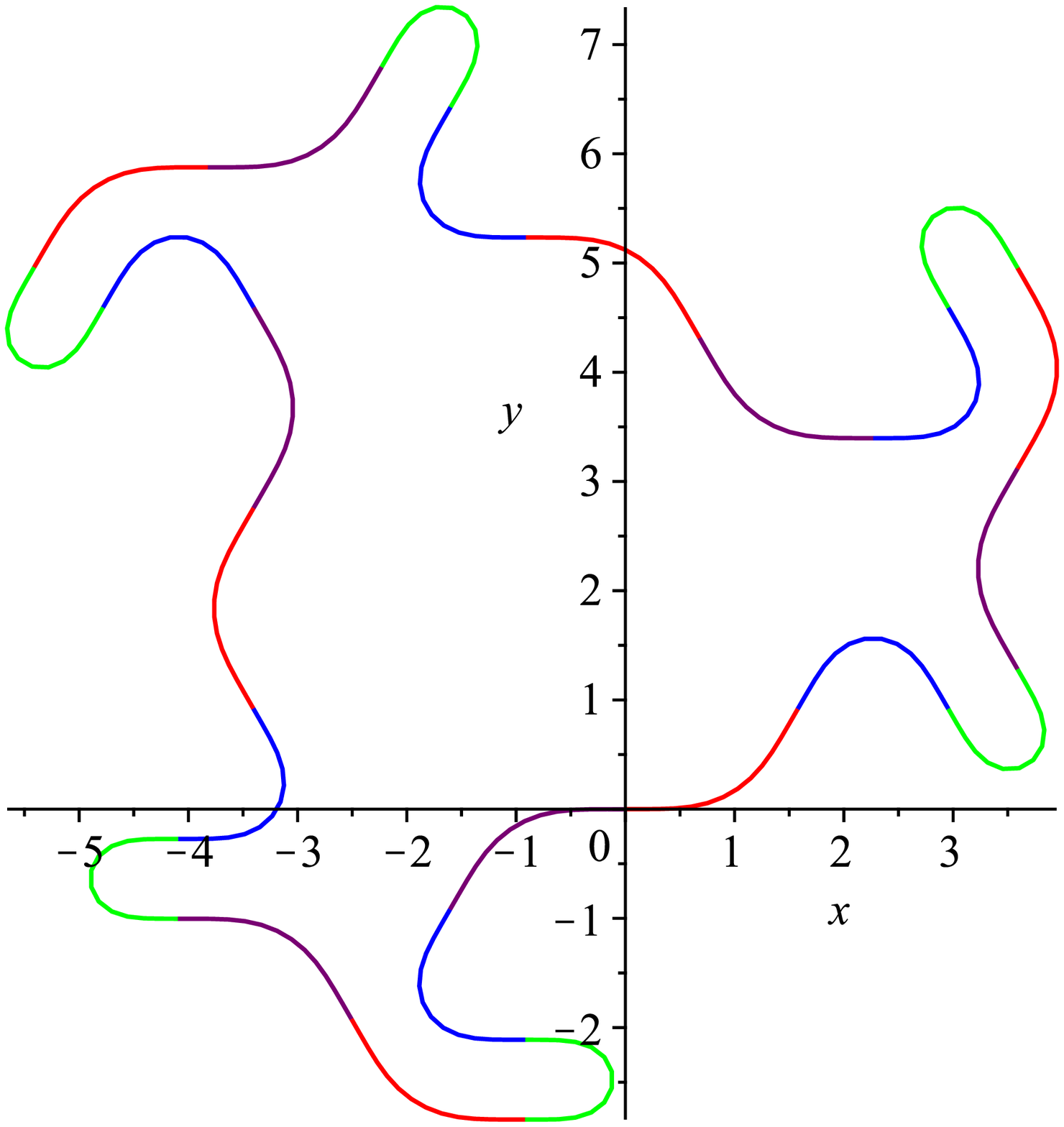}
    }
    \subfigure[$\Gamma_4$ has symmetry group $\mathbb{Z}_2$.]{
    \label{fig:c4}
    \includegraphics[width=6cm]{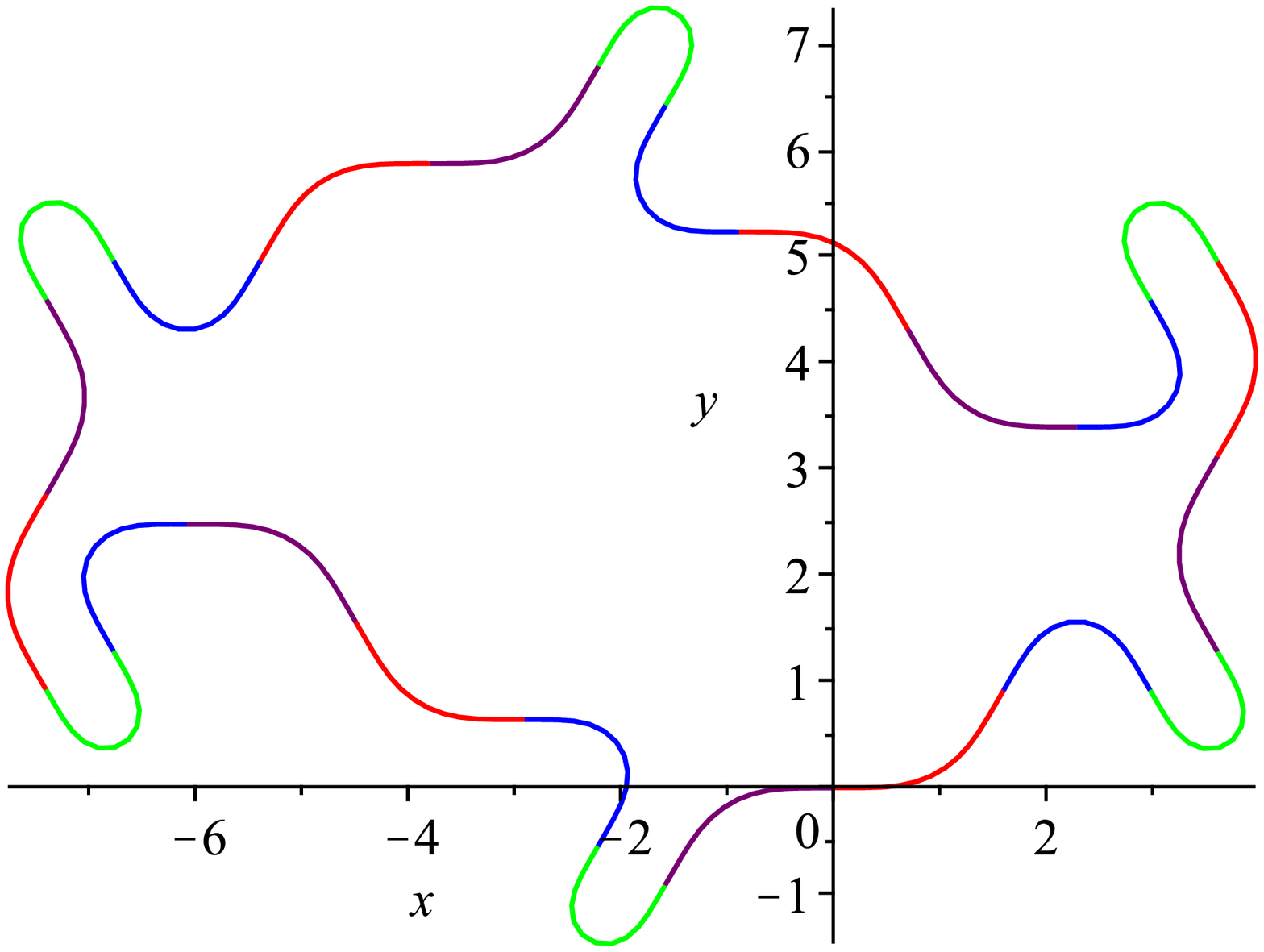}
  }
  \caption{$C^\infty$ curves with same signature curve, and different symmetry groups with curvatures $\ka 3$ and $\ka 4$ given by (\ref{ka3}) and (\ref{ka4}), respectively.}
    \label{fig:Cinf-sym}
\end{figure}
\begin{remark}\label{rem-loc-sym} 
  We can use curves constructed in this section to illustrate the notion of local symmetries based at a point (see Definition~\ref{def-loc-sym}).  Note that for the four curves pictured in Figures~\ref{fig:Cinf} and~\ref{fig:Cinf-sym} pieces of the same color are congruent.  If we take a point $p$ on a piece of certain color (say red) then a local symmetry based at this point is a composition of a rotation and translation that  maps this piece to another piece of the same color (red). Examining more closely, we see that for the two curves pictured in Figure~\ref{fig:Cinf} , the local symmetry set based at any point  coincides with the global symmetry group of the curve, i.e. the group  $\Z_6$, generated by rotations by 60 degrees. For the curves pictured in Figure~\ref{fig:Cinf-sym}, the situation is more subtle. Let $p$ be any point which is not  an end-point of a colored curve segment. Then the cardinality of the set of local symmetries at $p$ (the symmetry index at $p$) is 6 (the number of pieces of the same color), but these local symmetries, consisting of certain compositions of rotations and translations, \emph{do not} form a subgroup of $SE(2)$. At the end-points of the colored pieces, the situation is even more interesting. For instance, the end-points between a green and purple segments have symmetry index 3 for    $\Gamma_3$ and $4$ for  $\Gamma_4$.  The global symmetry groups   for these curves: $\Z_3 $ for $\Gamma_3$ and $\Z_2$ for $\Gamma_4$ and they are always contained in each set of local symmetries. {Although the points on the curves in Figure~\ref{fig:Cinf} and all but finitely many points on the curves in Figure~\ref{fig:Cinf-sym}
  have the same symmetry index, this is not a general phenomenon. }
  For instance, for the curve pictured in Figure~\ref{fig:Cinf_sym2_indexSwap_1} the symmetry indices  of points depend on their color: it is 10 for points inside a red segment\footnote{Except for the midpoint of ``long'' red segments, where it is 2, because the long red segment is an end-to-end attachment of two ``short" red pieces.}, 4 for points inside purple  and green pieces, and 6 for points inside blue pieces. It is a fun exercise for a reader to determine the symmetry indices of the end-points of colored segments. In the next section, we show how  the cardinalities of local symmetry sets are reflected in the paths along the signature quivers.
\end{remark}

\section{ Signature quiver}\label{sect-graph}
In Section~\ref{sect-smooth} we constructed our examples by a piecewise reshuffling of the curvature function under a non-continuous bijection $\rho$ of its minimal period interval $[0,\ell]$. In this section, we provide a general mechanism for constructing non-degenerate non-congruent curves with identical signatures. 
We focus our attention  on \emph{non-degenerate closed curves}, whose  signatures have a finite number of self-intersections, leaving the study of a rather exotic class of curves whose signatures have infinitely many points of self-intersection for future research.
We associate a directed graph that may have loops and multiple edges between vertices (a quiver) to a signature curve and show that different paths along the quiver produce non-degenerate non-congruent curves with identical signatures. We establish a one-to-one correspondence between equivalence classes of paths along the signature quiver and congruence classes of non-degenerate curves with the same signature.
To set the stage, we prove the following lemma
\begin{lemma} \label{lem-comp} Let  $S$ be a closed curve with a continuous locally injective parametrization $\sigma$ of period $\ell$.  Assume the set $Q$ of self-intersection points of $S$ is non-empty and finite. Then  
$S \backslash Q$  is the union   of finitely  many disjoint open curve segments $\{S_{i}\}$, $i=1,\dots, N$, and, if $I$ is a connected component of $\sigma^{-1}(S_i)$, then $\sigma|_I$ is a homeomorphism.  
\end{lemma}
\begin{proof} The restriction $\sigma|_{[0,\ell]}$ is a proper map \footnote{A map between topological spaces is called \emph{proper} if the preimage of every compact subset is compact. Any continuous map with a compact domain to $\R^2$ (and, more generally, to any Hausdorff space) is proper.}.  Up to  a shift in the parameterization, we can assume that $\sigma(0)\in Q$. Then 
$\sigma^{-1}(Q)\cap [0,\ell]$ is a finite set of points which can be ordered as follows $0=c_0< c_1<\dots <c_{k}=\ell$.  For $j=1,\dots, k$, let  $I_j=(c_{j-1},c_{j})$. The image $\sigma(I_j)$ of an interval $I_j$ is contained in a connected component $\hat S$ of $S \backslash Q$.
The curve piece $\hat S$ does not contain any self-intersection points and, therefore is locally Euclidean. Hence $\hat S$ is a one dimensional topological manifold embedded in $\R^2$.  It is  not compact because  it does not contain limit points $\sigma(c_{j-1})$ and $\sigma(c_{j})$.
By the classification of 1-dimensional topological manifolds, there exists a homeomorphism $\phi\colon \hat S\to\R$,  and $\hat S$ is an open curve segment.  By the intermediate value theorem, a locally injective,  surjective continuous  map $\phi\circ\sigma|_ {I_j}$  is a homeomorphism on its image $J\subset\R$, where $J$ is an open connected subset of $\R$.
If the closure of $J$ were a proper subset of $\R$, then  at least one of the limits $\lim_{t\to c_{j-1}} \phi\circ\sigma$ or $\lim_{t\to c_{j}} \phi\circ\sigma$ would be finite. This would imply that  $\hat S$  contains $\sigma(c_{j-1})\in Q$ or $\sigma(c_{j})\in Q$, contradicting the assumption that $\hat S\subset S \backslash Q$.
 Therefore, $J=\R$, $\sigma(I_j)=\hat S$ and $\sigma|_{I_j}\colon I_j\to \hat S$  is a homeomorphism.

Conversely, since 
$Q$ is a closed subset of $S$ in the subset topology, then  $S\backslash Q$  is open in the subset topology. Let $\hat S$ be a connected component of  $S\backslash Q$. The preimage $\sigma^{-1}(\hat S)$ is a union of open intervals, and since $\ell$ is the period of $\sigma$, then for any connected component $I\subset \sigma^{-1}(\hat S)$, there exists an interval $I_j\subset[0,\ell]$ defined in the previous paragraph and a constant $c$ such that for all $t\in I$,  $\sigma|_I (t)=\sigma|_{I_j}(t+c)$ and so is homeomorphism.  Since the number of open  intervals $I_j$, $j=1,\dots,k$,  defined in the previous paragraph is finite, so is the number of  connected components   $S\backslash Q$.
 
\end{proof}

 \begin{definition}[Signature quiver]\label{def-graph} Let $S$ be an oriented closed curve with at least one, but a finite number, of self-intersection points $q_1,\dots,q_n$, $n\geq 1$. By removing these points from $S$ we obtain a collection of disjoint
open  curve segments $\{S_{i}\}$, $i=1,\dots, N$. The \emph{quiver $\Delta_S$} of $S$ is defined as follows. The points, $q_1,\dots,q_n$ correspond to \emph{vertices}\footnote{A terminology warning: please do not confuse the vertices of the simple planar curve $\Gamma$ as in Definition~\ref{def-vertex} and vertices of the quiver  $\Delta_S$.}
 of $\Delta_S$. Each curve segment $\{S_{i}\}$, $i=1,\dots, N$  corresponds to an \emph{edge} connecting the appropriate vertices (corresponding to the end points of the segment $\{S_{i}\}$ in $S$), whose \emph{direction} is dictated by the orientation of $S$. If $S_\Gamma$ is the signature of a curve $\Gamma$, we call the quiver  $\Delta_{S_\Gamma}$ \emph{the signature quiver of $\Gamma$}.
 \end{definition}

 In what follows, we will  identify and use the same notation for a quiver-vertex $q_j$ and the corresponding  self-intersection point $q_j\in S$, as well as for a quiver-edge $S_i$ and the corresponding curve segment  $S_i\subset S$. 

\begin{definition}[Path]\label{def-path}
 Given a quiver, we define a \emph{path} to be a sequence of edges $S_{i_1}\dots S_{i_j}$, such that for $k=1,\dots,j-1$, the ending vertex of the edge $S_{i_k}$ is the beginning vertex of the following edge $S_{i_{k+1}}$. If the ending vertex of the last edge in a path coincides with the beginning vertex of the first edge in a path, the path is called \emph{closed}. We  call closed paths \emph{equivalent} if they agree up to a cyclic permutation. Each path induces a \emph{multiplicity} of an edge -  the number of times an edge is included in the path.  If each edge is contained in a path we  call such a path \emph{complete}.
 Two paths along the same quiver are called \emph{compatible} if they induce the same multiplicities.  \end{definition}
A non-degenerate closed curve $\Gamma$ with a prescribed parameterization   $\gamma(t)$   induces a  complete, closed  path on its signature quiver $\Delta_{S_\Gamma}$. Indeed, assume $\gamma(t)$ has a minimal period $L$. By possibly a shift of the parameter, we may assume that $\sigma_\gamma(0)$ is one of the points of self intersection of $S_\Gamma$.
Let $Q$ be the self intersection points of the signature of $\Gamma$. Similar to  what we did in the proof of Lemma~\ref{lem-comp}, the finite set  of points $\sigma_\gamma^{-1}(Q)\cap [0,L]$  can be ordered as follows $0=c_0< c_1<\dots <c_{k}=L$. This gives rise to a sequence of edges $S_i=\sigma_\gamma(I_i), i=1,\dots, k$, where $I_i=(c_{i-1},c_{i})$. As $t$ varies from $0$ to $L$,  the   map $\sigma_\gamma$ induces a path along $\Delta_{S_\Gamma}$. 
It is not difficult to see that if $\mu_i$ is the multiplicity of the edge $S_i$ induced by this path, then $\sigma_\gamma^{-1}(S_{i})\cap (0,L)$  is the union of $\mu_i$  disjoint open intervals. For simple curves it follows that any other parameterization $\tilde \gamma(t)$, such that   $\sigma_{\tilde\gamma}(0)$ is a self intersection-point, induces an equivalent path. {Non-simple curves may allow different parameterizations that induce non-equivalent paths (see Remark~\ref{rem-flower}).
  \begin{example}\label{ex-quiver}Figure~\ref{fig:directed-graph} shows the signature quiver for the signature shown in Figure~\ref{fig:CinfSig} from Section~\ref{sect-smooth}. It has a single vertex and 4 edges (loops). It turns out that the four paths induced by curves $\Gamma_1, \dots,\Gamma_4$ in Figures~\ref{fig:Cinf} and~\ref{fig:Cinf-sym} are compatible, inducing multiplicity 6 on each of the four edges. We find it useful to \emph{label} each edge by a letter and to use \emph{exponents} to indicate the multiplicities of the edges induced by a path. More graphs are shown in the subsequent sections, including graphs with more than one vertex (Figures~\ref{fig:MNgraph}) and with edges with varying multiplicities (Figures~\ref{fig:cogGraph}).
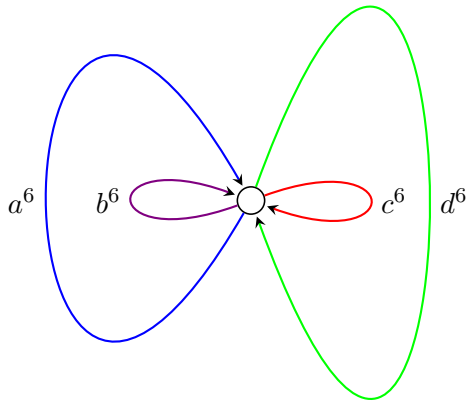
\begin{figure}
\centering
\begin{tikzpicture}[
            > = stealth, 
            shorten > = 1pt, 
            auto,
	    node distance = 2cm, 
            semithick 
        ]

        \tikzstyle{every state}=[
            draw = black,
            thick,
            fill = white,
            minimum size = 4mm
        ]
	 \useasboundingbox (-3,-3) rectangle (3,3);

	\node[shape=circle,draw=black] (5) {};

	 \path[->] (5) edge [out=20, in=340, loop, distance=2cm, draw=red, thick] node {$c^6$} (5);
	 \path[->] (5) edge [out=200, in=160, loop, distance=2cm, draw=violet, thick] node {$b^6$} (5);
	 \path[->] (5) edge [out=70, in=290, loop, distance=9cm, draw=green, thick] node {$d^6$} (5);
	 \path[->] (5) edge [out=240, in=120, loop, distance=7cm, draw=blue, thick] node {$a^6$} (5);
     \end{tikzpicture}
     \caption{The  signature quiver for the signature in Figure~\ref{fig:CinfSig}. 
     The superscript on each of the labeled edges denotes the multiplicity of the paths induced by Figures~\ref{fig:Cinf}, \ref{fig:Cinf-sym}, and \ref{fig:Cinf-ex}.}
\label{fig:directed-graph}
  \end{figure}
\end{example}


As in Example~\ref{ex-quiver}, we  continue to label edges by letters and we  interchangeably use terms ``path'' and ``word''. We call a path (a word) $W$ \emph{periodic} if it can be written $W = (w)^m$ such that $m > 1$.
If the word $W$ is periodic and $W = (w)^m$ such that $w$ is not periodic, then the $w$ is called a \emph{minimal repeated subword}, while the length of $w$ is called the \emph{minimal period} of $W$. The words generated by the curves $\Gamma_1, \Gamma_2, \Gamma_3$, and $\Gamma_4$ in Figures~\ref{fig:Cinf} and~\ref{fig:Cinf-sym} in Section~\ref{sect-smooth} are $(cadb)^6, (cdab)^6, (cadbcdab)^3$, and $(cadbcdabcadb)^2$, respectively.


The following proposition shows the relationships between the multiplicities induced  by $\Gamma$ on  its signature quiver and various indices introduced in Section~\ref{sect-prelim}
\begin{proposition}[Multiplicities and indices]\label{prop-mlt} Let $\Gamma$ be a simple, non-degenerate, closed curve, whose signature $S_\Gamma$ has finitely many points  of self-intersection: $q_1,\dots,q_n$, $n\geq1$ (corresponding to signature-quiver-vertices). By removing these points from $S_\Gamma$ we obtain a collection of disjoint open curve segments, $\{S_{i}\}$, $i=1,\dots, N$ (corresponding to the signature-quiver-edges).  For  $i\in\{1,\dots,N\}$, let   $\mu_i$ be the multiplicity of the edge $S_i$ induced by $\Gamma$. 
Then the preimage 
$\sigma_\Gamma^{-1}(S_i)$, under the signature map consists of $\mu_i$ connected congruent open curve segments     $\Gamma_i^1,\dots, \Gamma_i^{\mlt_i}$. 
For $i=1,\dots N$, the local signature index and the local symmetry index for any $p\in  \Gamma_i^j$, $j=1,\dots,\mlt_i$ are equal to $\mu_i$: 
\beq\label{eq-mlt}\mlt_i=\text{sig-index}(\Gamma_p)= \text{sym-index}(\Gamma_p).\eeq 
In addition, the (global) signature index of $\Gamma$:
\beq\label{eq-mlt-min}\text{sig-index}(\Gamma)=\min_{i=1,..,N}\{\mlt_i\}.\eeq
\end{proposition}

\begin{proof}Let $\gamma\colon \R\to \Gamma$ be a parameterization of $\Gamma$ with a minimal period $L$ and $\sigma_\gamma=\sigma_\Gamma\circ\gamma$ be the parameterized signature map defined in Definition~\ref{def-sig}. By shifting the parameter, we can always assume that $\sigma_\gamma(0)$ is one of the self intersection points, say $q_1$.
Then, since $\sigma_\gamma=\sigma_\Gamma\circ\gamma$ is continuous, $\sigma_\gamma^{-1}(S_{i})\cap (0,L)$ is the union of $\mu_{i}$ disjoint open intervals $I_i^1,\dots, I_i^{\mu_i}$.  Since  a phase portrait has a well defined orientation and the values where $\dot{\kappa}=0$ are discrete, the restriction of $\sigma_\gamma$ to each of these intervals is an injective continuous map and therefore is a homeomorphism on its image $S_i$ (the continuous bijection between two sets homeomorphic to open intervals is a homeomorphism). Since $\gamma$ is simple, it is also true that $\gamma|_{I_i^j}$ is also a homeomorphism on its image. It follows that  $\sigma_\Gamma^{-1}(S_i)$ is a collection of open curve segments and the map $\gamma$ induces a bijection between the set  $\{\sigma_\gamma^{-1}(S_i)\}$ and  $\{\sigma_\Gamma^{-1}(S_i)\}$.
 We order $\{\Gamma_i^k\}$, so that $\gamma(I_i^j)=\Gamma_i^j$.
Again using the injectivity of the restrictions $\sigma_\gamma|_{I_i^j}$ (and hence the injectivity $\sigma_\Gamma|_{\Gamma_i^j}$), we can now conclude that
\beq\label{eq-mult-aux}\mlt_i=\text{sig-index}(\Gamma_p) \text{ for any $p\in  \Gamma_i^j$, $j=1,\dots,\mlt_i$}.\eeq
  
The second equality in (\ref{eq-mlt}) could be derived from the more general Proposition 3.8 of \cite{olver15}, however we found it instructive to provide an explicit proof in this simpler case. {From Proposition~\ref{prop-loc}, it follows that  for any fixed $i=1,\dots,\mu_i$, all curve pieces $\Gamma_i^j$, $j=1,\dots,\mlt_i$ are congruent.  Therefore, the local symmetry set $\sym(\Gamma_p)$ contains at least $\mu_i$ elements. On the other hand, assume that  for some  $j=1,\dots,\mlt_i$ and  $p\in  \Gamma_i^j$, there exists  an open connected subset $U\subset \R^2$, containing $p$ and $g\in SE(2)$ satisfying the local symmetry condition (\ref{eq-loc-sym}). We may assume, by possibly shrinking $U$, that $\Gamma\cap  U\subset \Gamma_i^j$ and so $\sigma_\Gamma(\Gamma\cap U)\subset S_i$. Then due to the invariance of  the signature map $\sigma_\Gamma$, we have $\sigma_\Gamma(g\cdot (U\cap \Gamma))\subset S_i$. 
 Therefore, there exists $k\in {1,\dots,\mu_i}$, such that   $g\cdot (U\cap \Gamma)=(g\cdot U)\cap \Gamma\subset \Gamma_i^k$. Due to the congruence of these curve pieces, there exists $\hat g\in SE(2)$,  such that $\hat g\cdot \Gamma_i^j= \Gamma_i^k$. Then $g_1=\hat g^{-1} g$ is a local symmetry based at $p$, such that $g_1\cdot (U\cap \Gamma_i^j)\subset \Gamma_i^j$. Since the signature map is group invariant, we have that for any $p_1\in U\cap \Gamma_i^j$, $g_1p_1\in g_1 U\cap \Gamma_i^j$ and $\sigma_\Gamma(p_1)=\sigma_\Gamma(g_1p_1)$.   Since  $\sigma_\Gamma|_{\Gamma_i^j}$ is injective, this implies that $g_1p_1=p_1$ and so  $g_1\in SE(2)$ fixes (i.e. maps to itself) every point of $U\cap \Gamma_i^j$. The only element of $SE(2)$ that can fix every point of a curve segment  is the identity. Thus $g=\hat g$ and }

  $$\mlt_i=\text{sym-index}(\Gamma_p) \text{ for any $p\in  \Gamma_i^j$, $j=1,\dots,\mlt_i$}.$$
  Finally, to establish (\ref{eq-mlt-min}), we note that if $p\in\Gamma$ is such that $q=\sigma_\Gamma(p)$ is a self-intersection point of  $S_\Gamma$, then by continuity of $\sigma_\Gamma$,  for any edge $S_i$ that originates  at $q$, we have at least $\mu_i$  distinct preimages of $q$ under the signature map $\sigma_\Gamma$. Since   every self-intersection point of $S_\Gamma$ is a starting point for an edge,  it   follows that $\text{sig-index}(\Gamma_p)\geq \min_{i=1,..,N}\{\mlt_i\}$, and then (\ref{eq-mlt-min}) follows from (\ref{eq-sig-ind}) and~(\ref{eq-mult-aux}).
\end{proof}

\begin{remark} 
  If $\Gamma$ is not simple, then, in the context of Proposition~\ref{prop-mlt}, $\sigma_\Gamma^{-1}(S_i)$  may be  a union of disjoint open curve segments and curve pieces that are not homeomorphic to an open interval (see, for instances, red pieces of the curve pictured in Figure~\ref{fig:Cinf_notSimple}), and the statements about indices in  Proposition~\ref{prop-mlt} may not hold.
\end{remark}

We now formulate a symmetry and congruence criteria for  non-degenerate  curves in terms of paths they induce on  their signature quivers.
\begin{proposition}[Symmetry in terms of the signature quiver]\label{prop-sym-graph} Assume $\Gamma$ is a {simple} closed non-degenerate curve, whose signature $S_\Gamma$ has at least one, but a finite number of self intersections. If $\Gamma$ induces a word $W = (w)^m$ such that $w$ is not a periodic word, then
$$m=\symi(\Gamma)$$ 
and so $sym(\Gamma)$ is the  rotation group congruent to $\Z_m$. 
\end{proposition}
\begin{proof} Take $\gamma$ to be an arc-length parameterization of $\Gamma$, such that $\sigma_\gamma(0)$  is a self-intersection point of the signature $S_\Gamma$ and let $\kappa\colon\R\to\R$ be the corresponding curvature function. Assume that $L$ is the minimal period of $\gamma$ and let $\ell=\frac L m$. 
 By construction for $i=0,\dots, m-1$: 
\beq \label{eq-kappa_i}\kappa|_{[i\ell,(i+1)\ell]}(s)=\kappa|_{[0,\ell]}(s-i\ell). \eeq
and so 
 \beq\label{eq-int-Ll}\int_0^L\kappa(s)ds=m\int_0^\ell\kappa(s)ds=2\pi,
\eeq
and the conclusion follows from {Proposition~\ref{prop-integrals}.}

\end{proof}
\begin{example} The four curves constructed  in Section~\ref{sect-smooth} have the signature quiver pictured on Figure~\ref{fig:directed-graph}.
The two curves in Figure~\ref{fig:Cinf} induce words $(cadb)^6$, and $(cdab)^6$, respectively, while the two curves in Figure~\ref{fig:Cinf-sym}  induce words $(cadbcdab)^3$ and $(cadbcdabcadb)^2$, respectively. Both curves in Figure \ref{fig:Cinf} have symmetry group $\mathbb{Z}_6$, while the curve in Figure~\ref{fig:c3} has symmetry group $\Z_3$, and the one in Figure~\ref{fig:c4} has symmetry group $\Z_2$.
We observe that the exponent on a minimal repeated subword is the symmetry index of the corresponding curve.
\end{example}

\begin{theorem}[Congruence in terms of signature quivers]\label{thm-cong-graph} 
Assume $\Gamma_1$ and $\Gamma_2$ are closed non-degenerate curves with the same signature $S$. Assume $S$ has at least one, but a finite number of self intersections. If there exist parameterizations of  $\Gamma_1$ and $\Gamma_2$ that induce equivalent  paths on $\Delta_S$ then $\Gamma_1$ and $\Gamma_2$ are congruent.  If $\Gamma_1$ and $\Gamma_2$ are simple then the converse is true as well. \end{theorem}

\begin{proof}
  The second statement follows from Proposition~\ref{prop-cong}.  
  For the  first statement, we can equivalently assume that there exist  unit speed parameterizations $\gamma_1$ and $\gamma_2$ of  $\Gamma_1$ and $\Gamma_2$  that induce the same  path $S_{i_1},\dots, S_{i_k}$ on the signature quiver  $\Delta_S$. 
Here $S_{i_j}$ denotes both the edge of the quiver and the corresponding open segment of the curve $S$. 
Then  $\sigma_{\gamma_1}(0)=\sigma_{\gamma_2}(0)$ corresponds to the starting {quiver-}vertex of $S_{i_1}$. Let $L_1$ and $L_2$ be the minimal periods of $\gamma_1$ and $\gamma_2$ respectively, and $Q$ be the set of self intersection points of $S$. As in the proof of Lemma~\ref{lem-comp}, we let $0=c_{1;0}< c_{1;1}<\dots<c_{1;k}=L_1$ be the ordered finite set of points
 $[0,L_1] \cap \sigma_{\gamma_1}^{-1}(Q)$ and similarly let, $0=c_{2;0}< c_{2;1}<\dots<c_{2;k}=L_2$ be the ordered finite set of points
 $[0,L_2] \cap \sigma_{\gamma_1}^{-1}(Q)$. 

 Let $I_{1;r}=(c_{1;{r-1}}, c_{1;r})$ and $I_{2;r}=(c_{2;{r-1}}, c_{2;r})$ for $r=1,\dots,k$. Then $S_{i_r}=\sigma_{\gamma_1}(I_{1;r})=\sigma_{\gamma_2}(I_{2;r})$, and by Lemma~\ref{lem-comp}, the maps   $\sigma_{\gamma_1}|_{I_{1;r}}$ and  $\sigma_{\gamma_2}|_{I_{2;r}}$ are homeomorphisms onto $S_{i_r}$. From  Proposition~\ref{prop-loc}, for each $r=1,\dots,k$, the curve segments $\Gamma_{1;r}=\gamma_1\left(I_{1;r}\right)$ and $\Gamma_{2;r}=\gamma_1\left(I_{2;r}\right)$ are congruent.
Proposition~\ref{prop-cong} implies that the corresponding curvature functions are related by a shift: for each $r=1,\dots,k$ there exists a constant $d_r\in \R$, such that $\kappa_2|_{I_{2;r}}(s)= \kappa_1|_{I_{1;r}}(s+d_r)$.

Since $\sigma_{\gamma_1}(0)=\sigma_{\gamma_2}(0)$, we have $\kappa_2(0)=\kappa_1(0)$ and so $d_1=0$. It follows that $c_{2;1}=c_{1;1}$ and can be denoted $c_1$. By continuity of the curvature functions it follows that $\kappa_2(c_1)=\kappa_1(c_1)$, so $d_2=0$ and continuing inductively we get that $L_1=L_2$ and $\kappa_1|_{[0,L]}=\kappa_2|_{[0,L]}$.
 Since $L$ is a (possibly non-minimal) period of both $\kappa$'s we conclude that $\kappa_1(s)=\kappa_2(s)$ for all $s\in\R$. Then by Proposition~\ref{prop-cong}, $\Gamma_1\cong\Gamma_2$.
\end{proof}

Our next goal is to show that any path $W$ (compatible or not) on the signature quiver
$\Delta_{S_\Gamma}$ of a curve $\Gamma$ gives rise to a non-degenerate curve $\Gamma_W$ whose curvature $ \kappa_W$ is a $C^1$-smooth function obtained by a \emph{piece-wise reshuffling} of the curvature of $\Gamma$,
as we define rigorously below.

\begin{corollary}\label{cor-from-path-to-curve}
Any path $W$ (compatible or not) on the signature quiver
$\Delta_{S_\Gamma}$ of the curve $\Gamma$ induces a non-degenerate curve  $\Gamma_W$ that is unique up to actions of $SE(2)$.
\end{corollary}
\begin{proof}
  Assume {$\{S_{1},\dots, S_{N}\}$} is the set of edges  of $\Delta_{S_\Gamma}$ and the word (path)  $W$ is an ordered sequence of edges $S_{i_1},\dots, S_{i_k}$, where the edges may repeat.  Let $q_{i_{r-1}}$ and $q_{i_r}$ be the beginning and the ending quiver-vertices of edge $S_{i_r}$ respectively. An edge $S_{i_r}$ corresponds to an open curve segment of the signature $S_\Gamma$ whose end points are points of self intersection of the signature.
Let $\gamma$ again be a unit speed parameterization of $\Gamma$ and $\kappa$ be the corresponding curvature function.
As we discussed before, for each $r$, $\sigma^{-1}_\gamma(S_{i_r})$ is a union of disjoint intervals, such that the restrictions of $\kappa$ to any two of these intervals are related  by a shift in the parameter. For each $r=1,\dots, k$, we choose one of such intervals $I_r=(a_r,b_r)$. By construction we have that $\sigma_\gamma(b_r)=\sigma_\gamma(a_{r+1})=q_{i_r}$ for $r=1,\dots,k-1$ and since $\sigma_\gamma(s)=(\kappa(s),\dot{\kappa}(s))$, we have 
\beq\label{eq-cont}\kappa(b_r)=\kappa(a_{r+1}) \text{ and } \dot{\kappa}(b_r)=\dot{\kappa}(a_{r+1}).\eeq
In addition, if $W$ is  a closed path, we have 
\beq\label{eq-cont-c}\kappa(b_{k})=\kappa(a_{1}) \text{ and } \dot{\kappa}(b_k)=\dot{\kappa}(a_{1}).\eeq 
We  construct ${\kappa}_W$ by reshuffling the restrictions $\kappa|_{I_r}$ as follows. Let
$$c_0=0, \quad c_1=b_1-a_1, \dots, c_k=\ds{ \sum_{j=1}^{k}(b_{j}-a_{j})}.$$
We define a function $\kappa_W$ on the interval $[c_0,c_k]$, by 

\begin{align}
  \label{eq-kappaW} \kappa_W(c_k)=&\kappa(b_k) \text{ and } \kappa_W(s)=\kappa(s+a_r-c_{r-1}),\\\nonumber &\text{ for }r=1,\dots, k \text{ and } s\in [c_{r-1},c_r).
\end{align}

It follows from (\ref{eq-kappaW}) that
\begin{align}
  \nonumber \lim_{s\to {{c_r}^-}}\kappa_W(s)=\kappa({b_r}) &\text { and } 
\lim_{s\to {{c_r}^-}}\dot{\kappa}_W(s)=\dot{\kappa}({b_r}),\\
 & \text{ for } r=1,\dots,k \\
 \nonumber \lim_{s\to {{c_r}^+}}\kappa_W(s)=\kappa(a_{r+1})  &\text{ and }
 \lim_{s\to {{c_r}^+}}\dot{\kappa}_W(s)=\dot{\kappa}(a_{r+1}),\\
 & \text{  for } r=0,\dots,k-1.
 \end{align}
 From (\ref{eq-cont}) it follows that $\kappa_W$ is a $C^1$-smooth function on $[c_0,c_k]$. Moreover, if  $W$ is a closed path, then due to (\ref{eq-cont-c}) $\kappa_W$ can be periodically extended to a $C^1$-smooth function on $\R$.
 Using $\kappa_W$ in (\ref{eq-gamma}) we obtain a unit speed  curve $\Gamma_W$ of length $L_W = c_k - c_0$. If $\kappa_W$ is extended to $\R$, we can similarly extend the parameterization of $\Gamma_W$ to obtain an open or closed curve. By Theorem~\ref{thm-cong-graph} this curve is unique under actions of $SE(2)$.
\end{proof}


\begin{remark} In the context of the Corollary~\ref{cor-from-path-to-curve}, rather than starting with a quiver associated with the signature of a specific curve $\Gamma$, we could  start with a periodic $C^1$-smooth function $\kappa(s)$, with finitely many critical points per period, and consider various paths on the quiver associated   with the phase portrait ($\kappa(s),\dot\kappa(s)$). The resulting curves is non-degenerate, but as discussed in Remark~\ref{rem-closed-phase}, we are not guaranteed to be able to generate any closed or simple curves. Note also that once we have chosen $\kappa(s)$, then by Lemma~\ref{lemma-cong},  there is a unique, up-to $SE(2)$-action, curve $\Gamma$  whose curvature is $\kappa(s)$, and so starting with a phase portrait is essentially equivalent to starting with a curve.
\end{remark}

\begin{remark}\label{rem-reconstr}
In the context of the above construction, if the path $W$ is complete (i.e.~includes each edge of $\Delta_{S_\Gamma}$) the signature of curve (curve piece) $\Gamma_W$ equals $S_\Gamma$.   
However, the following examples show, that although we start with a closed simple curve $\Gamma$, and even if we take $W$ to be compatible with the word generated by $\Gamma$, the curve $\Gamma_W$ obtained by the above construction does not have to be simple, or closed. In addition $\Gamma_W$,  might not have a consistent orientation even at its simple points.
\end{remark}
\begin{example}
In Section~\ref{sect-smooth}, four closed simple curves with the signature shown  in Figure~\ref{fig:CinfSig} have been constructed, and the corresponding signature-quiver is shown   in Figure~\ref{fig:directed-graph}. A compatible non-periodic path $cadbcdabcdabcadbcadbcdab$ gives rise to \emph{an open curve piece} shown in Figure \ref{fig:Cinf_open} with the same signature.
\begin{figure}
  \centering
  \includegraphics[width=6cm]{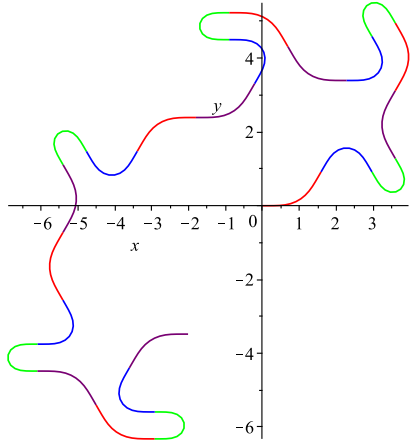}
  \caption{An open curve with the signature curve in Figure \ref{fig:CinfSig}, induced by the non-periodic compatible closed path $cadbcdabcdabcadbcadbcdab$ on the quiver in Figure~\ref{fig:directed-graph}.}  \label{fig:Cinf_open} 
\end{figure}
A compatible path $(cabdcdabcabd)^2$ gives rise to a non-simple closed curve shown in  Figure \ref{fig:Cinf_notSimple}.
\begin{figure}
  \centering
  \includegraphics[width=6cm]{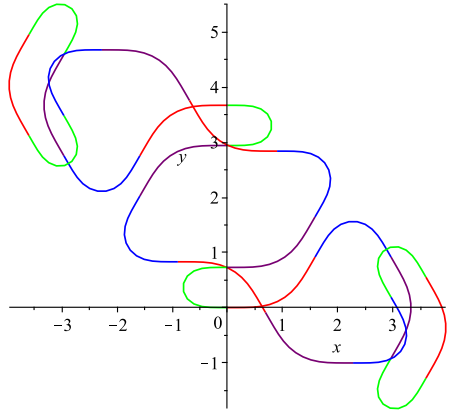}
  \caption{A non-simple curve with the signature curve in Figure \ref{fig:CinfSig} induced by the word $(cabdcdabcabd)^2$.}
  \label{fig:Cinf_notSimple}
\end{figure}
Figure~\ref{fig:Cinf_multi-oriented} shows a curve generated by a non-compatible path $ccbbadbbbcccccbcadbcbccc$ on the same quiver. A regular parametrization consistent with this path assigns some simple points a double orientation. These points are on the ``bridges'' connecting the ``roundabouts'' and are highlighted  in light blue. When traveled from left to right,  the bridges are congruent to the purple segments labeled by $b$ and,  when travelled in the opposite direction, they turn into the red segments labeled by $c$.  It is worth noticing that  this curve does not admit a consistently oriented regular parameterization.
\begin{figure}
  \centering
  \includegraphics[width=6cm]{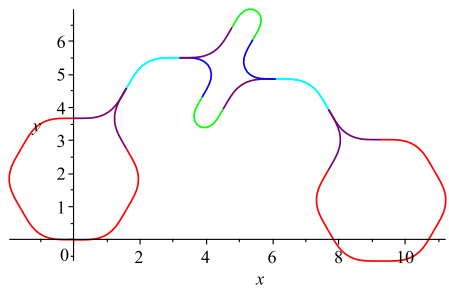}
  \caption{A non-simple curve, with signature curve in Figure \ref{fig:CinfSig}, is  induced by  the word $(ccbbadbbbcccccbcadbcbccc)$ on the signature quiver in Figure~\ref{fig:directed-graph}.  The curve has a double orientation on the points highlighted in light blue.}
  \label{fig:Cinf_multi-oriented}
\end{figure}
\end{example}
 To formulate a sufficient condition for  $\Gamma_W$  to be  closed,  we will find it useful to assign a \emph{weight} to each edge of the signature quiver $\Delta_{S_\Gamma}$ as follows (compare with Section 4 of ~\cite{olver15}). 
As was pointed out in the proof of Proposition~\ref{prop-mlt},  for every $i=1,\dots,N$ and each $j=1,\dots,\mlt_i$, the restrictions of $\gamma$ to intervals $I_i^1$ and $I_i^j$ satisfy the hypothesis of Proposition~\ref{prop-loc}, and so the corresponding curve segments $\Gamma_i^1=\gamma(I_i^1)$ and $\Gamma_i^j=\gamma(I_i^j)$ are congruent.
Then by Proposition~\ref{prop-cong}, the corresponding curvatures are related by a shift: there exists a constant $c_{i}^j$, such that $\kappa|_{I_i^j}(s + c_i^j)=\kappa|_{I_i^1}(s)$.
Hence for each $j=1,\dots, \mlt_i$, we have $\ds\int_{I_i^j}\kappa(s)ds=\int_{I_i^1}\kappa(s)ds$, and so with each $S_i$ we can associate \emph{weight}
 \beq\label{eq-weight}\omega_i = \ds{\int_{I_i^1} \kappa(s)ds}.\eeq

From the definitions of multiplicity and weight, it follows that 
\beq\label{eq-wmu}\ds{\sum_{i=1}^N} \mlt_i\omega_i = \ds{\int_0^{L_W}\kappa(s)ds}\eeq
where $L_W$ is the length of the reconstructed curve $\Gamma_W$ as described in the proof of Corollary~\ref{cor-from-path-to-curve}.
 \begin{remark}
   The notion of weights here differs from that in the weighted signatures introduced in \cite{olver15}. In the case of curves in  \cite{olver15}, a weight  is assigned to every subset of the signature  and it is equal to the total arclength of its preimage under the signature map. Then the  weights of the points not on horizontal axis  are zero, while, in the case of degenerate curves,  the weight of a point   $(\kappa_0,0)$  equals the total length of all generalized vertices with constant curvature $\kappa_0$ and may be non-zero. In future work, we will explore whether  incorporating these  ``atomistic'' weights into the  signature quiver construction allows us to extend the results of this paper to degenerate curves.
 \end{remark}
\begin{proposition}[Closed-curves]\label{prop-closed}
Let $\Gamma$  be a closed, non-degenerate, curve with turning number $\xi$.
Let $\Delta_{S_\Gamma}$ be the signature quiver of $\Gamma$ with edges $S_i$  equipped with multiplicities $\mu_i$ and  weights $\omega_i$, defined by (\ref{eq-weight}), $i=1,\dots,N$.
Let $W = (w)^m$ be a closed path (compatible or not) along $\Delta_{S_\Gamma}$ such that $\omega$ is not periodic, $m > 1$, and $m$ and $\xi$ are relatively prime. A path $W$ assigns  the multiplicity
   $\tilde{\mlt}_i$ to an edge $S_i$, equal to the number of times the  edge appears in the path (if  the edge $S_i$ is not included in the path, then $\tilde{\mlt_i}=0$).
   Let ${\Gamma_W}$ be the curve with the curvature function $\kappa_W$ defined by (\ref{eq-kappaW}). If

\beq \label{eq-closed-cond}\sum_{i=1}^N\tilde{\mlt}_i\omega_i=2\pi\xi,\eeq
 then  ${\Gamma_W}$ is a closed curve with turning number equal to $\xi$, symmetry index equal to $m$, and signature index equal
to $\ds{\min_{i=1,\dots,N}}\tilde{\mlt}_i$.
 
 \end{proposition}
 \begin{proof} The statements about indices follows from Propositions~\ref{prop-mlt} and \ref{prop-sym-graph}. To prove the closedness,   let $L_W$ be the minimal period of the unit speed parameterization $\gamma_W(s)$ associated with  $\kappa_W(s)$. Then $\ell=\frac{L_W}m$ is the minimal period of $\kappa_W$. By construction for $i=0,\dots, m-1$: 
\beq \label{eq-Ww}\kappa_W|_{[i\ell,(i+1)\ell]}(s)=\kappa_W|_{[0,\ell]}(s-i\ell). \eeq
and so, using (\ref{eq-wmu})   for the second to last equality and  (\ref{eq-closed-cond}) for the last equality:

\begin{align*}
  \int_0^\ell\kappa_W(s)ds&=\frac 1 m\int_0^{L_W}\kappa_W(s)ds\\ &=\frac 1 m \sum_{i=1}^N\tilde{\mlt}_i\omega_i=2\pi \frac \xi m.
\end{align*}
The closedness  and turning number results now follow from Lemma~\ref{lem-closed}.

\end{proof}
\begin{remark} In the context of Proposition~\ref{prop-closed}, if $W$ is a compatible word (i.e. $\tilde \mlt_i=\mlt_i$), then, since $\Gamma$ has turning number $\xi$, (\ref{eq-closed-cond}) is satisfied and so the conclusion of the proposition holds.
\end{remark}

We conclude this section with two collections of non-congruent non-degenerate closed curves with the same signature curve shown  in Figure~\ref{fig:CinfSig}, induced by various compatible and non-compatible paths on the signature quiver in Figure~\ref{fig:directed-graph}.
 \begin{example}
 In Figure~\ref{fig:Cinf-ex}, we show a collection of non-congruent non-degenerate closed curves with the same signature curve shown  in Figure~\ref{fig:CinfSig}, induced by various \emph{compatible} paths on the signature quiver in Figure~\ref{fig:directed-graph}. All curves in this example have signature index 6.
\end{example}

\begin{figure*}
  \centering
    \subfigure[The curve induced by the word $(cabd)^6$.]{
    \label{fig:Cinf_3}
    \includegraphics[width=5.5cm]{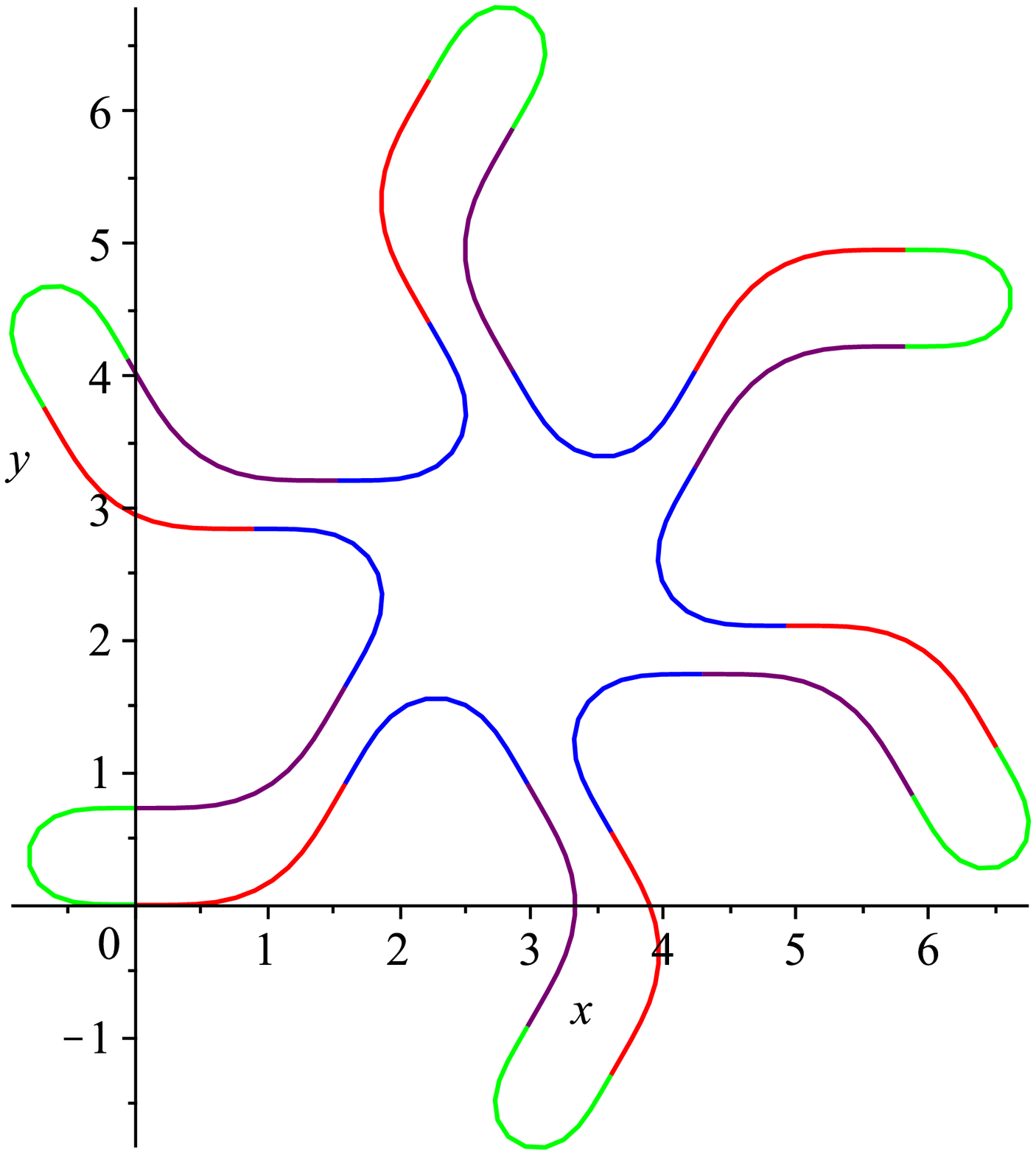}
    }
    \hspace{.5cm}
    \subfigure[The curve induced by the word $(cbdadbac)^3$.]{
    \label{fig:Cinf_sym3_2}
    \includegraphics[width=5.5cm]{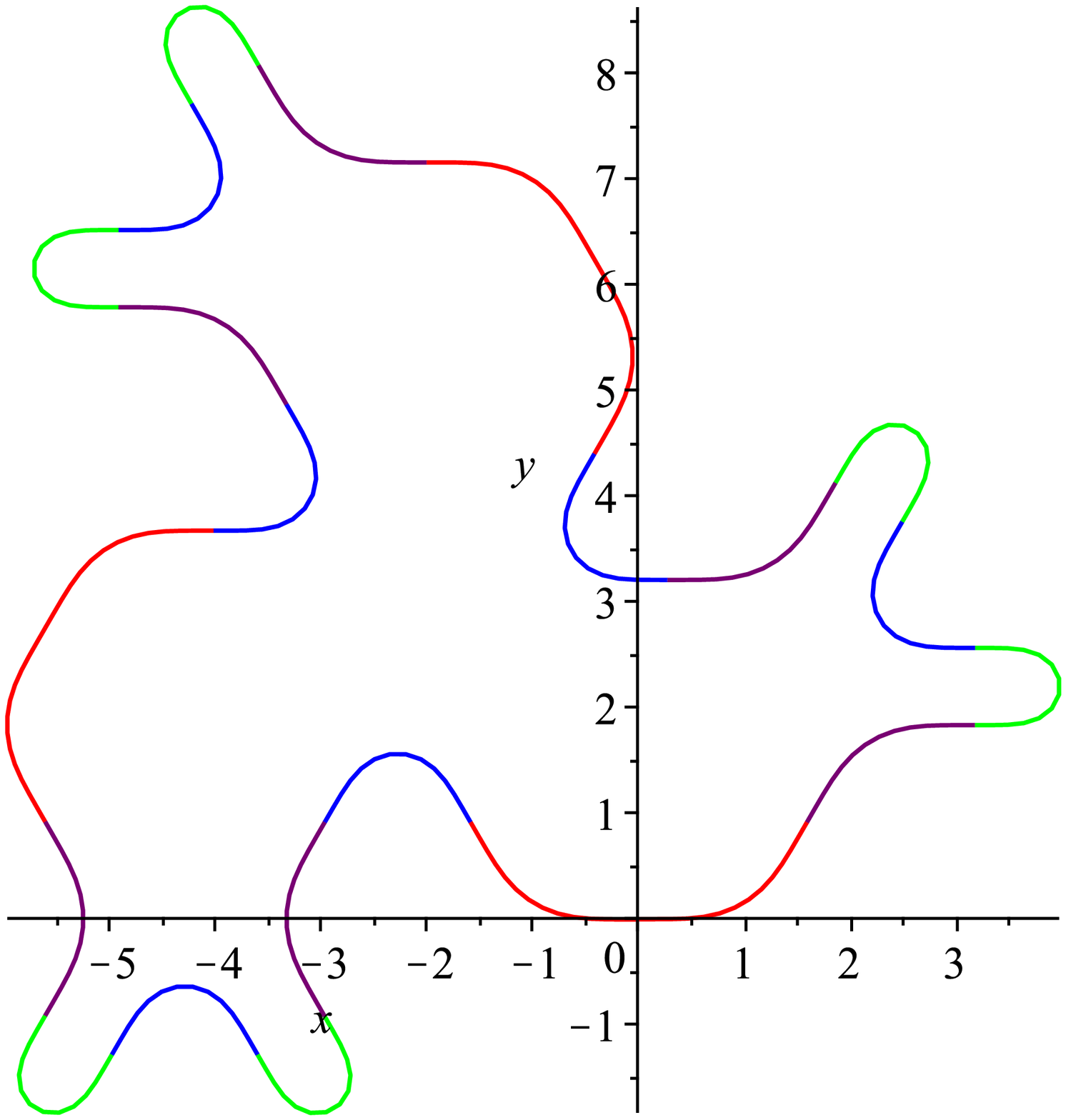}
    }
    \hspace{.5cm}
    \subfigure[The curve induced by the word $(cbdacbadcabd)^2$.]{
    \label{fig:Cinf_sym2_2}
    \includegraphics[width=5.5cm]{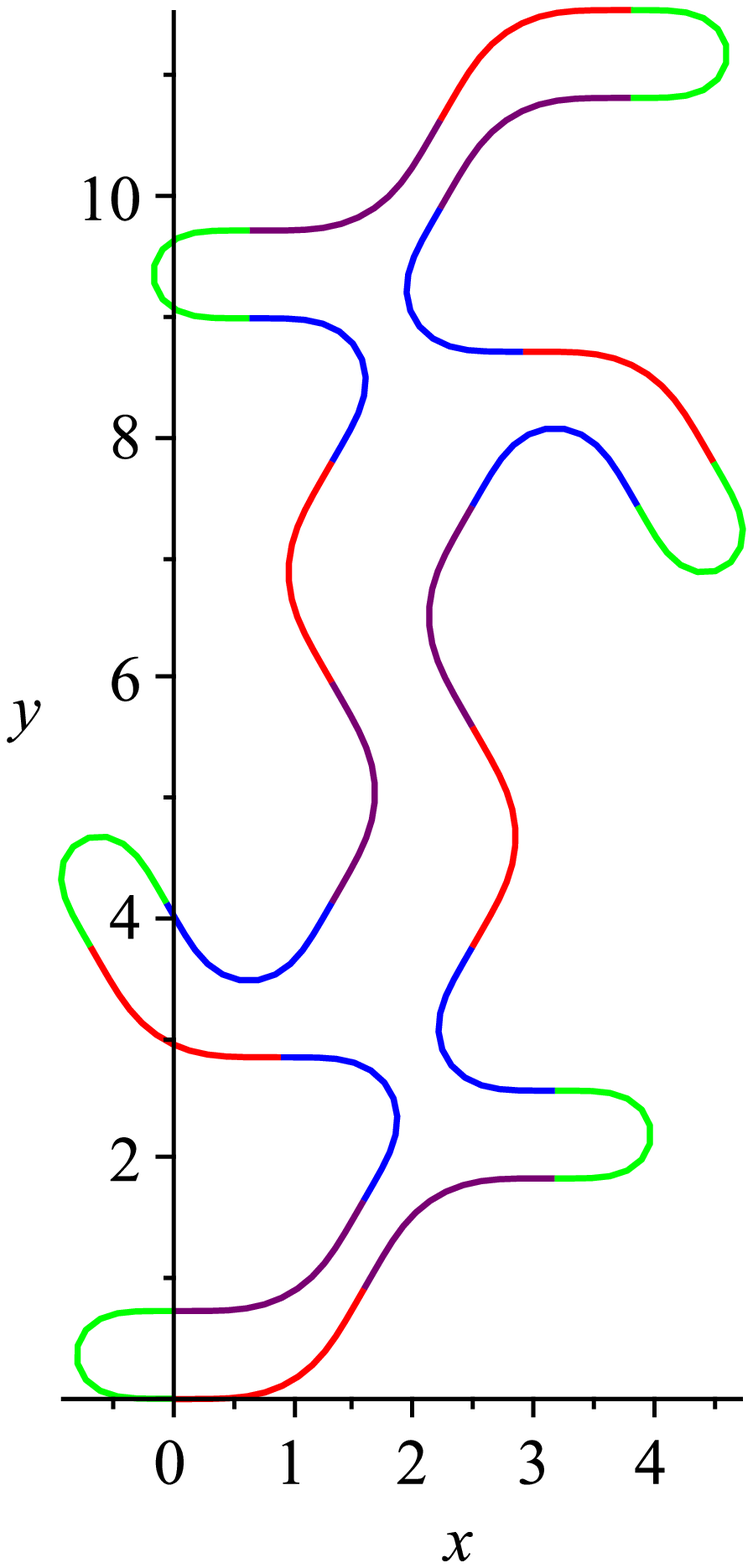}
    }
    \hspace{.5cm}
    \subfigure[The curve induced by the word $(cbdaadbc)^3$.]{
    \label{fig:Cinf_sym3_3}
    \includegraphics[width=5.5cm]{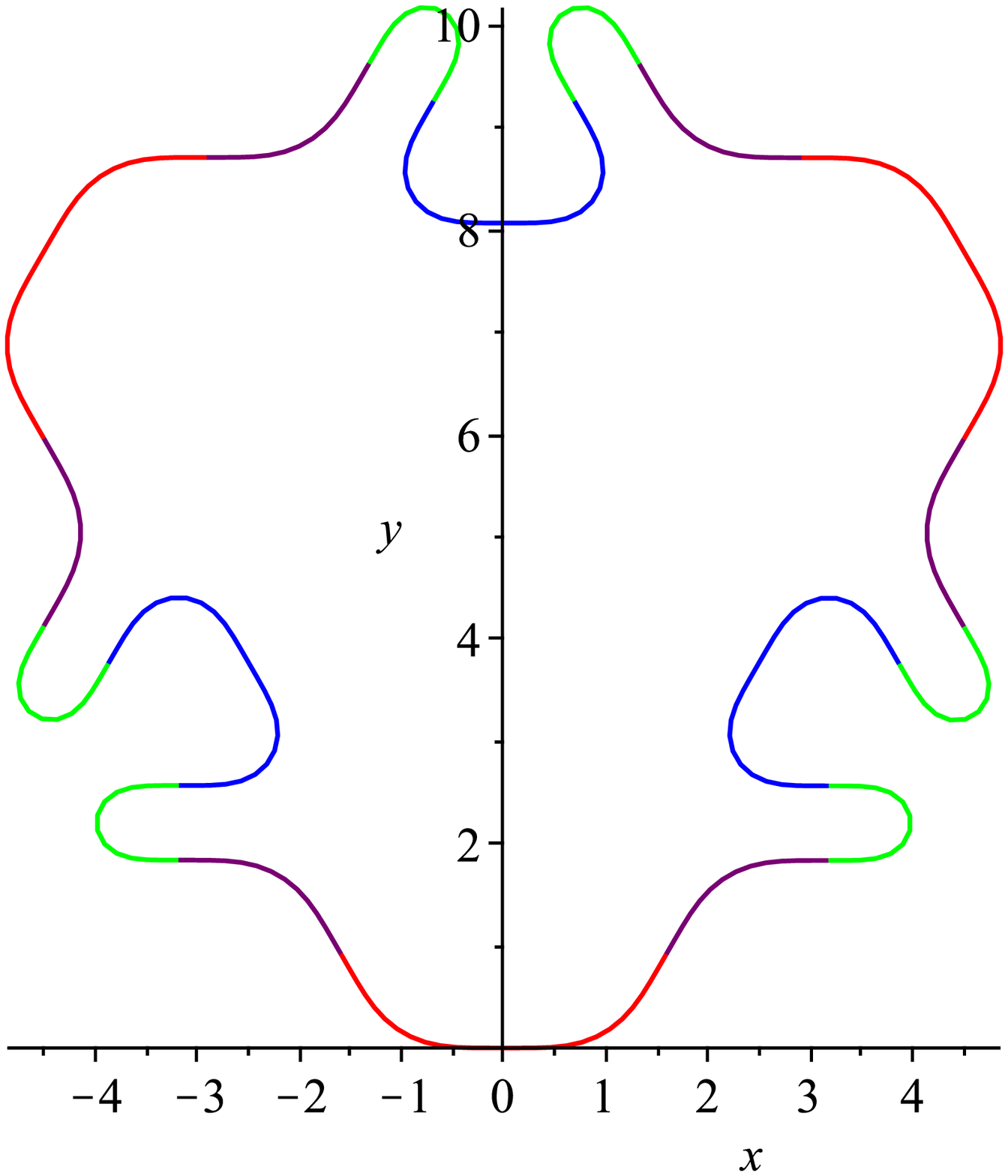}
    }
    \hspace{.5cm}
    \subfigure[The curve induced by the word $(cbdacdbacbda)^2$.]{
    \label{fig:Cinf_sym2_3}
    \includegraphics[width=5.5cm]{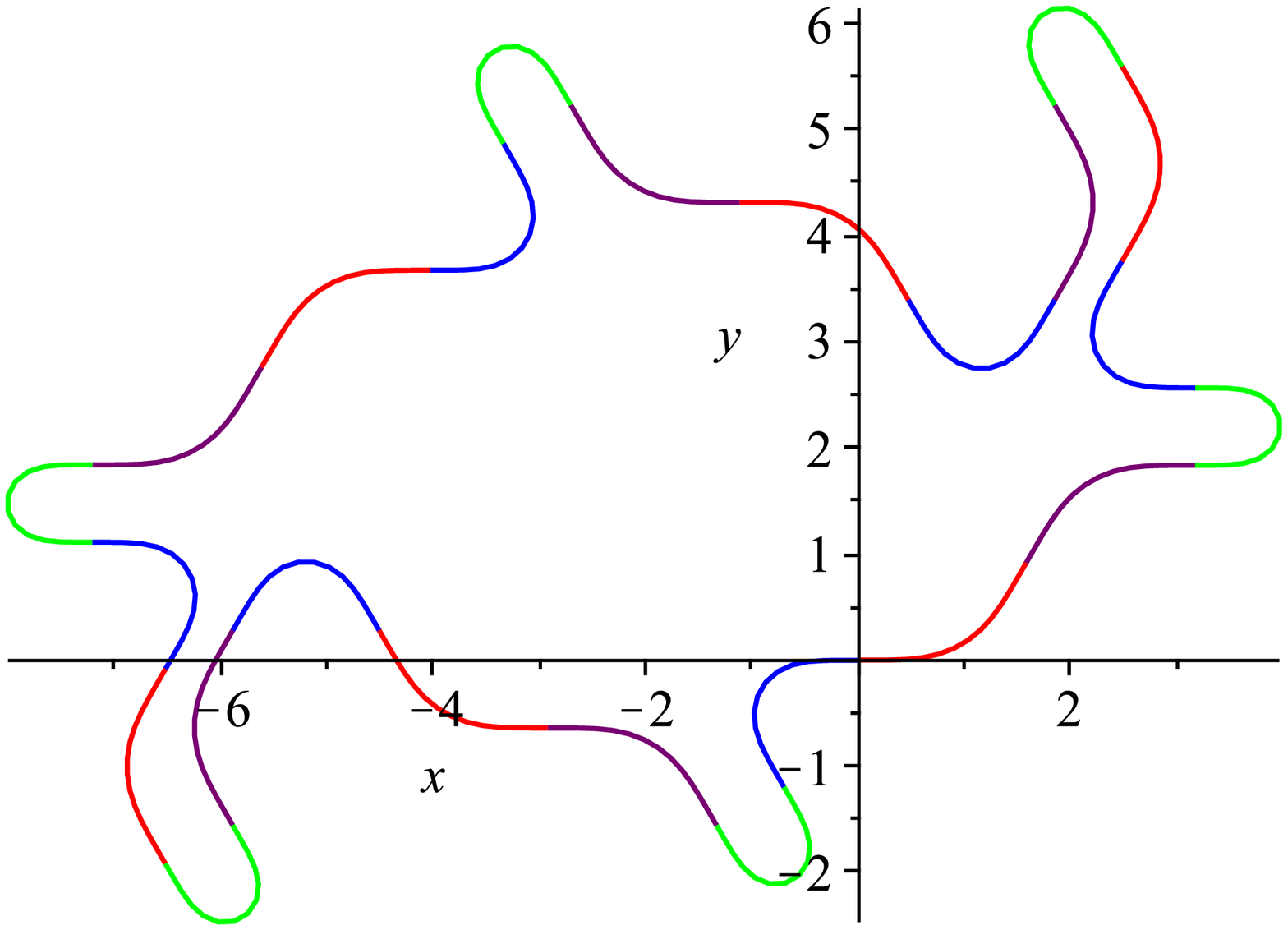}
    }
    \hspace{.5cm}
    \subfigure[The curve induced by the word $(cbddabac)^3$.
    {Note that the intersections lead to an ambiguity in parameterization and this curve is also induced by the word $(cccbddabac)^3$.}]{
    \label{fig:Cinf_notSimple_sym3_1}
    \includegraphics[width=5.5cm]{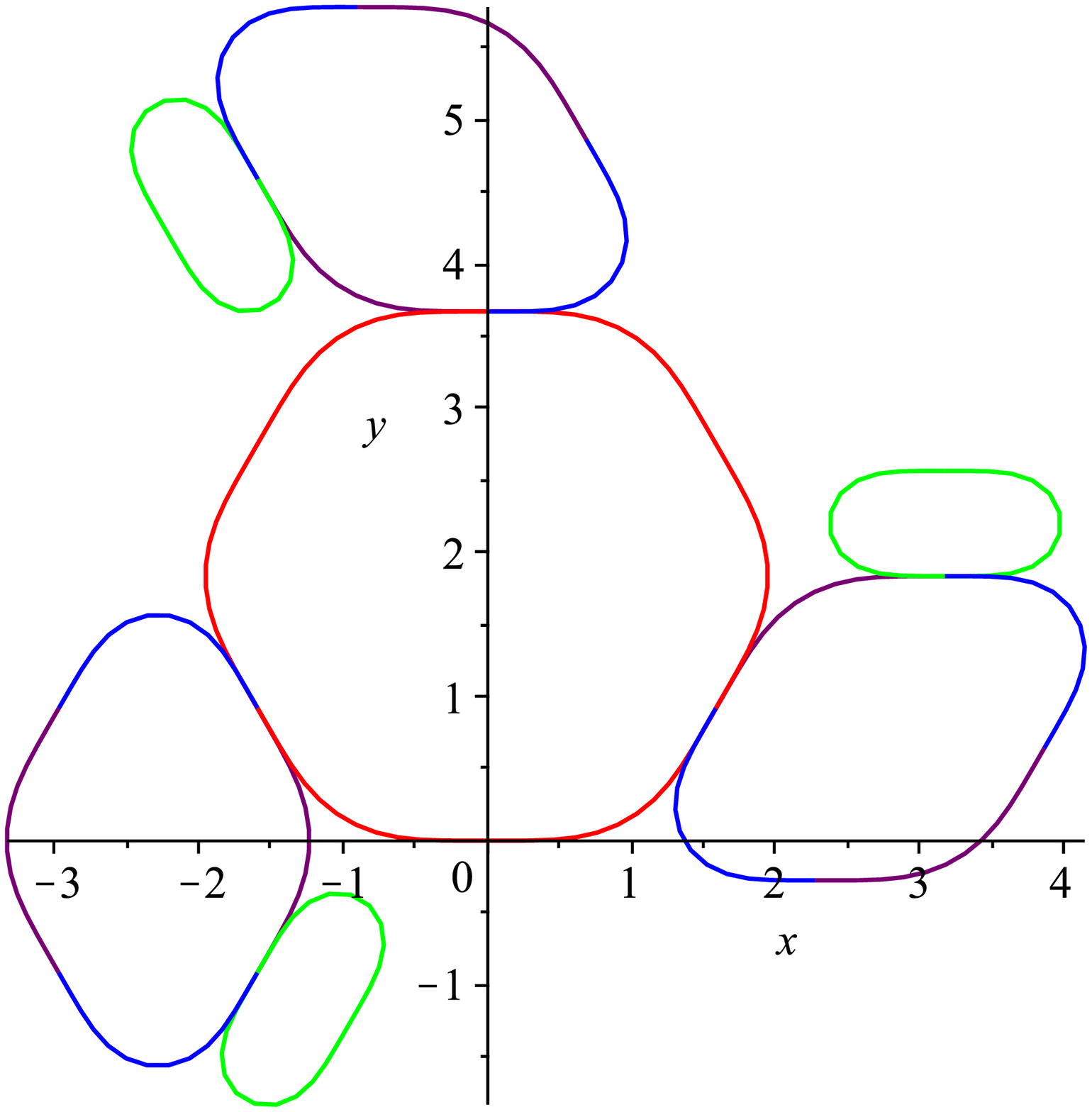}
    }
  \caption{Non-congruent curves with signature curve in Figure~\ref{fig:CinfSig} induced by compatible paths on the signature quiver in Figure~\ref{fig:directed-graph}.}
    \label{fig:Cinf-ex}
\end{figure*}


\begin{figure*}
  \centering
    \subfigure[The curve induced by the non-compatible word $(cbdacacccbda)^2$ {with signature index 4  and symmetry index 2}.]{
    \label{fig:Cinf_sym2_indexSwap_1}
    \includegraphics[width=6cm]{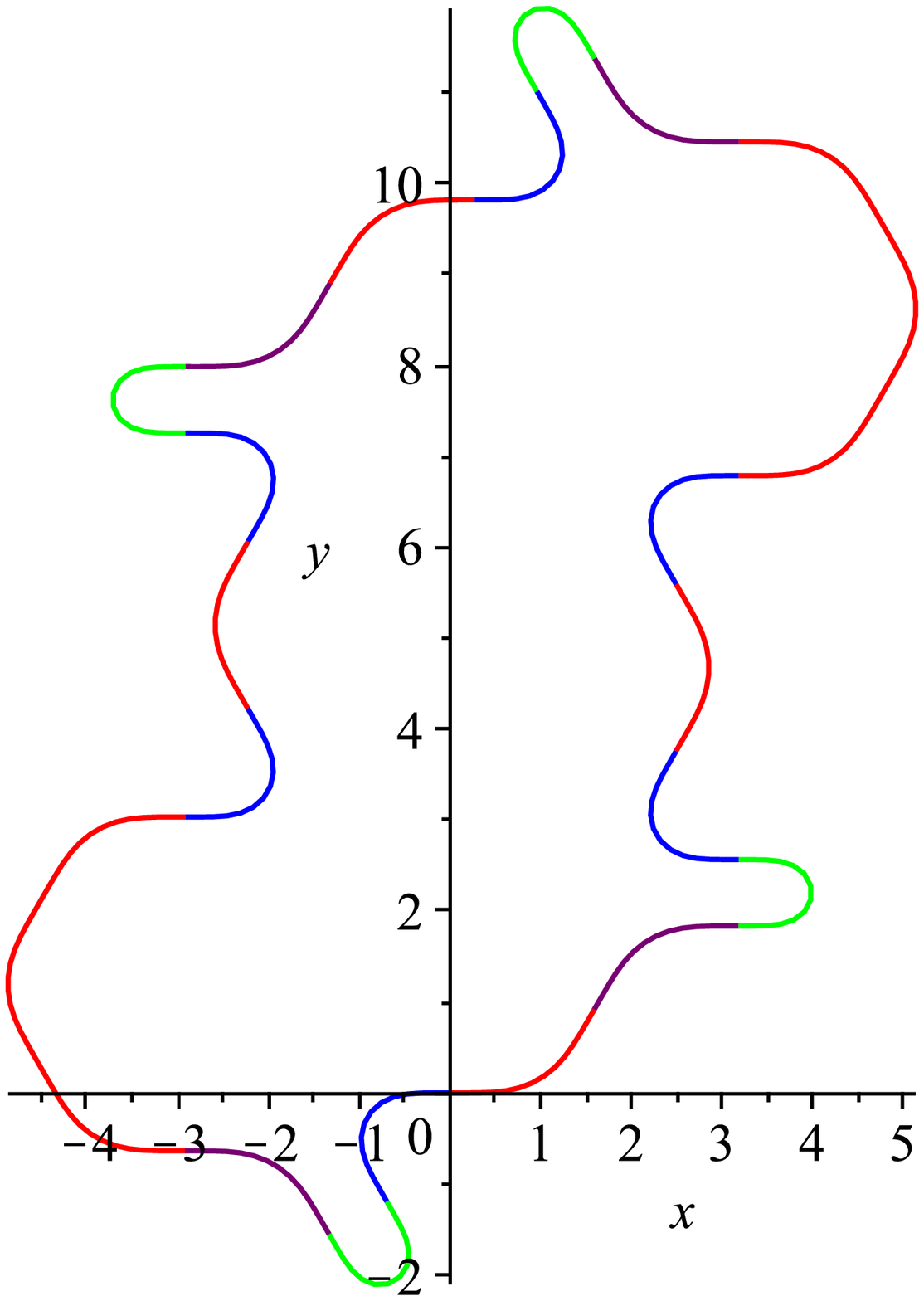}
    }
      \hspace{.5cm}
    \subfigure[The curve induced by the non-compatible word $(cbdacacc)^3$ {with signature index 3  and symmetry index 3}.]{
    \label{fig:Cinf_sym3_indexSwap_1}
    \includegraphics[width=6cm]{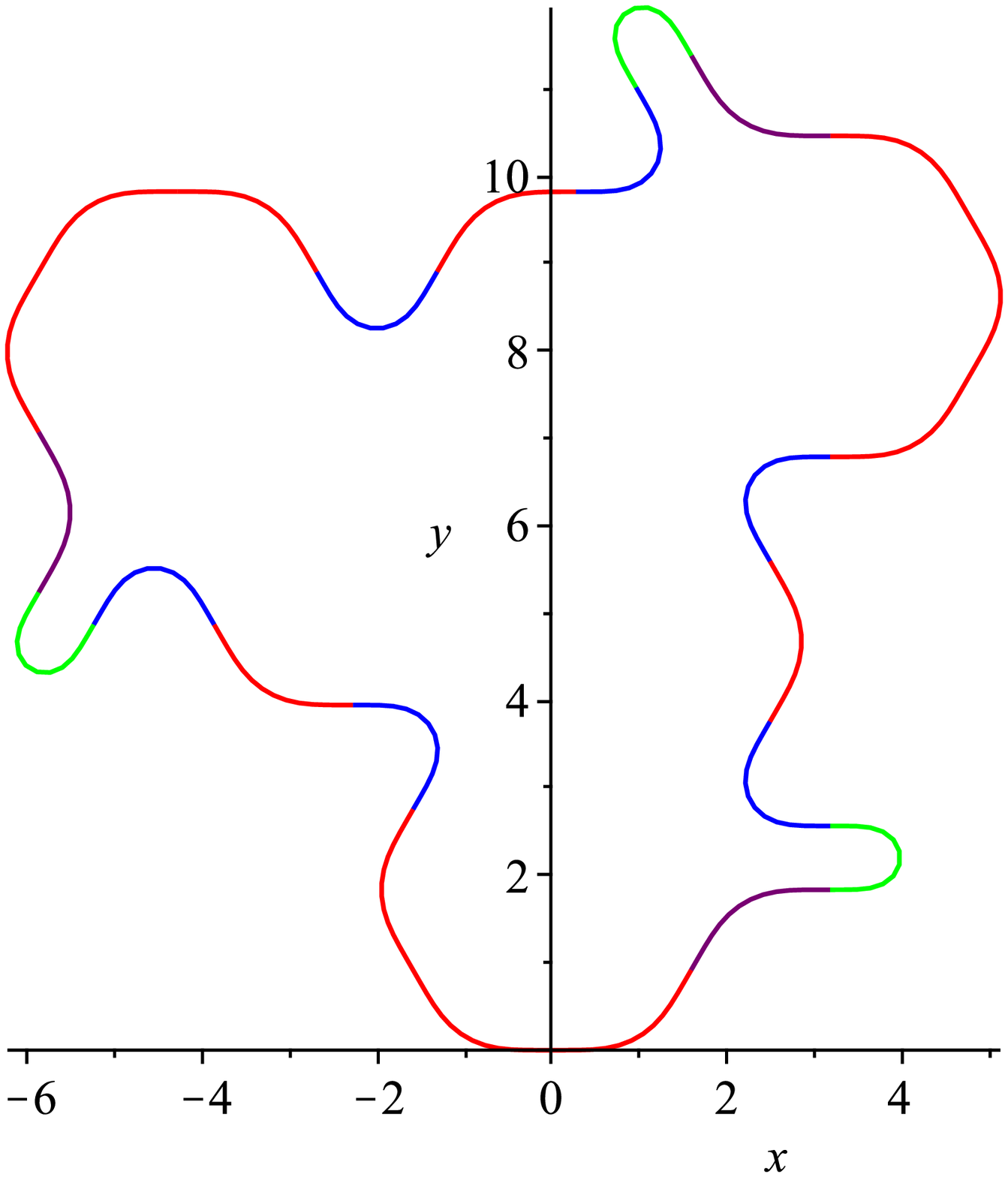}
    }
      \hspace{.5cm}
    \subfigure[The curve induced by the non-compatible word $(cacccbdacacc)^2$ {with signature index 2 and symmetry index 2}.]{
    \label{fig:Cinf_sym2_indexSwap_2}
    \includegraphics[width=6cm]{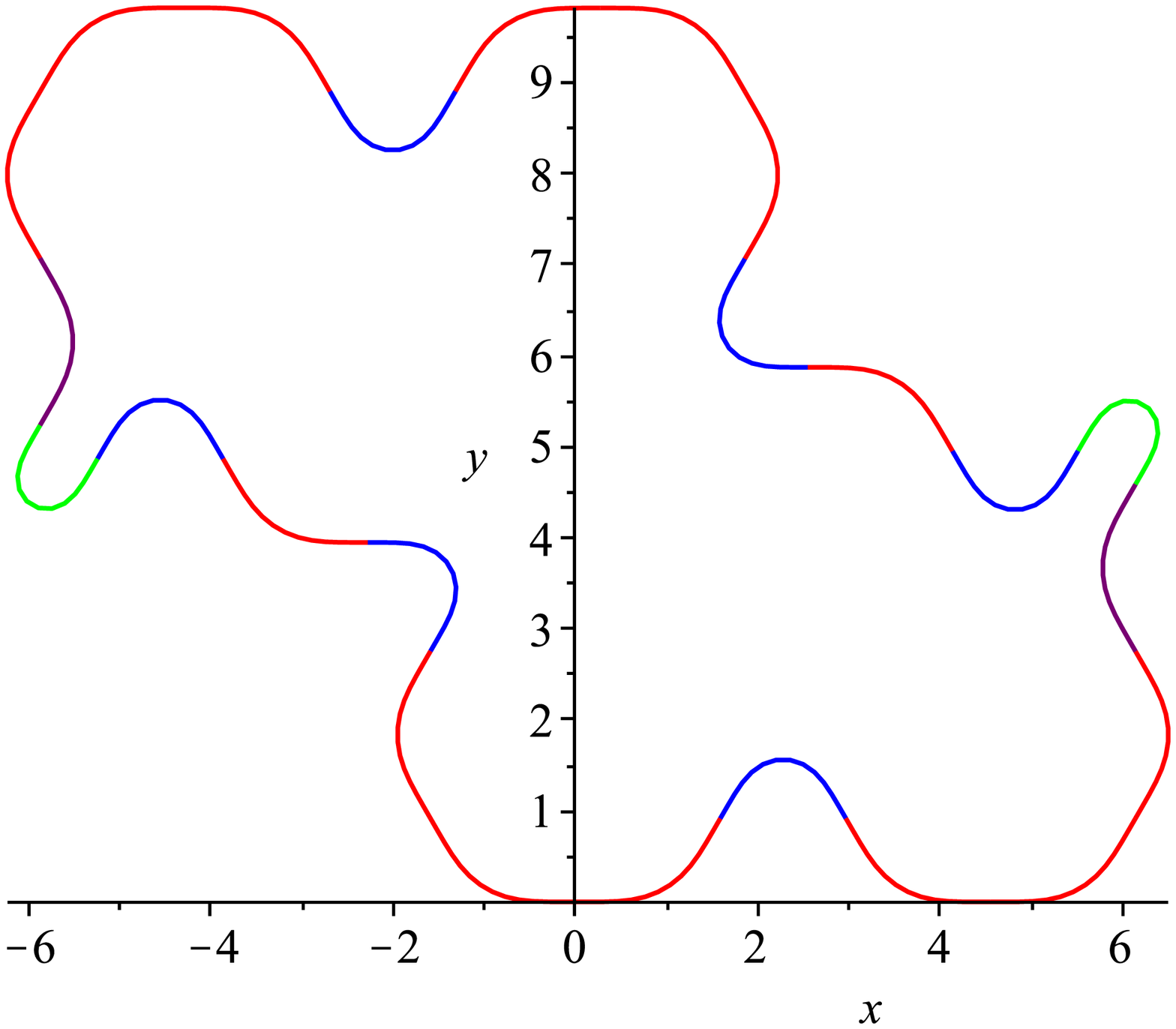}
    }
      \hspace{.5cm}
    \subfigure[The curve induced by the non-compatible word $(ccabdacc)^3$ {with signature index 3  and symmetry index 3}.]{
    \label{fig:Cinf_sym3_indexSwap_2}
    \includegraphics[width=6cm]{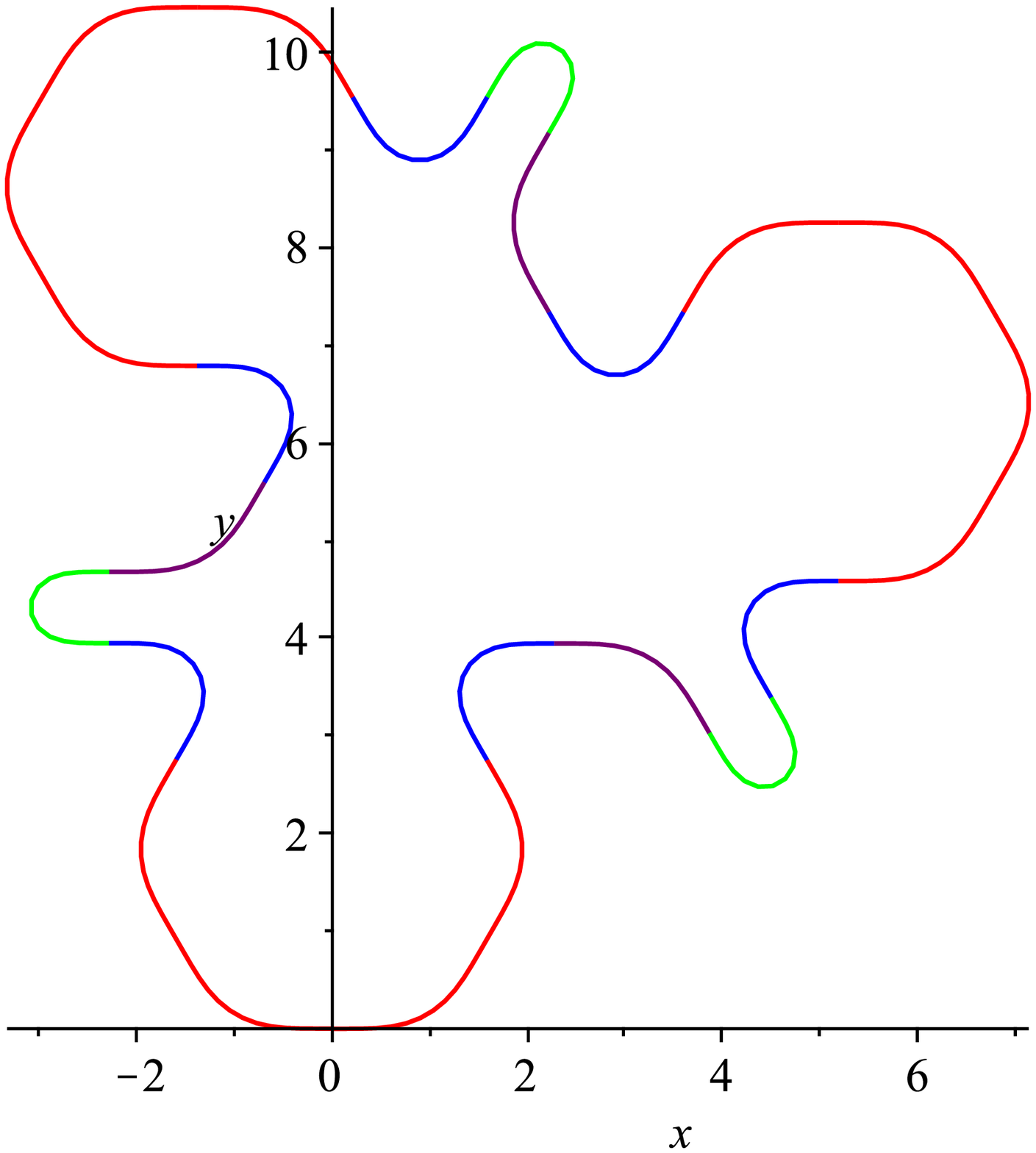}
    }
  \caption{Non-congruent curves with signature in Figure~\ref{fig:CinfSig} induced by non-compatible paths on the signature quiver in Figure~\ref{fig:directed-graph}.}
    \label{fig:Cinf-indexSwap_ex}
\end{figure*}

\begin{example}
 In Figure~\ref{fig:Cinf-indexSwap_ex}, we  show how carefully chosen \emph{non-compatible} paths along the quiver in Figure~\ref{fig:directed-graph} generate closed, non-degenerate, non-congruent curves with the signature pictured in Figure~\ref{fig:CinfSig}.
To this end, we consider a non-compatible closed periodic word $W = (w)^m$, $m>1$, such that each of the $N$ letters (edges) in the quiver appears at least once. Assume the $i$-th letter appears $\tilde{\mlt}_i$ times in $W$ and that the corresponding edge has weight $\omega_i$ (see (\ref{eq-wmu})).
If $W$ is chosen so that $\ds{\sum_{i=1}^{N}} \tilde{\mlt}_i\omega_i = 2\pi$ then by Proposition~\ref{prop-closed}, the curve $ \Gamma_W$ induced by $W$ is closed with the symmetry-index $m$.
Note that in this example, edge $a$ (blue) has weight $-\frac{2\pi}3$, edge $b$ (purple) has weight $-\frac\pi3$, edge $c$ (red) has weight $\frac\pi3$, and edge $d$ (green) has weight $\pi$.
\end{example}


\section{More examples}\label{sect-ex}


In this section, we revisit examples appearing in Sections 3 and 5 of \cite{Musso2009}. In Section~\ref{sect-ex-musso}, we use the example in Section 3 of \cite{Musso2009} and, considering various compatible paths along their signature quiver, construct \emph{non-degenerate} curves with identical signatures. In Section~\ref{sect-cog}, following general ideas in Section 5 of \cite{Musso2009}, we construct a different set of cogwheels, based on trigonometric functions.  In contrast with our previous examples, the cogwheels here have trivial global symmetry groups.
\subsection{Curves with different order of smoothness}\label{sect-ex-musso}

\begin{figure}
  \centering
  \subfigure[Curvature $\kappa_1$.]{
    \label{fig:k1}
    \includegraphics[width=6cm]{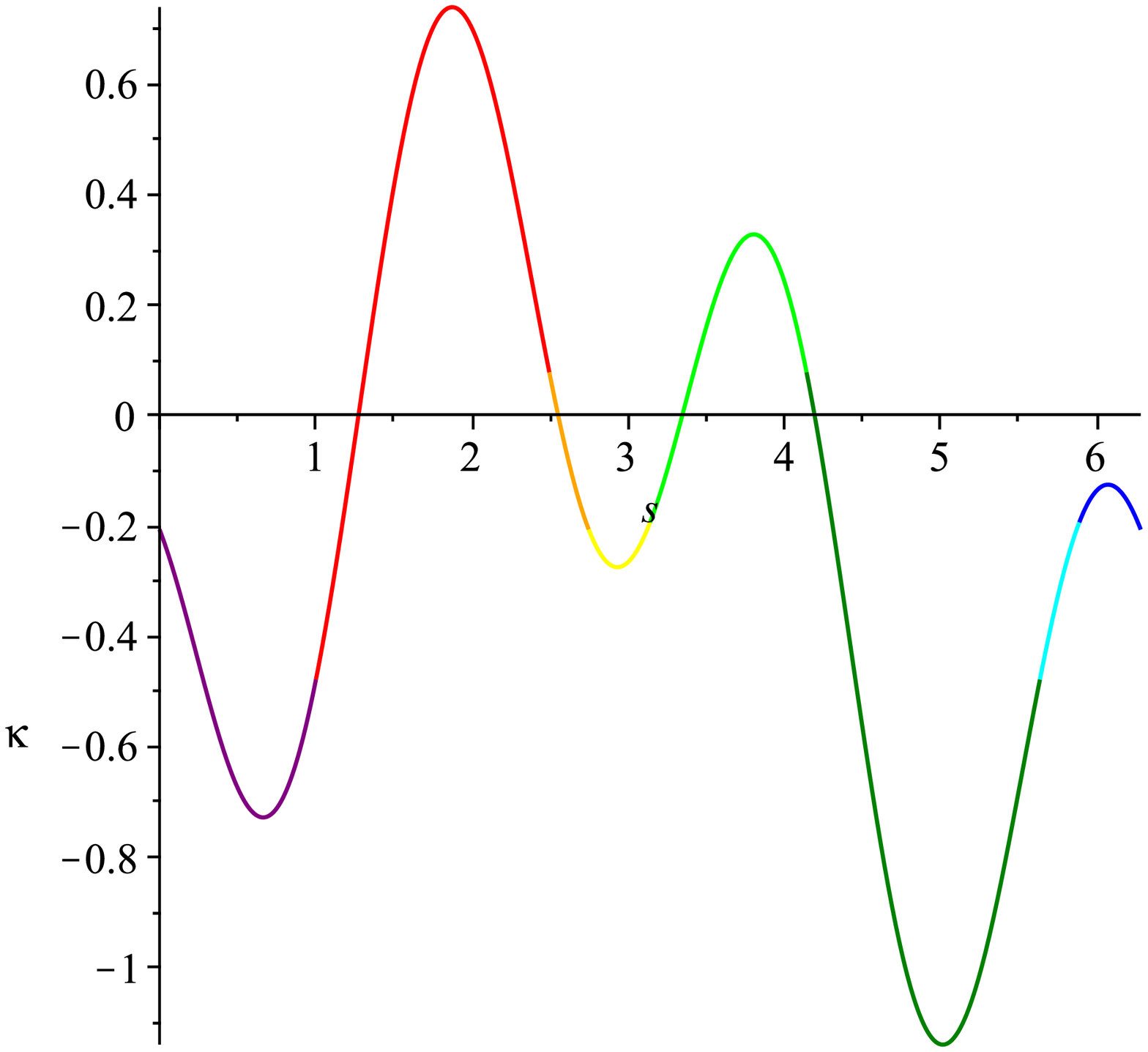}
  }
    \hspace{.1cm}
  \subfigure[Curve $\Gamma_1$, corresponding to parameterization $\gamma_1$.]{
    \label{fig:MNc1}
    \includegraphics[width=6cm]{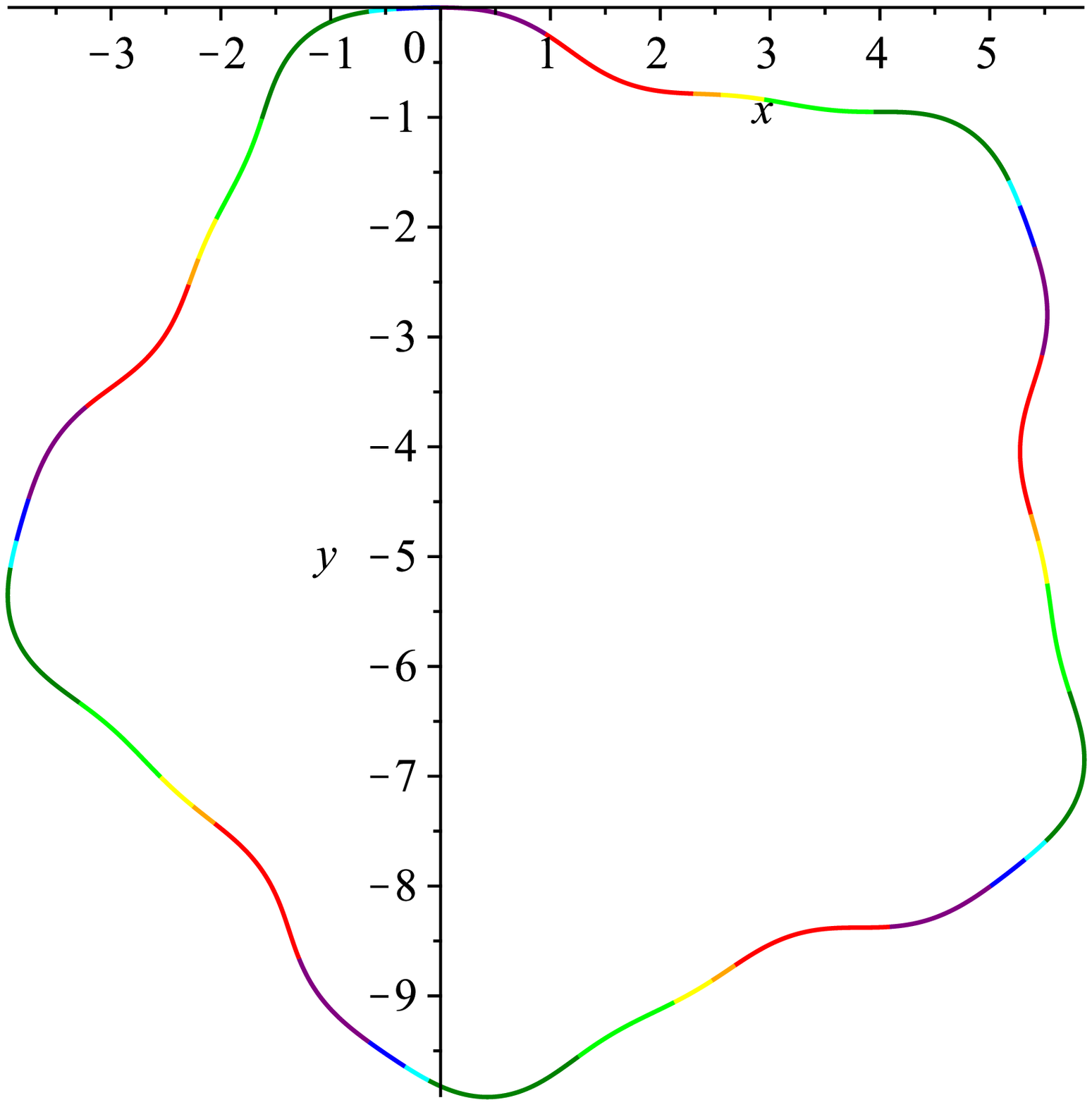}
  }
  \caption{Curvature function (\ref{eq:k1}) and a corresponding curve.}
\end{figure}

\begin{figure}
  \subfigure[The Euclidean Signature of $\Gamma_1$.]{
    \label{fig:MNSig}
    \includegraphics[width=6cm]{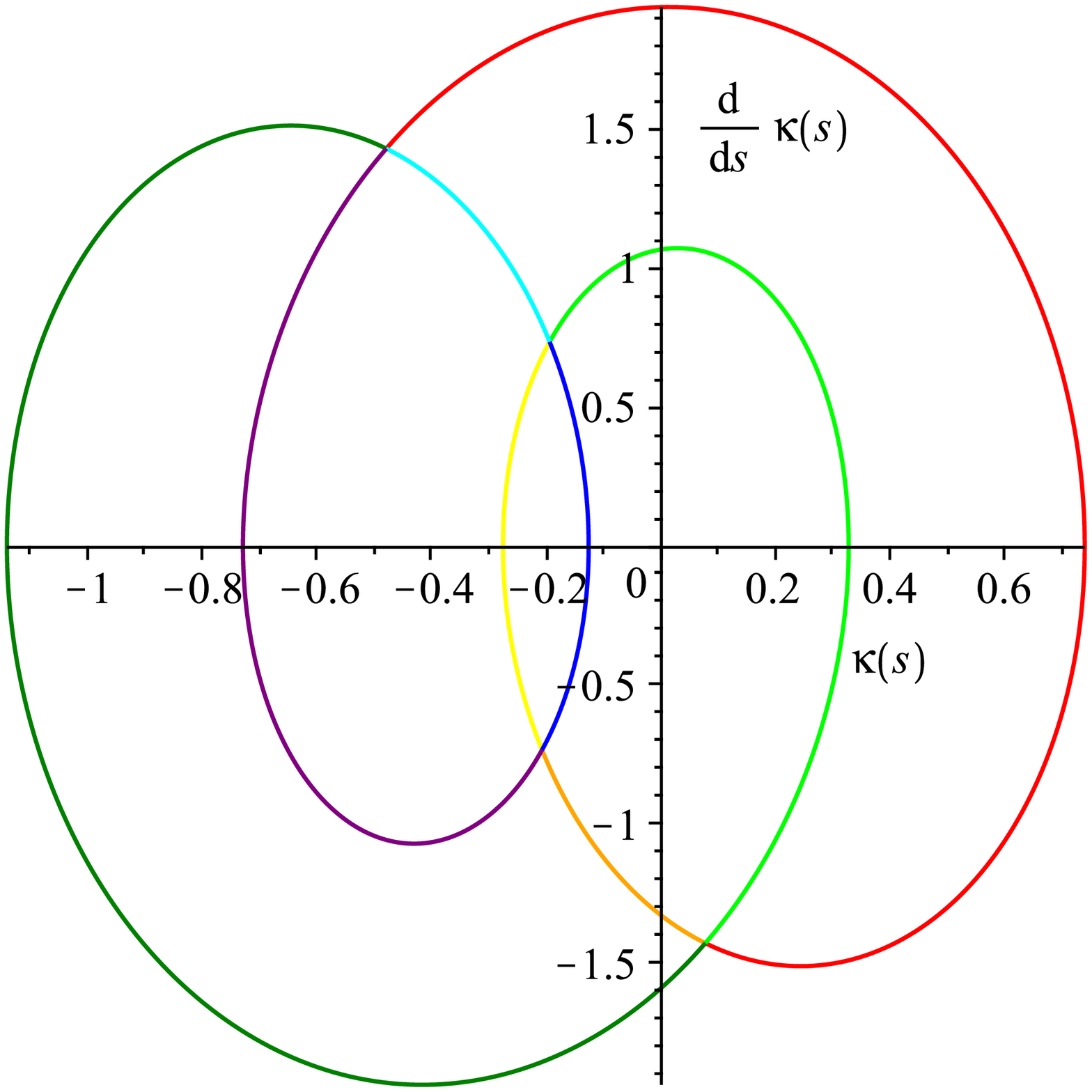}
  }
    \hspace{.1cm}
    \subfigure[The corresponding quiver. 
      The superscript on each of the labeled edges denotes the multiplicity of the paths induced by Figures~\ref{fig:MNc1}.]
{
\label{fig:MNgraph}
\begin{tikzpicture}[
            > = stealth, 
            shorten > = 1pt, 
            auto,
	    node distance = 2.3cm, 
            semithick 
        ]

        \tikzstyle{every state}=[
            draw = black,
            thick,
            fill = white,
            minimum size = 4mm
        ] 

	\node[shape=circle,draw=black] (1) {};
        \node[shape=circle,draw=black] (2) [below right of=1] {};
        \node[shape=circle,draw=black] (3) [below of=2] {};
        \node[shape=circle,draw=black] (4) [below right of=3] {};

	\path[->] (1) edge [bend left=70, distance=3.5cm, draw=red, thick] node {$h^5$} (4);
	\path[->] (4) edge [bend left=70, distance=3.5cm, draw=myGreen, thick] node {$a^5$} (1);
	\path[->] (2) edge [bend left=30, draw=green, thick] node {$g^5$} (4);
        \path[->] (3) edge [bend left=30, draw=violet, thick] node {$b^5$} (1);
        \path[->] (2) edge [bend left=30, draw=blue, thick] node {$e^5$} (3);
        \path[->] (3) edge [bend left=30, draw=yellow, thick] node {$d^5$} (2);
        \path[->] (1) edge [bend left=20, draw=myCyan, thick] node {$c^5$} (2);
        \path[->] (4) edge [bend left=20, draw=orange, thick] node {$f^5$} (3);

    \end{tikzpicture}
  }
  \caption{The signature and quiver for the curve in Figure~\ref{fig:MNc1}.} 
\end{figure}

Here we construct five non-congruent curves with identical signatures, lengths, signature index, and symmetry groups.
As in \cite{Musso2009}, only one of the curves is $C^\infty$-smooth  and, in fact, analytic, while the rest are only $C^3$-smooth. 
Note that the signature appearing in Section~\ref{sect-smooth} (see Figure~\ref{fig:CinfSig}) had a single point of self-intersection at the origin, and so the corresponding points on curves constructed in that section are curve-vertices (points where $\dot\kappa=0$). In contrast, the signature appearing in this section (see Figure~\ref{fig:MNSig}) has four self-intersection points all off the horizontal axis, and so these points do not correspond to the vertices of the curves.

Using Example 1 of \cite{Musso2009} on page 73 in Sect. 3.3,
we take the periodic function
\begin{equation}
  \label{eq:k1}
\kappa_1(s) = \frac{1}{2}(\sin(s)-\cos(3s)) - \frac{1}{5}
\end{equation}
to be the curvature function of $\gamma_1$
with minimal period $2\pi$
and examine the signature curve and signature quiver
in Figure \ref{fig:MNSig}.

\begin{figure*}
  \centering
  \subfigure[Curvature $\kappa_2$ induced by word $(bhacedgf)^5$.]{
    \label{fig:k2}
    \includegraphics[width=4cm]{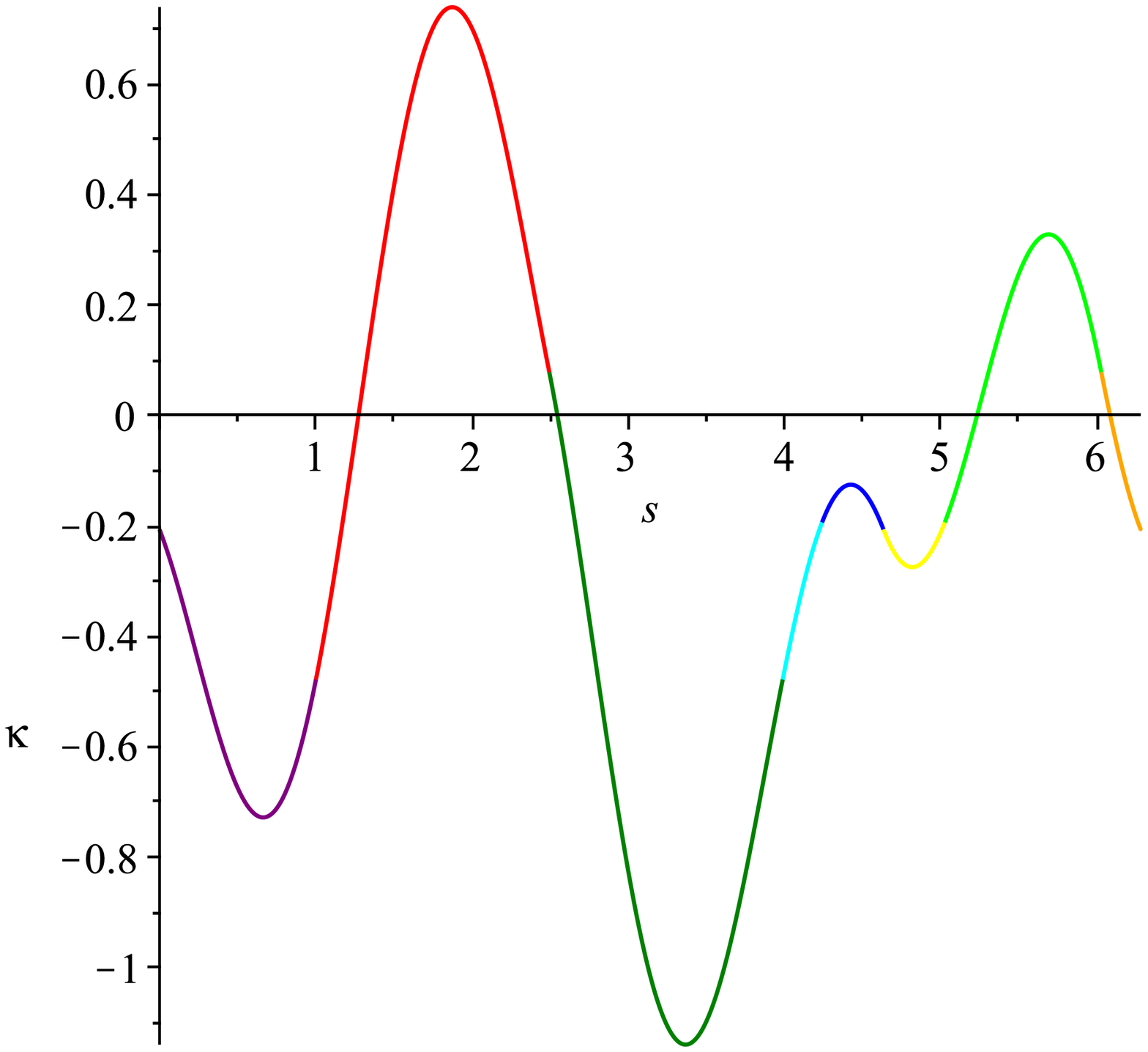}
  }
  \hspace{1cm}
  \subfigure[Curve $\Gamma_2$.]{
    \label{fig:MNc2}
    \includegraphics[width=4cm]{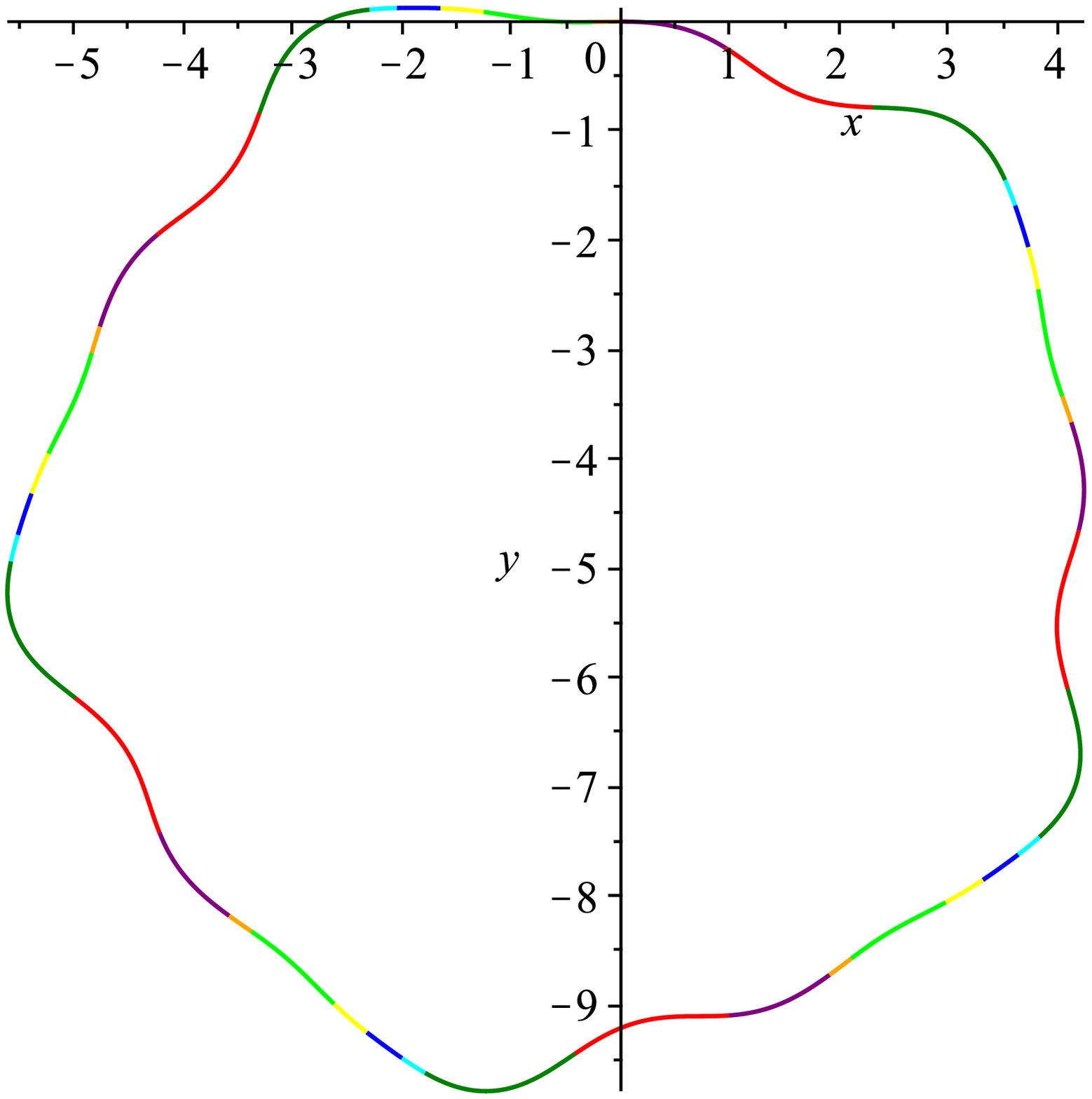}
  }\\
  \subfigure[Curvature $\kappa_3$ induced by word $(bhfcgfde)^5$.]{
    \label{fig:k3}
    \includegraphics[width=4cm]{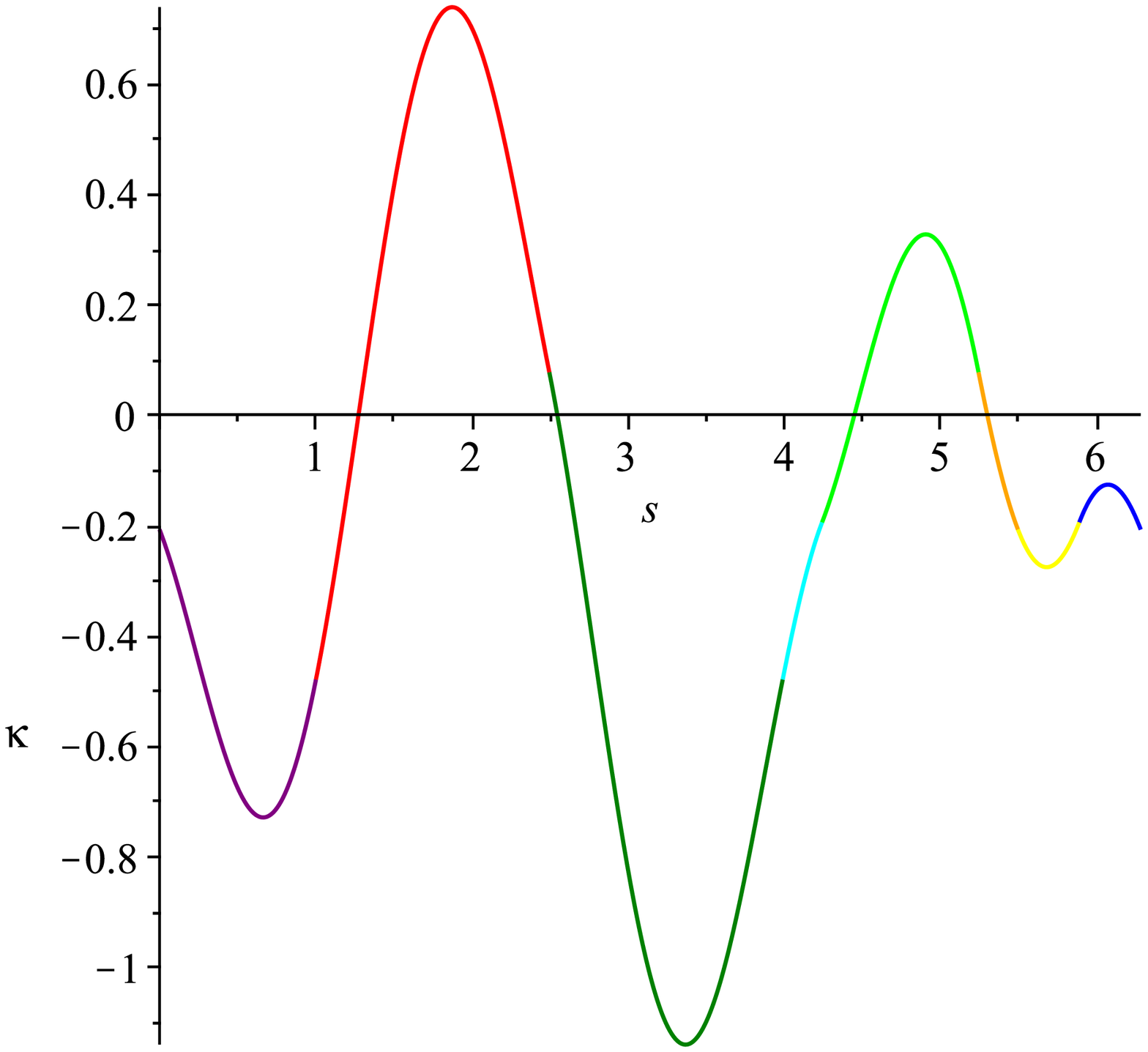}
  }
  \hspace{1cm}
  \subfigure[Curve $\Gamma_3$.]{
    \label{fig:MNc3}
    \includegraphics[width=4cm]{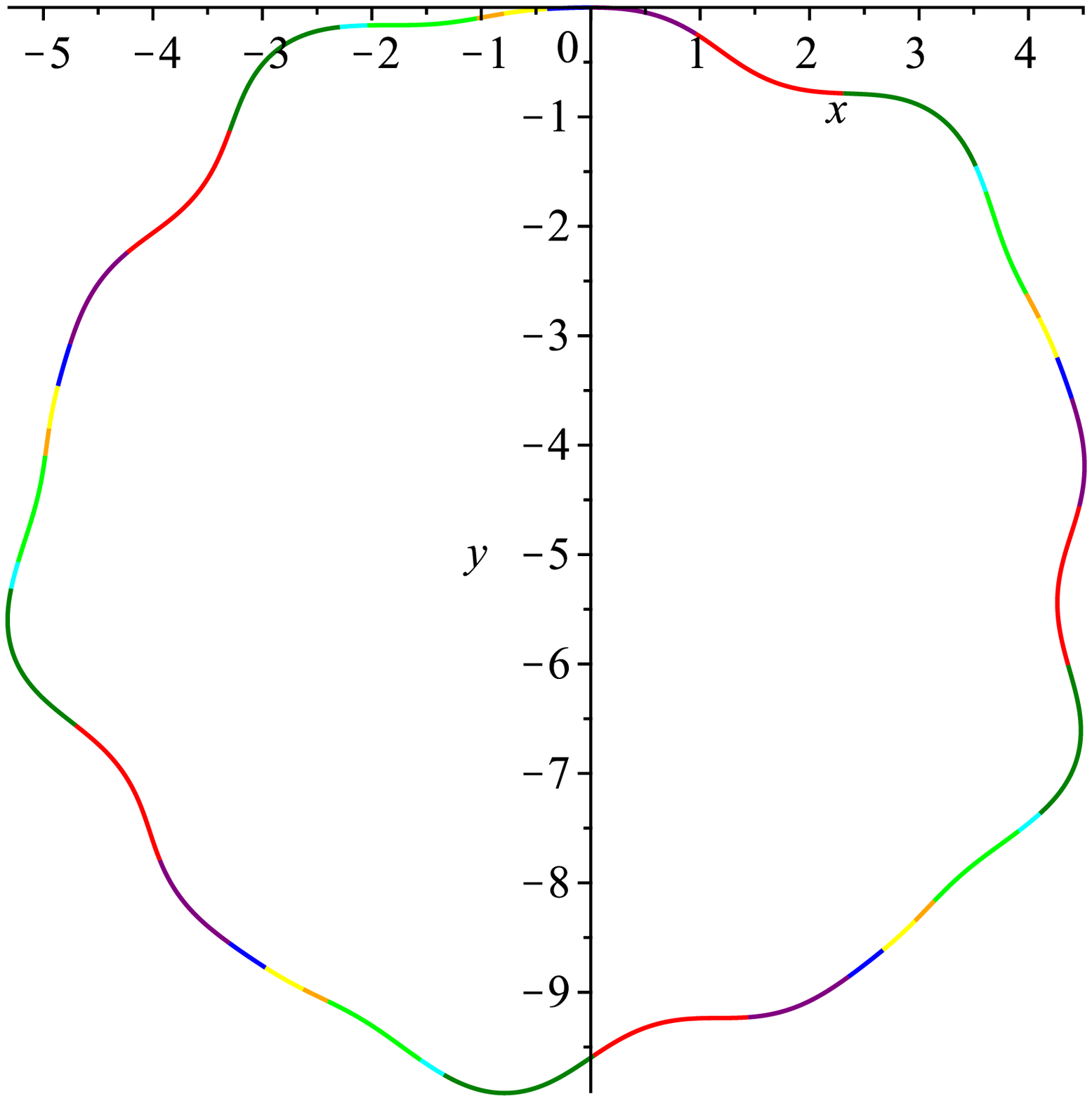}
  }\\
  \subfigure[Curvature $\kappa_4$ induced by word $(bcedgfahf)^5$.]{
    \label{fig:k4}
    \includegraphics[width=4cm]{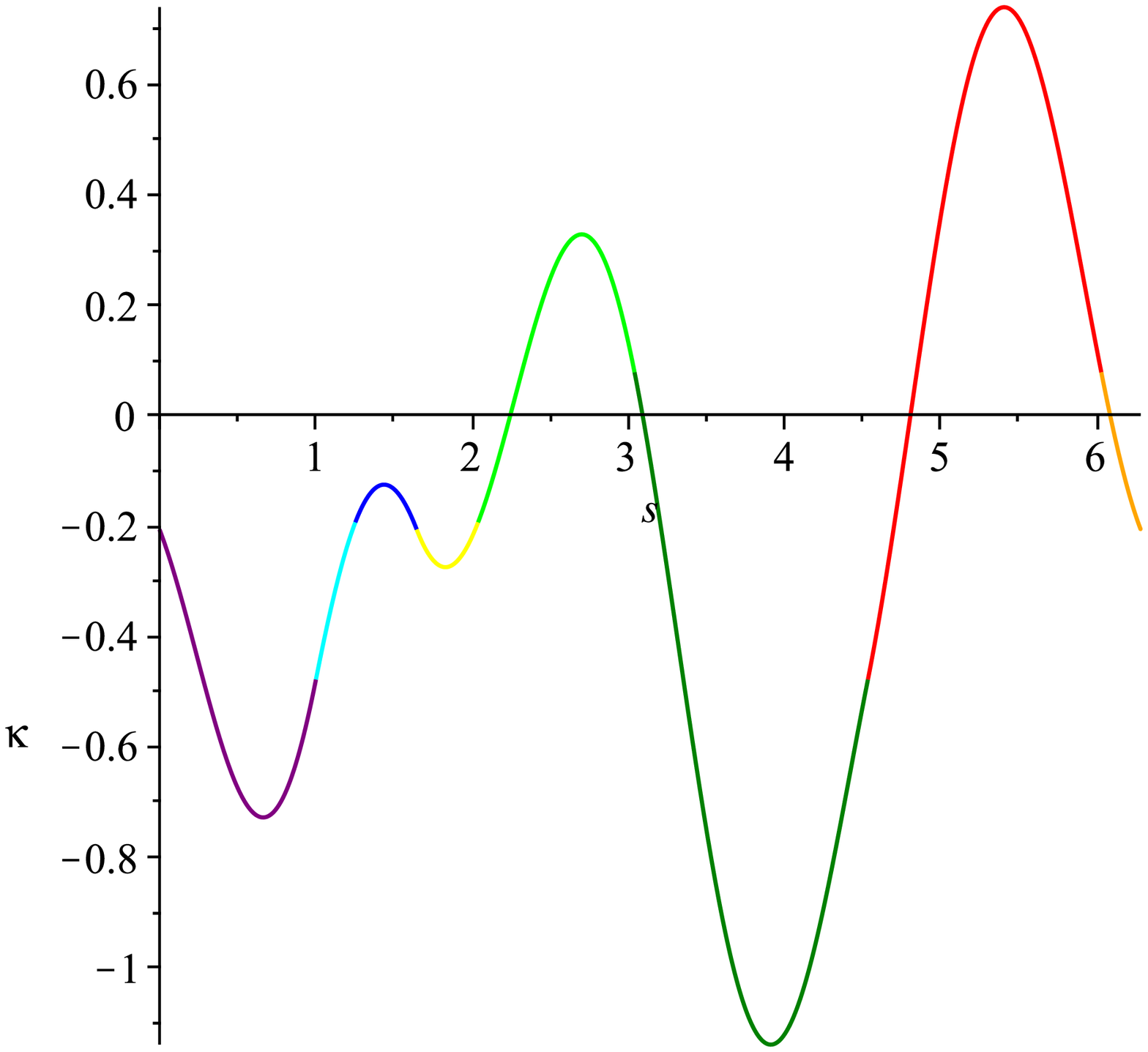}
  }
  \hspace{1cm}
  \subfigure[Curve $\Gamma_4$.]{
    \label{fig:MNc4}
    \includegraphics[width=4cm]{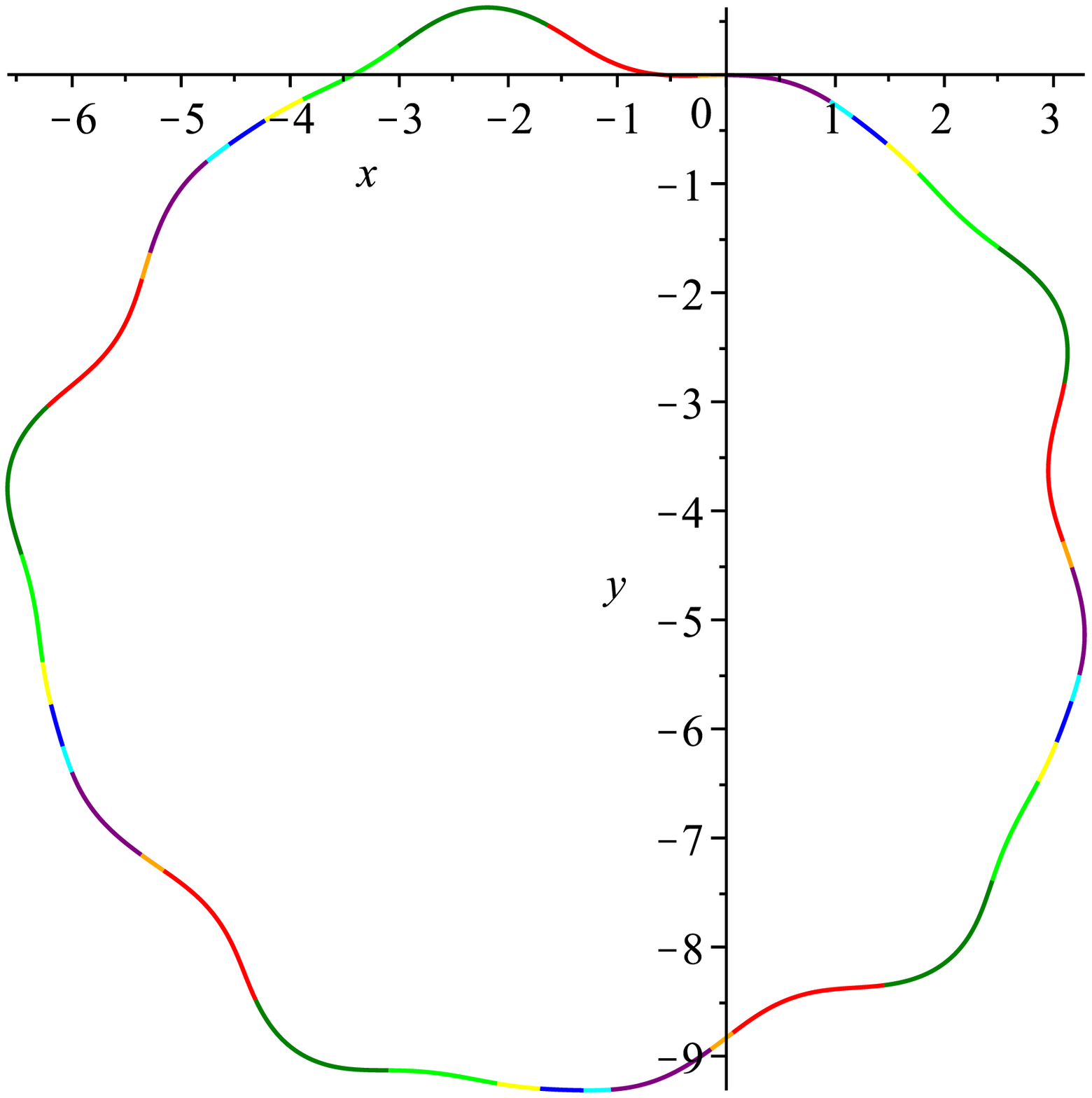}
  }\\
  \subfigure[Curvature $\kappa_5$ induced by word $(bcgahfde)^5$.]{
    \label{fig:k5}
    \includegraphics[width=4cm]{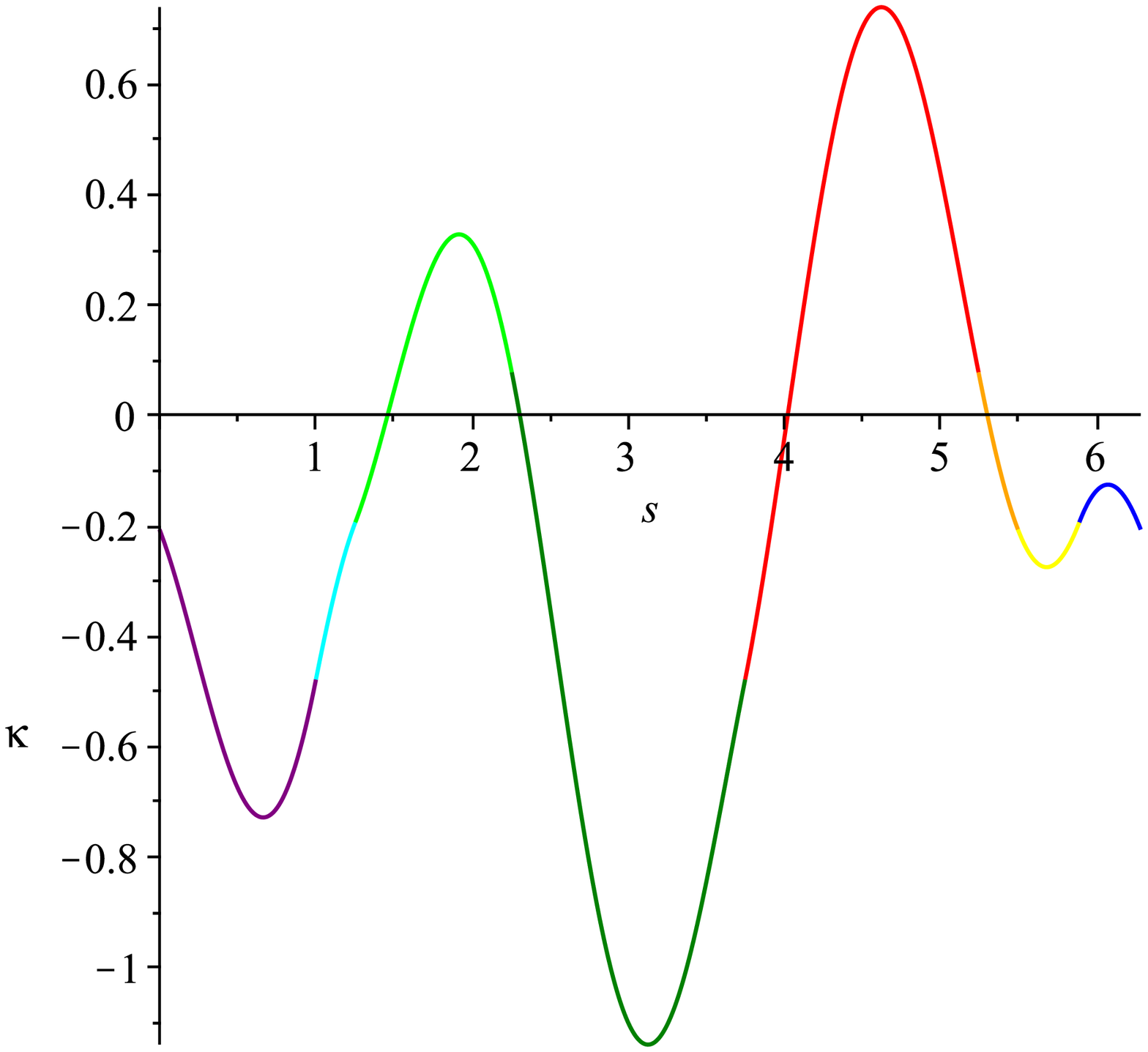}
  }
  \hspace{1cm}
  \subfigure[Curve $\Gamma_5$.]{
    \label{fig:MNc5}
    \includegraphics[width=4cm]{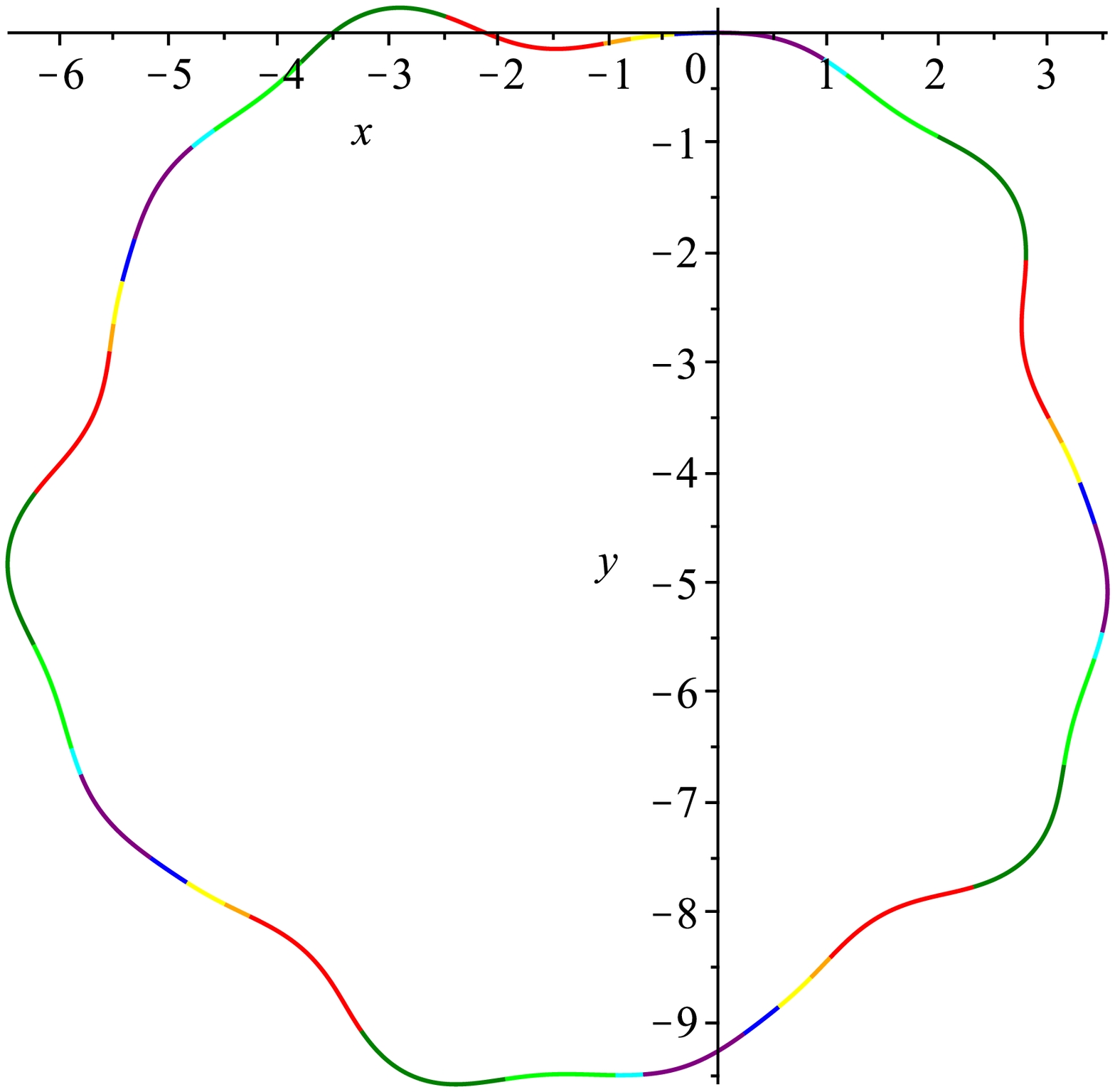}
  }
  \caption{Curves and curvatures with signature in Figure~\ref{fig:MNSig}.}
\end{figure*}

Our original  curve  $\Gamma_1$ pictured in Figure~\ref{fig:MNc1} induces the word $(bhfdgace)^5$ on the quiver in Figure~\ref{fig:MNc1}.
Since the multiplicity on each edge is 5, the only option for compatible \emph{periodic} words is to have  symmetry index 5.
Through an analysis of the quiver, we find that there  are exactly 4 other such words, distinct up to a cyclic permutation, inducing another 4 non-congruent curves with the signature  pictured in Figure~\ref{fig:MNgraph}. All five curves constructed in this section have the symmetry and signature index equal to 5.

An animation (Online Resource 1) explicitly shows how each curve piece of $\Gamma_1$ in Figure~\ref{fig:MNc1} is locally congruent to a corresponding curve piece of $\Gamma_2$ in Figure~\ref{fig:MNc2}. We can also see the difference in the parameterized signature maps of these two curves and how they trace out their shared signature in a way that directly reflects the different paths they induce on the signature quiver in another animation (Online Resource 2).

\subsection{Non-degenerate Cogwheels}\label{sect-cog}

Section 5 of \cite{Musso2009},  introduces {\it $n$-cogwheels} to 
construct simple closed curves with identical signatures
and signature index, but with different symmetry groups.
Here we introduce non-degenerate cogwheels which are constructed to have a simpler signature quiver than the ones arising in \cite{Musso2009}. Also note that while the cogwheels in \cite{Musso2009} were built using error functions (which can be replaced with smooth bump functions), cogwheels appearing here are built using a trigonometric function.
It is also worth noticing that the global symmetry group for every cogwheel pictured in this section is trivial.
The cogwheel pictured in  Figure~\ref{fig:Cog_indexSwap_1}  has signature index 2, while the rest have signature index 3.  The induced multiplicities of the edges of the signature quiver are the same for all these cogwheels except  for the cogwheels pictured in Figures~\ref{fig:Cog_indexSwap_1} and \ref{fig:Cog_indexSwap_2}.  Moreover, for any two cogwheels shown in Figures~\ref{fig:cogEx}, \ref{fig:cogwheel2}, and \ref{fig:cogwheel3} one can define a bijection, such that the corresponding points have the  same  local symmetry and signature indices. 

The construction starts with a choice of  $n \in \mathbb{Z}_+$ and partition of the interval $[0,2\pi]$ into $n$-sub-intervals
\[
  I_j = \bigg[ \frac{2(j-1)\pi}{n}, \frac{2j\pi}{n} \bigg],\, j=1,\dots,n.
\]

For each $j \in {1,\dots, n}$ 
prescribe a positive constant $a_j \in \mathbb{Z}^+$,
and consider the function
$r_j: \mathbb{R} \to \mathbb{R}_{\geq 0}$
with support in $I_j$.

\[
  r_j(t) = \begin{cases}
    a_j^{-2}(1-\cos(n a_j t)) & t \in I_j \\
    0 & t \notin I_j
  \end{cases}
\]

Let $r_0$ be a positive constant and let $\rho: \mathbb{R} \to \mathbb{R}^+$
be the unique periodic extension, with period $2\pi$, of the function
\[
  r_0 + \sum_{j=1}^{n}r_j : [0,2\pi] \to \mathbb{R}^+.
\]
Since $r_j$ is $C^\infty$ on the interior of $I_j$ for every $j$ then so is $\rho$. At the endpoints of the intervals $I_j$ one can directly check that 

\begin{align*}
  \rho(\frac{2j\pi}{n}) &= r_0,\\
  \rho'(\frac{2j\pi}{n}) &= 0,\\
  \rho''(\frac{2j\pi}{n}) &= n^2,
\text{ and } \\
\rho'''(\frac{2j\pi}{n}) &= 0\text{ for all }j \in \{0,\dots, n\}.
\end{align*}

Note that fourth derivatives do not exist at the endpoints of $I_j$, so the resulting curves are strictly $C^3$-smooth.
We define the $C^3$, non-degenerate, simple, closed curve $\Gamma$, by the parameterization:
\[
  \gamma(t) = \rho(t)(\cos t, \sin t) 
\]

\begin{figure*}
  \centering
  \subfigure[A non-degenerate cogwheel $\Gamma_1$ with $n=4$.]{
  \label{fig:cogEx}
  \includegraphics[width=6cm]{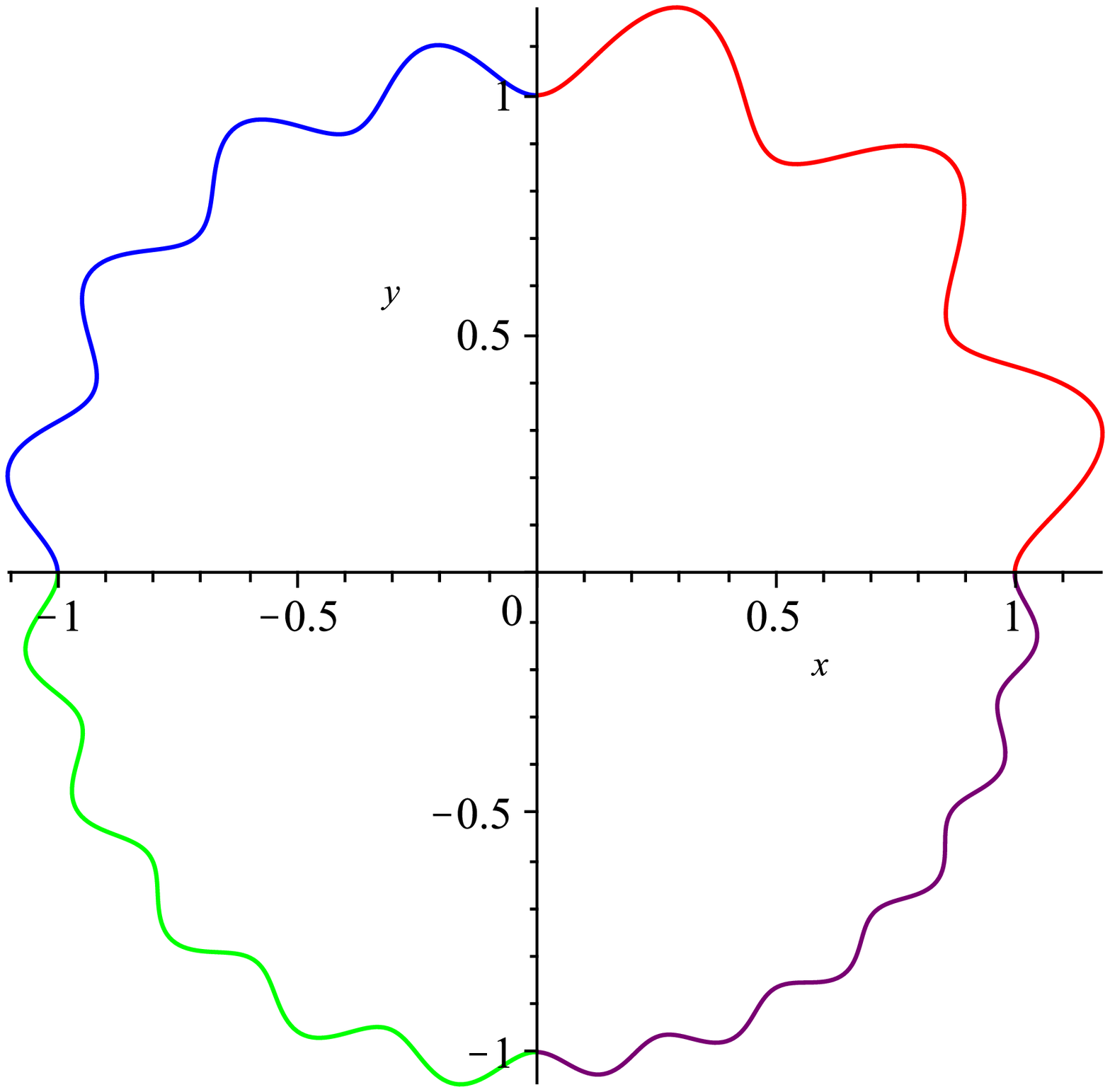}
  }
  \subfigure[The curvature $\kappa_1$ of $\Gamma_1$.]{
  \label{fig:cogEx_k}
  \includegraphics[width=6cm]{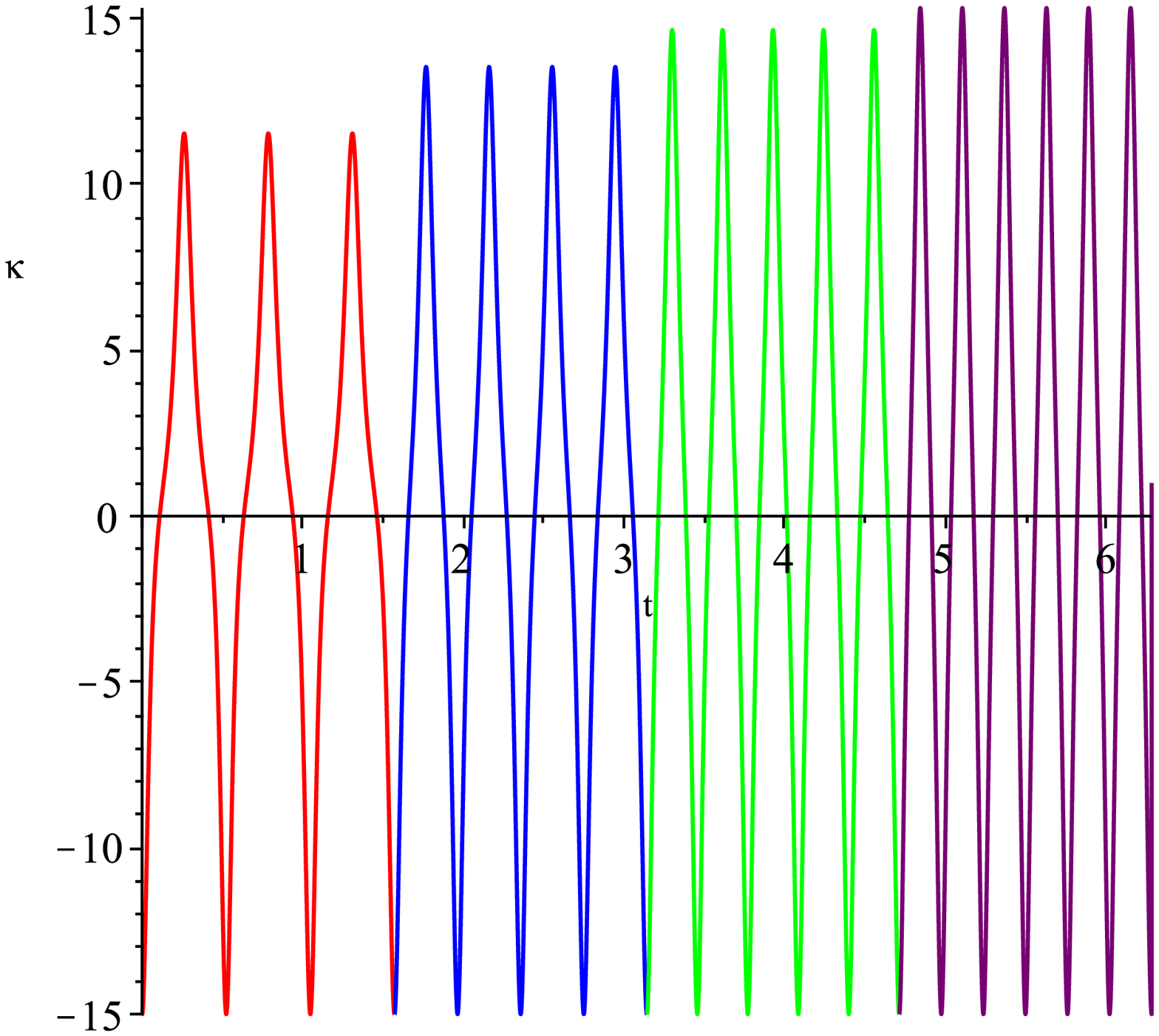}
  }
\caption{A Cogwheel and its curvature.}
\end{figure*}

\begin{figure*}
\vspace{2cm}
  \subfigure[The signature for the non-degenerate cogwheel in Figure \ref{fig:cogEx}.]{
  \label{fig:cogSig}
  \includegraphics[width=6cm]{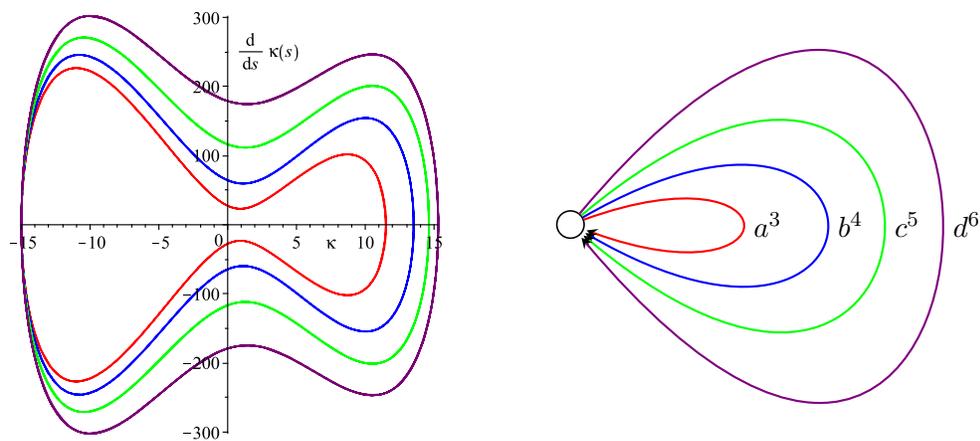}
}
  \hspace{.1cm}
  \subfigure[The corresponding quiver. The superscript on labeled edges denote the multiplicity of the path induced by Figure~\ref{fig:cogEx}.]{
\label{fig:cogGraph}
\begin{tikzpicture}[
            > = stealth, 
            shorten > = 1pt, 
            auto,
            semithick 
        ]

        \tikzstyle{every state}=[
            draw = black,
            thick,
            fill = white,
            minimum size = 4mm
        ]

	 \useasboundingbox (-1,-3) rectangle (7,2);

	\node[shape=circle,draw=black] (6) {};

	 \path[->] (6) edge [out=20, in=340, loop, distance=3cm, draw=red, thick] node {$a^3$} (6);
	 \path[->] (6) edge [out=30, in=330, loop, distance=5cm, draw=blue, thick] node {$b^4$} (6);
	 \path[->] (6) edge [out=40, in=320, loop, distance=7cm, draw=green, thick] node {$c^5$} (6);
	 \path[->] (6) edge [out=50, in=310, loop, distance=10cm, draw=violet, thick] node {$d^6$} (6);	 

    \end{tikzpicture}
  }
  \caption{The signature and quiver for the cogwheel in Figure~\ref{fig:cogEx}}
\end{figure*}

According to Definition 3 in \cite{Musso2009}, the curve $\Gamma_1$ is an example of an {\it $n$-cogwheel} with radial function $\rho$, and inner radius $r_0$.
The $j$-th section (\emph{$j$-th cog}) has $a_j$ teeth. The radial changes in the $j$-th cog is given by the function $r_j$. 
For an example see Figure \ref{fig:cogEx}, a non-degenerate cogwheel
with $r_0=1, n=4, a_1=3, a_2=4, a_3=5, a_4=6$. 

For $j=1,\dots, n$,  on the interval $I_j$, the curvature  function   is given by:
\begin{equation*}
  \kappa(t) = \frac{||\gamma'||^2a_j^2 - (a_j^2n^2r_0+n^2)\cos(na_jt)+n^2}{||\gamma'||^3a_j^2}, \hspace{.1cm} t\in I_j,
\end{equation*}
where
\begin{align*}
  ||\gamma'|| =& \frac{1}{a_j^2}
  \bigg((1-a^2_jn^2)\cos^2(na_jt) - \\
  &(2r_0a_j^2 + 2)\cos(na_jt) + r_0^2a_j^4+\\
  &(n^2 + 2r_0)a_j^2 + 1\bigg)^{1/2}.
\end{align*}

\begin{figure*}
  \centering
  \subfigure[The non-degenerate cogwheel $\Gamma_2$ induced by the word $a^3c^5b^4d^6$.]{
  \label{fig:cogwheel2}
    \includegraphics[width=6cm]{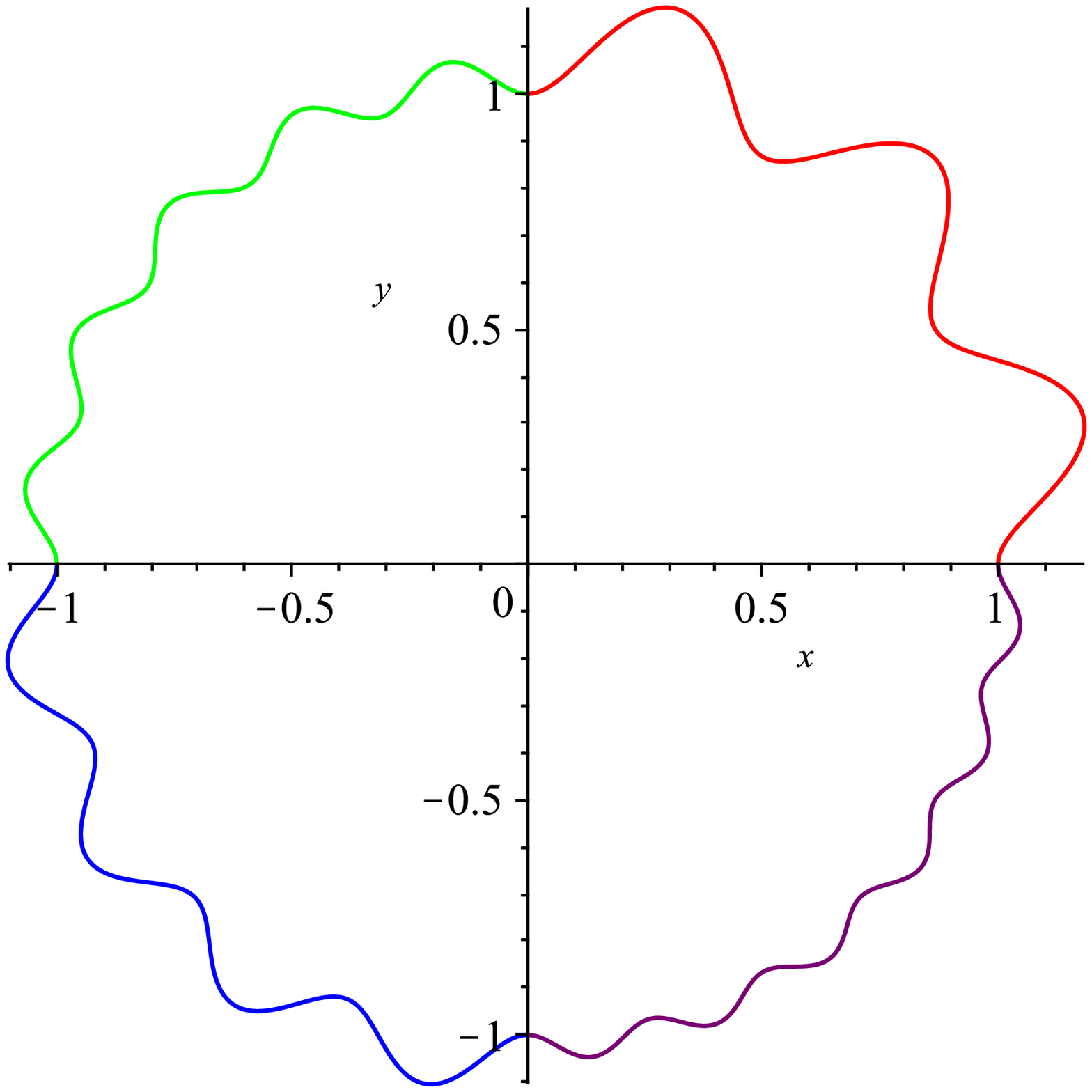}
  }
  \hspace{1cm}
  \subfigure[The non-degenerate cogwheel $\Gamma_3$ induced by the word $a^3c^5d^6b^4$.]{
  \label{fig:cogwheel3}
    \includegraphics[width=6cm]{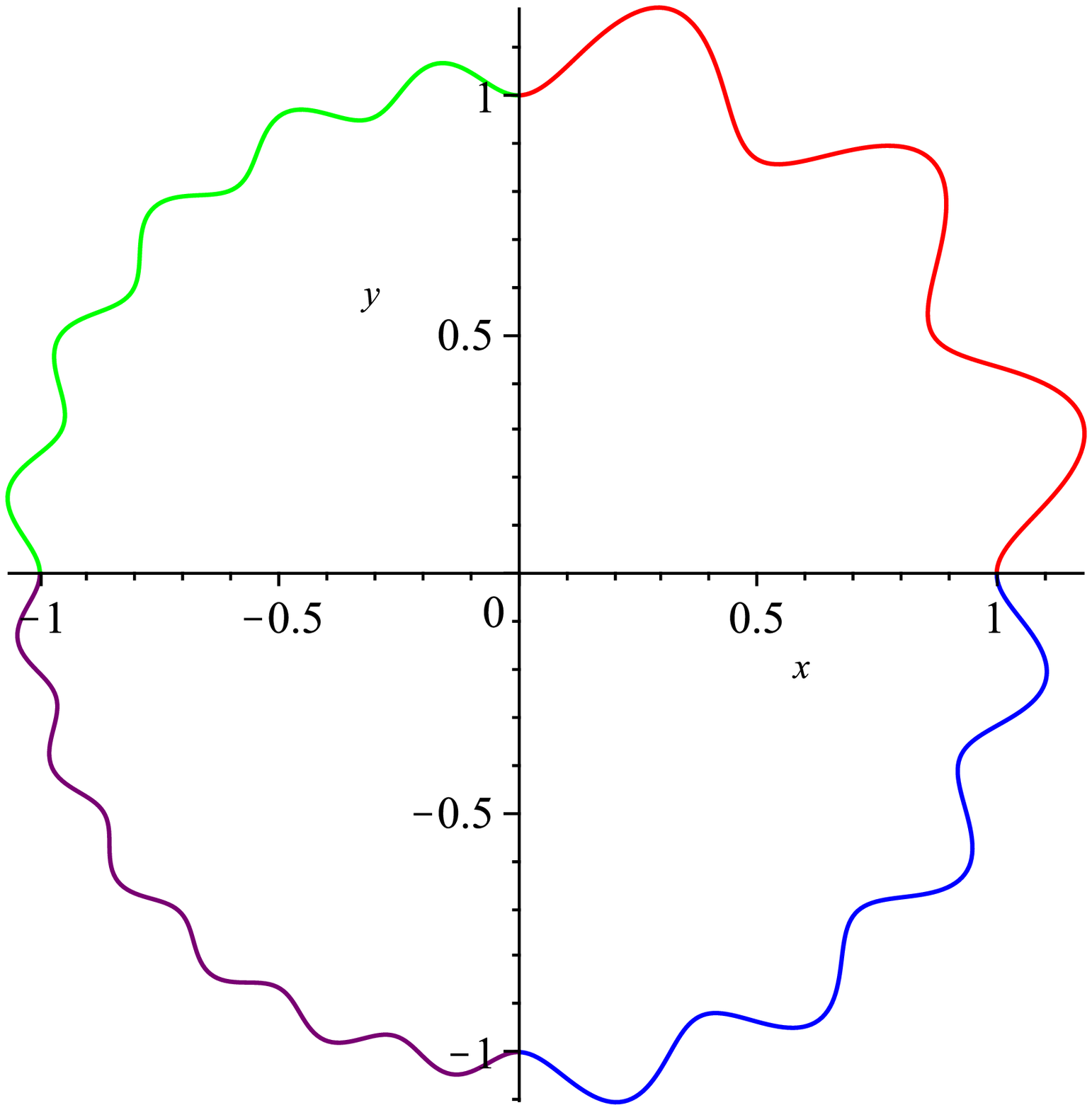}
  }
  \caption{Two cogwheels non-congruent to $\Gamma_1$ resulting from a permutation of cogs.}
\end{figure*}

\begin{figure*}
  \centering
  \subfigure[Curve induced by the word $ab^4ac^5ad^6$.]{
    \label{fig:Cog_orderSwap_1}
    \includegraphics[width=6cm]{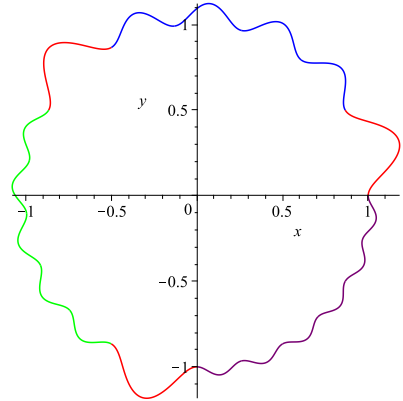}
  }
  \hspace{1cm}
  \subfigure[Curve induced by the word $bcdbdbacacdadcdbcd$.]{
    \label{fig:Cog_orderSwap_2}
    \includegraphics[width=6cm]{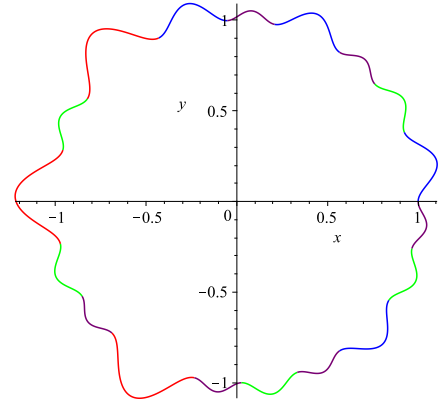}
  }
  \caption{Non-congruent curves obtained by a permutation of individual teeth of a non-degenerate cogwheel.}
  \label{fig:Cog_orderSwap}
\end{figure*}

\begin{figure*}
  \centering
  \subfigure[Cogwheel induced by word $a^2b^4c^5d^8$.]{
    \label{fig:Cog_indexSwap_1}
    \includegraphics[width=6cm]{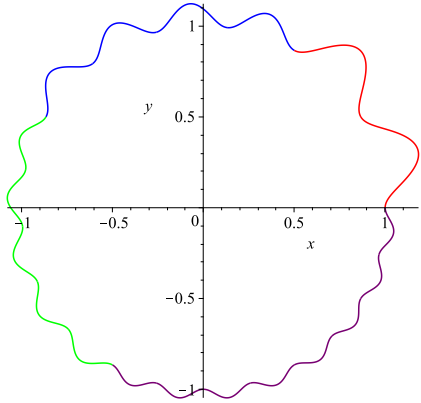}
  }
  \hspace{1cm}
  \subfigure[Cogwheel induced by word $a^3b^6c^5d^3$.]{
    \label{fig:Cog_indexSwap_2}
    \includegraphics[width=6cm]{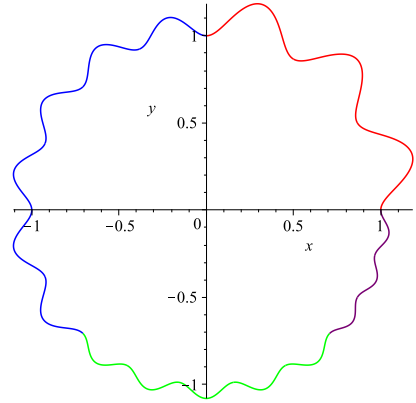}
  }
  \caption{Non-congruent curves with signature in Figure~\ref{fig:cogSig} induced by non-compatible words to that of Figure~\ref{fig:cogEx}.}
    \label{fig:Cog-exs}
\end{figure*}

The Euclidean signature and signature quiver of $\Gamma_1$ is shown in Figures~\ref{fig:cogSig} and \ref{fig:cogGraph}.
Note that this signature is significantly less complicated and is  easier to analyze than the ones corresponding to the cogwheels constructed in \cite{Musso2009}.

By Lemma 6 in \cite{Musso2009}, a permutation of cogs yields non-degenerate cogwheels with identical signature that corresponds to different compatible words from paths on $\Delta_{S_{\Gamma_1}}$ such that the same letters are grouped together.
The word that induces $\Gamma_1$ is $a^3b^4c^5d^6$ and the words inducing the cogwheels in Figures~\ref{fig:cogwheel2} and \ref{fig:cogwheel3} correspond to a permutation of cogs.


However, there are many other compatible paths allowed on $\Delta_{S_{\Gamma_1}}$ that correspond to permuting individual teeth as seen in Figure~\ref{fig:Cog_orderSwap}. These result in non-congruent curves with the same signature.
Additionally, closed curves can be generated from non-compatible paths by a consideration of the weights of each edge as seen in Figure~\ref{fig:Cog-exs}.
\\
\newpage

{\bf Acknowledgments}: We are grateful to Peter Olver for pointing out the relationship  between  the cardinalities of  local symmetry sets of a curve and multiplicities of the edges of its signature quiver and to  Ekaterina Shemyakova for suggesting  we use the term quiver for the graph associated with the signature. We  would like to thank the anonymous reviewers   for their exceptionally careful reading of the paper: their remarks and questions  were very helpful. We acknowledge NSF conference grant DMS-1952694 for providing travel funding to present an earlier version of this paper at  DART X conference.

\bibliographystyle{amsplain}

\end{document}